\documentclass[11pt,times,letter]{article}

\usepackage{amsmath, amsfonts,amsthm,amssymb,bm}
\usepackage{dsfont}

\usepackage[english]{babel}

\usepackage[normalem]{ulem}

\usepackage{latexsym,amssymb,amsfonts,graphicx}
\usepackage{amsmath,dsfont}
\usepackage{verbatim}
\usepackage{mathrsfs}
\usepackage{bm}
\usepackage{color}
\usepackage{epsfig}
\usepackage{subfigure}

\usepackage{url}
\usepackage{bm}

\usepackage{natbib}
\bibpunct[, ]{(}{)}{,}{a}{}{,}%
\usepackage{subfigure}
\usepackage[colorlinks,linkcolor=blue,citecolor=blue,anchorcolor=blue]{hyperref}

\usepackage{algorithm}
\usepackage{algpseudocode}

\newcommand{\Px}{ \mathbb{P} }
\newcommand{\Qx}{ \mathbb{Q} }

\newcommand{\Ex}{ \mathbb{E} }

\def\esssup_#1{\underset{#1}{\mathrm{ess\,sup\, }}}
\def\essinf_#1{\underset{#1}{\mathrm{ess\,inf\, }}}
\def\argmax_#1{\underset{#1}{\mathrm{arg\,max\, }}}
\def\argmin_#1{\underset{#1}{\mathrm{arg\,min\, }}}

\newcommand{\Fx}{\mathbb{F} }

\newcommand{\F}{\mathcal{F}}

\newcommand{\R}{\mathds{R}}

\newtheorem{theorem}{Theorem}[section]
\newtheorem{definition}{Definition}[section]
\numberwithin{equation}{section}

\newtheorem{proposition}[theorem]{Proposition}
\newtheorem{remark}[theorem]{Remark}
\newtheorem{lemma}[theorem]{Lemma}
\newtheorem{corollary}[theorem]{Corollary}

\allowdisplaybreaks[4]

\definecolor{Red}{rgb}{1.00, 0.00, 0.00}

\definecolor{DRed}{rgb}{0.5, 0.00, 0.00}

\definecolor{Blue}{rgb}{0.00, 0.00, 1.00}

\definecolor{Green}{rgb}{0.0, 0.4, 0.0}

\title{Continuous-time q-learning in jump-diffusion models\\ under Tsallis entropy}

\date{\vspace{-0.5in}}

\author{Lijun Bo \thanks{Email: lijunbo@ustc.edu.cn, School of Mathematics and Statistics, Xidian University, Xi'an, 710126, China.}
\and
Yijie Huang \thanks{Email: huang1@mail.ustc.edu.cn,  Department of Applied Mathematics, The Hong Kong Polytechnic University, Kowloon, Hong Kong, China.}
\and
Xiang Yu \thanks{Email: xiang.yu@polyu.edu.hk, Department of Applied Mathematics, The Hong Kong Polytechnic University, Kowloon, Hong Kong, China.}
\and
Tingting Zhang \thanks{Email: ttzhang1118@suda.edu.cn, Center for Financial Engineering, Soochow University, Suzhou, 215006, China.}
}

\begin{document}
\maketitle

\begin{abstract}
This paper studies continuous-time reinforcement learning in jump-diffusion models by featuring the q-learning (the continuous-time counterpart of Q-learning) under Tsallis entropy regularization. Contrary to the Shannon entropy, the general form of Tsallis entropy renders the optimal policy not necessarily a Gibbs measure. Herein, the Lagrange multiplier and KKT condition are needed to ensure that the learned policy is a probability density function. As a consequence, the characterization of the optimal policy using the q-function also involves a Lagrange multiplier. In response, we establish the martingale characterization of the q-function and devise two q-learning algorithms depending on whether the Lagrange multiplier can be derived explicitly or not. We also study two numerical examples, namely, an optimal liquidation problem in dark pools and a non-LQ control problem. It is interesting to see therein that the optimal policies under the Tsallis entropy regularization can be characterized explicitly, which are distributions concentrated on some compact support. The satisfactory performance of our q-learning algorithms is illustrated in each example.

\vspace{0.1in}
\noindent\textbf{Keywords}: Continuous-time q-learning, Tsallis entropy, optimal policy distribution, Lagrange multiplier, jump-diffusion processes, portfolio liquidation
\end{abstract}

\section{Introduction}

Reinforcement learning (RL) has witnessed fast-growing advancements in recent years, especially in the continuous-time framework. Q-learning algorithm (see \citealt{watkins1989learning,watkins1992q})  is widely known as a popular value-based learning method in the discrete-time framework. By learning a Q-function that maps state-action pairs to expected rewards, Q-learning enables agents to optimize policies by iteratively updating Q-values to reflect the expected returns. (\citealt{sutton2018reinforcement}). { However, learning the Q-function in a continuous-time setting is not straightforward, as the Q-function collapses to a value function independent of actions. Continuous-time models have been popularized in quantitative finance thanks to their merits that can effectively capture real-time adjustments of dynamics to the fast-changing environment and facilitate the characterization of more precise control. This is particularly advantageous in tasks requiring fine-grained decision-making, such as high-frequency trading. Given that real-world decision-making often unfolds in continuous time, there has been growing interest and debate among researchers in developing effective continuous-time RL algorithms. Early research on continuous-time RL dates back to \cite{baird1994reinforcement}, who introduced "advantage updating" for continuous-time control, and \cite{munos1997reinforcement}, who explored RL for continuous stochastic control problems. Recent pioneer studies \cite{wang2020reinforcement}, \cite{jia2022policyaa,jia2022policymm,jia2023q} have laid the theoretical foundations for continuous-time RL with continuous state and action spaces. In particular, \cite{jia2023q} propose a continuous-time q-learning approach, generalizing the conventional Q-function and Q-learning algorithm from a discrete-time setting to continuous-time counterparts by utilizing the first-order approximation of the advantage function (the difference between the Q-function and the value function) with respect to time.

Compared to discrete-time RL algorithms in the literature, continuous-time RL methods devise the policy iteration rules and the loss functions for policy evaluations in the continuous-time framework without any prior time-discretization, therefore making the algorithms stable and robust with respect to the size of time-discretization in the later implementation steps; see some discussions on the sensitivity of discrete time RL algorithms with respect to time-discretization in \cite{tallec2019making}. Moreover, the continuous-time RL framework allows the interplay of advanced mathematical tools and techniques such as stochastic differential equations and control theory in establishing the theoretical foundations of the algorithms. The continuous-time RL theories and algorithms have been generalized in various directions recently. To name a few, \cite{wang2023reinforcement} propose an actor-critic RL algorithm for optimal execution in the continuous-time Almgren-Chriss model, employing entropy regularization; \cite{WY2023} generalize the continuous-time q-learning algorithm in the learning task of mean-field control problems where the integrated q-function and the essential q-function together with test policies play crucial roles in their model-free algorithm; \cite{dai2023learning}  apply reinforcement learning to Merton's utility maximization problem in an incomplete market, focusing on learning optimal portfolio strategies without knowing model parameters; \cite{bo2023optimal} utilize the continuous-time q-learning method to address the optimal tracking portfolio problems with state reflections; \cite{han2023choquet} integrate the Choquet regularizers into continuous-time entropy-regularized RL,  exploring explicit solutions for optimal strategies in the linear-quadratic (LQ) setting; \cite{giegrich2024convergence} investigate a global linear convergence of policy gradient methods for continuous-time exploratory LQ control problems, employing geometry-aware gradient descents and proposing a novel algorithm for discrete-time policies.  

In many real-life applications, the state dynamics of interest often incur abrupt changes, and the classical diffusion models fail to capture the sudden shocks. For instance, stock prices can experience sharp jumps in response to unexpected news, and similar phenomena are observed in neuron dynamics, climate data, and other domains. To address these limitations, extending the existing continuous-time RL theory and algorithms is imperative to account for jump-diffusion processes. In quantitative finance, jump-diffusion models have been widely used to capture market behavior in response to sudden asset price changes. For example, \cite{merton1976option} incorporates jumps into the underlying asset price model to extend the classical Black-Scholes model. In particular, dark pool trading in equity markets is a prime example in which jump-diffusion models are essential. Dark pools are alternative trading venues that allow large orders to be executed without significant market impact but with the uncertainty of order execution. The liquidity in dark pools is not publicly quoted, and trades are settled based on prices determined by traditional exchanges, leading to sudden, unpredictable execution events (see \citealt{kratz2014optimal,kratz2015portfolio}).  This makes the dark pool trading a suitable model to employ the jump-diffusion processes. For theoretical studies on RL in jump-diffusion framework, \cite{gao2024reinforcement} recently generalize the continuous-time q-learning from \cite{jia2023q} to jump-diffusion models and examined some financial applications; \cite{meng2024reinforcement} investigate the RL algorithms for intensity control in jump-diffusion models with an application to choice-based network revenue management; \cite{WY2024} study the unified q-learning for mean-field control and mean-field game problems with distribution-dependent McKean-Vlasov jump-diffusion processes. However, the aforementioned results only focus on the Shannon entropy in order to derive some explicit expressions of the optimal policy in the form of Gibbs measure.

In contrast, the present paper aims to develop a continuous-time q-learning method for jump-diffusion models with state-dependent intensity under Tsallis entropy.  \cite{tsallis1988possible} proposes a generalization of Shannon entropy that provides greater flexibility and robustness to handle learning tasks with diverse policy distributions especially for the purpose of  concentrated sample actions. Particularly, Tsallis entropy is superior in scenarios with prevalent non-Gaussian, heavy-tailed behavior, on the compact support. As a direct consequence, the sampled actions are more concentrated in certain regions such that some extreme and risky decisions can be avoided during the learning procedure (see  \citealt{mertikopoulos2016learning,chow2018path}).
By adjusting its index parameter, Tsallis entropy regularization can turn the learned optimal policy into different types, offering greater flexibility in managing uncertainty and incentivizing exploration in RL. \cite{lee2018sparse,lee2019tsallis} study a class of Markov decision processes (MDP) with Tsallis entropy maximization. \cite{donnelly2024exploratory} recently investigate the optimal control in models with latent factors where the agent controls the distribution over actions by rewarding exploration with Tsallis entropy in both discrete and continuous time.  

Continuous-time q-learning under general entropy regularization is still underdeveloped. 
 We study in the present paper the general Tsallis entropy to encourage exploration. To overcome the measure-theoretical issues with continuous sampling observed in \cite{szpruch2024optimal,bender2024grid}, we follow \cite{jia2025accuracy} to consider the time-discretely sampled action process in the jump-diffusion models. We show that the controlled jump-diffusion state process converges weakly to the dynamics with coefficients aggregated according to the stochastic policy when the sampling mesh size tends to zero. Accordingly, we provide the exploratory formulation and derive the associated exploratory HJB equation. To guarantee that the learned policy is indeed a probability density function, some additional constraints are inevitable. To tackle this issue, we characterize the optimal policy by using Lagrange multipliers and Karush-Kuhn-Tucker condition. As a result, the Lagrange multiplier appears in the characterization of the optimal policy, which may not admit an explicit expression. This results in a possibly implicit characterization of the optimal policy differing significantly from the Gibbs measure under the Shannon entropy; see \cite{jia2023q}. We establish the policy improvement result and generalize the martingale characterization of the q-function and the value function in our setting that involves the Lagrange multiplier. In particular, we devise the offline q-learning algorithms depending on whether the Lagrange multiplier can be explicitly derived or not.

Our paper then applies the proposed q-learning algorithms under Tsallis entropy regularization to two financial applications. The first example employs the q-learning method to solve an LQ control problem with pure jumps that optimizes the trading strategies in dark pools as in \cite{kratz2014optimal,kratz2015portfolio}. When trading occurs concurrently in both the primary market and dark pools, the distribution of trades in these venues follows a two-dimensional random vector. Notably, the optimal policy in this LQ framework can be obtained explicitly, which is a non-Gaussian distribution with a compact support. The second example adopts the q-learning to solve a class of non-LQ jump-diffusion control problems related to selecting different repo rates (\citealt{bichuch2018arbitrage}). When dealing with two distinct repo rates, the trading proportions of these financial products are governed by a two-dimensional random vector. An interesting finding is that the optimal policy under a general power utility can be explicitly derived when the Tsallis entropy index is set to be $2$, but no explicit characterization of the optimal policy can be obtained under the conventional Shannon entropy, illustrating one technical advantage of Tsallis entropy over the Shannon entropy. However, we emphasize that the goal of this study on Tsallis entropy is not to demonstrate better learning performance over Shannon entropy in convergence or accuracy. Rather, the motivation is to understand how Tsallis entropy may help to derive some optimal policies on compact support, which can be appealing in practical applications thanks to the sampled actions in bounded domain without ad-hoc truncations.

The remainder of this paper is organized as follows. Section \ref{sec:jumpmodel} introduces the exploratory formulation of the jump-diffusion control problem under the Tsallis entropy regularization. Section \ref{sec:q-func} derives the q-function and establishes its martingale characterization, where the optimal policy relates to the q-function depending on the Lagrange multiplier. In Section \ref{sec:algorithm}, the q-learning algorithms are devised respectively when the Lagrange multiplier is known or not. As applications, Section \ref{sec:appl} examines an optimal portfolio liquidation problem and a non-LQ optimal repo rates control problem. In both cases, the optimal value functions and q-functions admit exact parameterizations. We further present numerical results demonstrating the satisfactory convergence of our q-learning algorithms in these settings. 

\ \\
 \noindent{\bf Notations.}\quad Throughout this paper, we will use the following notations frequently.  For $a,b\in\R^n$, let $a\cdot b:=a^{\top}b$. For a matrix \(A\in\mathbb{R}^{n\times d}\),  we use \(A^{\top}\) for its transpose,  \(\operatorname{tr}(A)\) for its trace,   \(|A|\) for its Euclidean (Frobenius) norm, and we write \(A^2:=AA^{\top}\).  For $A,B\in\mathbb{R}^{n\times d}$, 
let $A \circ B:= \operatorname{tr}(A B^{\top})$. If \(A\) is positive semidefinite, its square root \(\sqrt{A}\) is defined via the singular value decomposition \(A=UBV^{\top}\), where \(U,V\) are orthogonal and \(B\) is diagonal. We set \(\sqrt{A}=UB^{1/2}V^{\top}\), with \(B^{1/2}\) denoting the diagonal matrix of square roots of the entries of \(B\).   For an open set \(\mathcal{O}\subset\mathbb{R}^n\),   \(C^k(\mathcal{O})\) is the space of real functions on \(\mathcal{O}\) with continuous derivatives up to order \(k\);  \(C^{1,2}([0,T]\times\mathcal{O})\) is the space of functions \(u\) on \([0,T]\times\mathcal{O}\) for which \(\partial_t u\), \(\partial_{x_i}u\), and \(\partial_{x_i x_j}u\) (for \(1\le i,j\le n\)) exist and are continuous.   For \(u\in C^2(\mathcal{O})\), denote by \(u_x\) its gradient and by \(u_{xx}\) its Hessian.  Finally, for \(a,b\in\mathbb{R}\) let \(a\wedge b:=\min\{a,b\}\), and let \(\mathbf{1}_A(x)\) be the indicator function of a set \(A\).

\section{Problem Formulation}\label{sec:jumpmodel}

\subsection{Exploratory formulation in reinforcement learning}
For a fixed time horizon $T\in(0,\infty)$, let $(\Omega,\F,\Fx,\Px)$ be a filtered probability space with the filtration $\Fx=(\F_t)_{t\in[0,T]}$ satisfying the usual conditions. On this probability space, the process $W=(W_t)_{t\in[0,T]}$ is an $m$-dimensional Brownian motion and the process ${N}=(N(t,z,r);~(z,r)\in\R^n\times\R_+)_{t\in[0,T]}$ is a Poisson point process independent of $W$, having the intensity measure $dt\nu(dz)dr$ for a $\sigma$-finite measure $\nu$ on ${\cal B}(\R^n\setminus\{0\})$ satisfying $\int_{\R^n\setminus \{0\}}(|z|^2\wedge1)\nu(dz)<\infty$. Let us consider the controlled jump-diffusion process that $X_0^u=x\in\R^n$, and for $t\in(0,T]$,
\begin{align}\label{eq:abstractmodel}
dX_t^u &= b(t,X_t^u,u_t)dt +\sigma(t,X_t^u,u_t)dW_t\nonumber\\
  &\quad+\int_{\R^n\times\R_+} \varphi(t,X_{t-}^u,u_{t},z){\bf 1}_{\{r\leq \lambda(t,X_{t-}^u,u_t)\}} N(dt,dz,dr),
\end{align} where $u=(u_t)_{t\in[0,T]}$ is an $\Fx$-predictable process taking values on  $U\subseteq\R^d$, and the set of admissible controls is denoted by $\mathcal{U}$. Here, $b(t,x,u):[0,T]\times \R^n\times U\to\R^n$, $\sigma(t,x,u):[0,T]\times \R^n\times U\to\R^{n\times m}$, $\varphi(t,x,u,z):[0,T]\times \R^n\times U\times \R^n\to\R^n$ and $\lambda(t,x,u):[0,T]\times \R^n\times U\to\R_+$ are assumed to be measurable. Here, we consider the general jump intensity that is allowed to depend on the state and control. 

We are interested in the stochastic control problem, in which the agent aims to find an optimal control $u^*\in\mathcal{U}$ to maximize the following objective function that
\begin{align}\label{eq:objective}
    J(t,x;u):= \Ex_t^{\mathbb{P}}\left[\int_t^{T}f(s,X_s^u,u_s)ds+g(X_T^u)\right],\quad \forall u\in\mathcal{U},
\end{align}
where $f(t,x,u):[0,T]\times \R^n\times U\to\R$ stands for the  running reward function,  $g(x):\R^n\to \R$ is the terminal reward function, and the expectation operator  $\Ex_t^{\mathbb{P}}[\cdot]:=\Ex^{\mathbb{P}}\left[\cdot|X_t=x\right]$.
Given the full knowledge of the coefficients $b,\sigma,\varphi,f,g$, and the intensity measure $\nu$ in \eqref{eq:abstractmodel}-\eqref{eq:objective}, the classical methods such as the dynamic programming principle and the stochastic maximum principle can be employed to solve the above optimal control problem \eqref{eq:abstractmodel}-\eqref{eq:objective}. However, in reality, the decision maker may have limited or no information about the environment (i.e., $b,\sigma,\varphi,f,g,\nu$ are {\it unknown}). The reinforcement learning approach provides an efficient way to learn the optimal control in \eqref{eq:abstractmodel}-\eqref{eq:objective} in the unknown environment through the repeated trial-and-error procedure by taking actions and interacting with the environment.  Specifically, he tries a sequence of actions $u=(u_t)_{t\in[0,T]}$ and observe the corresponding state process $X=(X^{u}_t)_{t\in [0,T]}$ along with a stream of running rewards $(f(t,X^{u}_t,u_t))_{t\in[0,T]}$ and the terminal reward $g(X_T^u)$, and continuously update and improve his or her actions based on these observations. 

To describe the exploration step in reinforcement learning, we can randomize the action $u$ and consider its  probability density function. Let ${\cal P}(U)$ be the set of  probability density functions on $U$ and $\pi=(\pi_t)_{t\in[0,T]}$  be a given stochastic feedback policy with $\pi_t=\pi(\cdot|t,x)\in {\cal P}(U)$ for any $t\in[0,T]$.
For a Borel space $(E,{\cal B}(E))$, consider another probability space $(\Omega^\xi, \F^\xi, \Px^{\xi})$ and a measurable function $\phi:[0,T]\times \mathbb R^n\times E\to  U$ such that, for all $(t,x)\in [0,T]\times \mathbb R^n$, the $U$-valued random variable $\omega\to \phi(t,x,\xi(\omega))$ has the probability density $\pi(\cdot|t,x)$. Let $\mathbb N_0 =\mathbb N\cup\{0\}$ and $(\Omega^{\xi_{n_0}}, \F^{\xi_{n_0}}, \Px^{\xi_{n_0}},   \xi_{n_0})_{n_0\in \mathbb N_0}$ be independent copies of $(\Omega^\xi, \F^\xi, \Px^{\xi},   \xi) $. 
Introduce the following probability space with the form given by
\begin{align}
	\label{eq:space}
	(\Omega', \F',\Qx):=\left(\Omega \times  \prod_{n_0=0}^\infty \Omega^{\xi_{n_0}}, \F\otimes \bigotimes_{n_0=0}^\infty
	\F^{\xi_{n_0}},
	\Px \otimes
	\bigotimes_{n_0=0}^\infty
	\Px^{\xi_{n_0}}\right),
\end{align}
 and denotes $\Fx'=(\mathcal{F}_t \vee \F'_t)_{t\in[0,T]}$. Furthermore, consider the time grid $ G_{t:T}:= \{t=s_0<s_1<\cdots <s_{n_0} =T\} $ and sample actions from $\pi$ only at the grid points in $ G_{t:T}$.  The  state process satisfies, for all $i=0,\ldots, n_0-1$ and $s\in [s_i,s_{i+1}]$,
\begin{align}\label{X}
X_s^{u} &= X_{s_i}^{u} + \int_{s_i}^s b(\ell, X_\ell^{u},u_{s_i}) d\ell +
		\int_{s_i}^s \sigma(\ell, X_\ell^{u},u_{s_i})d W_\ell\nonumber\\
		&\qquad+\int_{s_i}^s \int_{\R^n\times\R_+} \varphi(\ell,X_{\ell-}^u,u_{s_i},z) {\bf 1}_{\{r\leq \lambda(\ell,X_{\ell-}^u,u_{s_i})\}}N(d\ell,dz,dr),\\
u_{s_i}&= \phi(s_i,X_{s_i},\xi_i).     \nonumber
\end{align}
Note that, we can also rewrite SDE \eqref{X} as follows:
\begin{align}\label{eq:X-pi}
	d X_s^{\pi} &= b(s,X_s^{\pi},u_{s}^{\pi})d s + \sigma(s,X_s^{\pi},u_s^{\pi}) d W_s +\int_{\R^n\times\R_+} \varphi(s,X_{s-}^{\pi},u_{s}^{\pi},z) {\bf 1}_{\{r\leq \lambda(s,X_{s-}^u,u_{s})\}}N(ds,dz,dr),\nonumber\\
 u_{s}^{\pi} &= u_{\delta(s)}~\text{with~$\delta(s) = s_i$ when $s\in [s_i,s_{i+1})$}.   
\end{align}

To encourage the exploration in RL, we consider the so-called Tsallis entropy regularization. Then, the objective functional is defined by
\begin{align}\label{eq:J-pi}
J^0(t, x ; \pi)= & \mathbb{E}_t^{\Qx}\left[\int_t^{T}\left(f\left(s,X_s^{\pi}, u_s^{\pi}\right)+\gamma l_p(\pi\left(u_s^{\pi})\right)\right)ds+g(X_T^{\pi})\right],
\end{align} 
where  $\gamma>0$ stands for the temperature parameter, and  the expectation operator $\Ex_t^{\mathbb{Q}}[\cdot]:=\Ex^{\mathbb{Q}}\left[\cdot|X_t=x\right]$. The Tsallis entropy with order $p\geq1$ is defined by, for $z\in\R_+$,
\begin{align}\label{eq:Tsallis}
l_p(z)=\begin{cases}
\displaystyle \frac{1}{p-1}(1-z^{p-1}), &p>1,\\[0.6em]
\displaystyle~~~~~~ -\ln z, &p=1.
\end{cases}
\end{align}
By observing \eqref{eq:Tsallis}, the Tsallis entropy with order $p\geq1$ generalizes the Shannon entropy (\citealt{tsallis1988possible}). Here, $p$ is called the entropy index, and $p=2$ corresponds to the sparse Tsallis entropy \citep{lee2018sparse}.

However, the representation \eqref{eq:X-pi}-\eqref{eq:J-pi} cannot  be applied to derive exploratory HJB equation directly from the point of view of DPP. To this purpose, it is necessary to provide the relaxed version of the control problem.
Inspired by \cite{kushner2000jump} and \cite{gao2024reinforcement} , let us introduce the process $\tilde{X}=(\tilde{X}_s)_{s\in[t,T]}$ satisfying the following SDE, $\tilde{X}_t^{\pi}=x$, and for $s\in(t,T]$,

\begin{align}\label{eq:martingale-measure}
    d\tilde{X}_s^{\pi} &=\int_{U} b(s,\tilde{X}_s^{\pi},u)\pi_s(u)duds + \sqrt{\int_U \sigma^2(s,\tilde{X}_s^{\pi},u)\pi_s(u) du}dB_s\nonumber\\
    &\quad + \int_{U} \int_{\R^n\times\R_+}\varphi(s,\tilde{X}_{s-}^{\pi},u,z){\bf 1}_{\{r\leq \lambda(s,\tilde{X}_{s-}^{\pi},u)\}} \mathcal{N}(ds,dz,dr,du),
\end{align}
where $B=(B_t)_{t\in[0,T]}$  is a scalar Brownian motion and $\mathcal{N}(ds,dz,dr,du)$ is a Poisson random measure  on $\R_+\times{\R}^n\times\R_+\times U$ with compensator $ds\nu(dz)dr\pi_s(u)du$ independent of $B$. An interesting finding is that, for the pure jump controlled state model, the representation \eqref{eq:martingale-measure} of the relaxed controlled state process can be applied to derive exploratory HJB equations directly from the point of view of DPP, which is different from the controlled diffusion case as in \cite{wang2020reinforcement}. Therefore, we can formulate our reinforcement learning problem for the jump-diffusion controlled model \eqref{eq:abstractmodel} based on the relaxed control form \eqref{eq:martingale-measure}.  Thus, our reinforcement learning problem associated with the jump-diffusion controlled state process \eqref{eq:abstractmodel}  can be stated as follows:
\begin{align}\label{eq:RL-problem}
&V(t,x):=\sup_{\pi\in\Pi_t} {J}(t,x;\pi):=\sup_{\pi\in\Pi_t}\Ex_t^{\Qx}\left[\int_t^{T}\left( \tilde{f}(s,\tilde{X}_s^{\pi},\pi)+\gamma \tilde{l}_p(s,x,\pi)\right)ds+g(\tilde{X}_T^{\pi})\right],\\
&~~{\rm s.t.}~\tilde{X}_s^{\pi}=x+\int_t^s \tilde{b}(\ell,\tilde{X}_\ell^{\pi},\pi_{\ell})d\ell + \int_t^s \tilde{\sigma}(\ell,\tilde{X}_\ell^{\pi},\pi_\ell)dB_\ell\nonumber\\
&\qquad\qquad\qquad+ \int_t^s\int_{U} \int_{\R^n\times\R_+}\varphi(\ell,\tilde{X}_{\ell-}^{\pi},u,z){\bf 1}_{\{r\leq \lambda(\ell,\tilde{X}_{\ell-}^{\pi},u)\}} \mathcal{N}(d\ell,dz,dr,du).\nonumber
\end{align}
Here, $\Pi_t$ is the set of admissible (randomized) policies on $U$ and the coefficients $\tilde{b},\tilde{\sigma},\tilde f, \tilde l_p$ are defined by, for $(s,x,\pi)\in[0,T]\times \R^n\times {\cal P}(U)$, 
\begin{align*}
&\tilde{b}(s,x;\pi):=\int_{U} b(s,x,u)\pi(u|s,x)du,\quad \tilde{\sigma}(s,x;\pi):= \sqrt{\int_U \sigma^2(s,x,u)\pi(u|s,x)du},\\
 &\tilde{f}(s,x;\pi):=\int_{U} f(s,x,u)\pi(u|s,x)du,\quad \tilde{l}_p(s,x;\pi):=\int_{U} l_p(\pi(u|s,x))\pi(u|s,x)du.
\end{align*}

Next, we provide the precise definition of admissible policies as follows:
\begin{definition}\label{def:admissible-pi}
 A policy $\pi=\pi(\cdot|
 \cdot,\cdot)$ is admissible,  i.e., $\pi\in\Pi_t$ with $t\in[0,T]$, if it holds that
\begin{itemize}
\item[{\rm(i)}]   $\pi(\cdot|s,x)\in{\cal P}( U)$ for all $(s,x)\in[t,T]\times \R^n$, and $\pi(\cdot|\cdot,\cdot):U\times [t,T]\times \R^n \to \R$ is a measurable function;
\item[{\rm(ii)}]$\pi(\cdot|s,x)$ is continuous in $(s,x)$ in the sense that $\int_U |\pi(u|s,x)-\pi(u|s',x')|du\to 0$ as $(s',x')\to (s,x)$. Moreover, there is a constant $C>0$  independent of $(s, u)$  such that 
\begin{align*}
    \int_{U}\left|\pi(u | s, x_1)-\pi\left(u|s, x_2\right)\right| d u \leq C\left|x_1-x_2\right|,\quad \forall x_1,x_2\in\R^n.
\end{align*}
\item[(iii)]For any $k \geq 1$ and $(s,x,x')\in[t,T]\times\R^{2n}$, there exist constants $C,C_k,C_p>0$ and $M_p \geq 1$ such that 
\begin{align*}
&\int_U |u|^{k}\pi(u|s,x)du<C_k(1+|x|^{k}),\quad \int_U l_p(\pi(u|s,x))\pi(u|s,x)du<C_p(1+|x|^{M_p}),\\
&\qquad\quad\int_U\int_{\R^n} |\varphi(s,x',u,z)\lambda(s,x',u)||\pi(u|s,x)-\pi(u|s,x')|\nu(dz)du\leq C|x-x'|.
\end{align*}
\end{itemize}
\end{definition}

To ensure the well-posedness of the stochastic control problem \eqref{eq:RL-problem}, we impose the following assumptions:
\begin{itemize}    
\item[{\bf(A$_{1}$)}] There exists $C>0$ such that, for all $(t,x_1,x_2,x,u)\in[0,T]\times\R^{3n}\times U$, 
\begin{align*} 
\left|b(t,x_1,u)-b(t,x_2,u)\right|+|\sigma(t,x_1,u)-\sigma(t,x_2,u)|\leq C|x_1-x_2|,&\nonumber\\
\left|b(t,x,u)\right|+\left|\sigma(t,x,u)\right|\leq C(1+|x|+|u|)&.
\end{align*}

\item[{\bf(A$_{2}$)}] For any $k \geq 1$, there exist $C,C_{k}>0$ such that, for all $(t,x_1,x_2,u)\in[0,T]\times\R^{2n}\times U$, 
\begin{align*} 
\int_{\R^n}|\varphi(t,x_1,u,z)-\varphi(t,x_2,u,z)||\lambda(t,x_1,u)|\nu(dz)&\leq C|x_1-x_2|,\\
\int_{\R^n}|\varphi(t,x_2,u,z)||\lambda(t,x_1,u)-\lambda(t,x_2,u)|\nu(dz)&\leq C|x_1-x_2|,\\
\int_{\R}|\varphi(t,x,u,z)|^k\lambda(t,x,u)\nu(dz)&\leq C_{k}(1+|x|^{k}+|u|^k).
\end{align*}

\item[{\bf(A$_{3}$)}]  The functions $f$ and $g$ are continuous satisfying the polynomial growth in the sense that, for some constants $C>0$ and $k \geq 1$,
\begin{align*}
|f(t,x,u)|+|g(x)|&\leq C(1+|x|^{k}+|u|^{k}),~~ \forall (t,x,u)\in[0,T]\times\R^n\times U.
\end{align*}
\end{itemize}

\subsection{Weak convergence of sampled dynamics to relaxed version}

This subsection shows that when actions are sampled at a discrete time grid, the resulting state process \eqref{eq:X-pi} converges weakly to the relaxed version of state process \eqref{eq:martingale-measure} as the time stepsize goes to zero. The following lemma establishes the existence and uniqueness of strong solution to the sampled dynamics  \eqref{eq:X-pi} and the relaxed version of dynamics \eqref{eq:martingale-measure}.

\begin{lemma}\label{lem:solution}
Let Assumptions {\bf(A$_{1}$)} and {\bf(A$_{2}$)} hold. Consider $(t,x)\in[0,T]\times\R^n$, $\pi\in \Pi_t$ and $(\Omega', \F',\Fx',\Qx)$ which is the probability space specified in \eqref{eq:space}. Then, we have
\begin{itemize}
\item[{\rm(i)}] for any grid $G_{t:T}$, the SDE \eqref{eq:X-pi} with $X_t^{\pi}=x$ has a unique strong solution. Moreover, for all $k\geq 2$, there exists a constant $C_k>0$ such that $\Ex^{\Qx}[|X^{\pi}_s|^k]\leq C_k(1+|x|^k)e^{C_k s}$ for any grid $G_{t:T}$ and $s\in[t,T]$.
\item[{\rm(ii)}] the SDE \eqref{eq:martingale-measure} with $\tilde{X}_t^{\pi}=x$ has a unique strong solution.
\end{itemize}
\end{lemma}

\begin{proof}
For simplicity, let $C>0$ be a generic constant independent of $(s,x,u)$, and $C_k$ be a generic constant independent of $(s,x,u)$ but depending on $k$, which may differ from line to line. We first show item (i). Fix a grid $G_{t:T}$. By using Assumption {\bf(A$_{2}$)},  we deduce that, for any $(s,x_1,x_2,u)\in[t,T]\times\R^{2n}\times U$, 
\begin{align}\label{eq:lip-jump}
&\left|\int_{\R^n}\left(\varphi(s,x_1,u,z)\lambda(s,x_1,u)-\varphi(s,x_2,u,z)\lambda(s,x_2,u)\right)\nu(dz)\right|\nonumber\\
&\leq \int_{\R^n}|\varphi(s,x_1,u,z)-\varphi(s,x_2,u,z)|\lambda(s,x_1,u)\nu(dz)\nonumber\\
&\quad+\int_{\R^n}|\varphi(s,x_2,u,z)||\lambda(s,x_1,u)-\lambda(s,x_2,u)|\nu(dz)\leq C|x_1-x_2|.
\end{align}
As $x \in \mathbb{R}$ is deterministic, in view of \eqref{eq:lip-jump} and Assumptions (\textbf{A$_1$})–(\textbf{A$_2$}) and following arguments in the proof of Theorem 1.2 in \cite{graham1992mckean} and Theorem 3.2 in \cite{kunita2004stochastic}, it holds that SDE \eqref{eq:X-pi} with $X_{t}^{\pi} = x$ admits a unique strong solution on $[s_0, s_1]$ satisfying
\begin{align*}
\sup_{s\in[s_0,s_1]}\Ex^{\Qx}\left[\left|X_s^{\pi}\right|^{k}\right]\leq C_k(1+|x|^k), \quad \forall k\geq 2.
\end{align*}

We next assume that SDE \eqref{eq:X-pi} with $X_t^{\pi}=x$ has a unique strong solution on $[s_0,s_i]$ for some $i=1,\ldots,n_0-1$ satisfying $\sup_{s\in[s_0,s_i]}\Ex^{\Qx}[|X_s^{\pi}|^{k}]\leq C_k(1+|x|^k)$. On the time interval $[s_{i},s_{i+1}]$, as $\Ex^{\Qx}[|X_{s_i}^{\pi}|^{k}]\leq C_k(1+|x|^k)<\infty$, we have, for all $s\in[s_{i},s_{i+1}]$,
\begin{align}\label{eq:growth-u}
\Ex^{\Qx}\left[\left|u_{s}^{\pi}\right|^k \right]=\Ex^{\Qx}\left[\int_{U}|u|^k\pi(u|s_i,X_{s_i}^{\pi})du \right]&\leq C_k\left\{1+\Ex^{\Qx}\left[\left|X_{s_i}^{\pi}\right|^k \right]\right\}\leq C_k(1+|x|^k).
\end{align}
By applying \eqref{eq:growth-u} and the same arguments in \cite{graham1992mckean}  and   \cite{kunita2004stochastic} again, we get the existence and uniqueness of strong solution to SDE \eqref{eq:X-pi}. Furthermore, it holds that
\begin{align*}
\sup_{s\in[s_i,s_{i+1}]}\Ex^{\Qx}\left[\left|X_s^{\pi}\right|^{k}\right]\leq C_k\left\{1+\Ex^{\Qx}\left[\left|X_{s_i}^{\pi}\right|^{k}\right]\right\}\leq C_k(1+|x|^k), \quad \forall k\geq 2.
\end{align*}
By induction, we establish the well-posedness of SDE \eqref{eq:X-pi} with the strong solution.

Next, we handle item (ii). From the proof of Lemma 2 in \cite{jia2022policymm}, for all $(s,x_1,x_2)\in[t,T]\times\R^{2n}$, it follows that
\begin{align}\label{eq:Lip-1} 
|\tilde{b}(s,x_1;\pi)-\tilde{b}(s,x_2;\pi)|+|\tilde{\sigma}(s,x_1;\pi)-\tilde{\sigma}(s,x_2;\pi)|&\leq C|x_1-x_2|,
\end{align}
and for all $(s,x)\in[t,T]\times\R^{n}$,
\begin{align}\label{eq:growth-1}
|\tilde{b}(s,x;\pi)|+|\tilde{\sigma}(s,x;\pi)|&\leq C(1+|x|).
\end{align}
For $({t},x,x')\in[0,T]\times\R^{2n}$, it follows from $\pi \in\Pi_t$, Assumption {\bf(A$_{2}$)} and \eqref{eq:lip-jump} that
\begin{align}\label{eq:Lip-2} 
&\left|\int_{U}\int_{\R}\varphi(t,x,u,z)\lambda(t,x,u)\nu(dz)\pi(u|t,x)du-\int_{U}\int_{\R}\varphi(t,x',u,z)\lambda(t,x',u)\nu(dz)\pi(u|t,x')du\right|\nonumber\\
&\qquad\leq \int_{U}\int_{\R}\left|\varphi(t,x,u,z)\lambda(t,x,u)-\varphi(t,x',u,z)\lambda(t,x',u)\right|\nu(dz)\pi(u|t,x)du\nonumber\\
&\qquad\quad+\int_{U}\int_{\R}|\varphi(t,x',u,z)\lambda(t,x',u)\left|\pi(u|t,x)-\pi(u|t,x')\right|\nu(dz)du\nonumber\\
&\qquad\leq \int_{U}C|x-x'|\pi(u|t,x)du+C|x-x'|\nonumber\\
&\qquad \leq C|x-x'|,
\end{align}
and for $(t,x)\in[0,T]\times\R^{n}$,
\begin{align}\label{eq:growth-2}
\int_{U}\int_{\R^n}|\varphi(t,x,u,z)|^k\lambda(t,x,u)\nu(dz)\pi(u|t,x)du&\leq \int_{U}C_{k}(1+|x|^{k}+|u|^k)\pi(u|t,x)du\nonumber\\
&\leq C_{k}(1+|x|^{k}).
\end{align}
In view of \eqref{eq:Lip-1}-\eqref{eq:growth-2} and  Theorem 1.2 in \cite{graham1992mckean}, the SDE \eqref{eq:martingale-measure} with $\tilde{X}_t^{\pi}=x$ admits a unique strong solution.
\end{proof}

Similar to \cite{jia2025accuracy}, we introduce the following notations for the convergence analysis. For each $\ell \in \mathbb{N}$ and $k>0$, a function $v:[0,T]\times \R^n\to \R$ belongs to $C^{\ell}_k([0,T]\times\R^{n})$ provided for all $\ell_1,\ell_2\in\mathbb{N}_0$ satisfying $2\ell_1+\ell_2 \le \ell$, the partial derivative $\partial_t^{\ell_1}\partial_x^{\ell_2} v$ exists, is continuous on $[0,T]\times\R^{n}$, and all such derivatives satisfy a polynomial growth condition in $x$:
\begin{align*}
\|v\|_{C_k^\ell} := \sum_{\substack{2\ell_1+\ell_2\le \ell}} \sup_{(t,x)\in [0,T]\times\R^{n}} \frac{\left|\partial_t^{\ell_1}\partial_x^{\ell_2} v(t,x)\right|}{1+|x|^k} < \infty.
\end{align*}
Let $C^{\ell,0}_k([0,T]\times\R^{n}\times U)$ be the space of functions $v: [0,T]\times \R^n \times U\to \R$ such that, for all $\ell_1,\ell_2\in \mathbb{N}_0$ satisfying $2\ell_1+\ell_2\le \ell$, the partial derivative $\partial_t^{\ell_1}\partial_x^{\ell_2} v(t,x,u)$ exists, is continuous in $(t,x)$ for all $(t,x,u)\in [0,T]\times\R^{n}\times U$, and all such derivatives satisfy a polynomial growth condition in $x$ and $u$:
\begin{align*}
\|v\|_{C_k^{\ell,0}} := \sum_{\substack{2\ell_1+\ell_2\le \ell}} \sup_{(t,x,u)\in [0,T]\times\R^n \times U} \frac{\left|\partial_t^{\ell_1}\partial_x^{\ell_2} v(t,x,u)\right|}{1+|x|^k + |u|^k} < \infty.
\end{align*}
To establish the weak convergence of the sampled dynamics to the relaxed version of the dynamics, we impose the following assumptions on the regularity of coefficients.
\begin{itemize}    
\item[{\bf(A$_{4}$)}] $b,\sigma^2\in C^{2,0}_k([0,T]\times \R^{n}\times U)$,  $\lambda(\cdot,\cdot,u),\varphi(\cdot,\cdot,u,z)\in C^{1,2}([0,T]\times\R^{n})$ for any $(u,z)\in U\times\R^{n}$, and  there exists some constant $C>0$ such that, for all $(t,x,u,z)\in[0,T]\times\R^{n}\times U\times\R^{n}$, 
\begin{align*} 
|\varphi_t(t,x,u,z)|+|\varphi_x(t,x,u,z)|+|\varphi_{xx}(t,x,u,z)|\leq C,\\
|\lambda(t,x,u)|+|\lambda_t(t,x,u)|+|\lambda_x(t,x,u)|+|\lambda_{xx}(t,x,u)|\leq C.
\end{align*}

\item[{\bf(A$_{5}$)}] for $\pi\in\Pi_0$, $\tilde{b},\tilde{\sigma},\varphi,\nu$ are sufficiently regular in the sense that, for any $h\in C_k^4(\R^n)$ and $t'\in[0,T]$, the following PDE, for $(t,x)\in[0,t')\times\R^n$,
\begin{align}\label{eq:PDE-v}
\begin{cases}
\displaystyle v_t(t,x)+\tilde{b}(t,x;\pi)\cdot v_x(t,x) +\frac{1}{2}\tilde{\sigma}^2(t,x;\pi) \circ v_{xx}(t,x)\\
\displaystyle\quad+\int_U\int_{\R^n}\left(v\left(t,x+\varphi(t,x,u,z)\right)-v(t,x)\right)\lambda(t,x,u)\nu(dz)\pi(u|t,x)du=0,\\[0.6em]
\displaystyle v(t',x)=h(x),~~\forall x\in\R^n,
\end{cases}
\end{align} 
has a unique solution $v^{h}(\cdot,\cdot;t')\in C_k^4([0,t']\times\R^n)$. Moreover, there  exists a constant $C>0$ depending only on $(T,k,\tilde{b},\tilde{\sigma},\varphi,\pi)$ such that $\|v^{h}(\cdot,\cdot;t')\|_{C_k^4}\leq C\|h\|_{C_k^4}$.
\item[{\bf(A$_{6}$)}] $f\in C_{k}^{2,0}$, $\tilde{f}\in C_k^2$ and $\tilde{f}(t,\cdot),g\in C_k^4(\R^n)$ for any $t\in[0,T]$.
\end{itemize}

\begin{remark}\label{rem:assumption-A4}
Note that Assumption {\bf(A$_{5}$)}  involves the solvability of PDE \eqref{eq:PDE-v}, which must be checked in each specific case. Here, we can provide some examples in which Assumption {\bf(A$_{5}$)} holds. Let us consider the case where the jump size function $\varphi(t,x,u,z)$ is independent of the control variable $u$ (thus we omit $u$ in the functions $\varphi$ and $\lambda$) and satisfies $\varphi(t,x,0)=0$ for all $(t,x)\in[0,T]\times \R^n$, and the function $\lambda(t,x)\equiv 1$. Furthermore, for $\pi\in\Pi_0$, assume that the functions  $\tilde{b},\tilde{\sigma}\in C^{4}([0,T]\times \R^n), \varphi\in C^{4}([0,T]\times\R^n\times \R^n)$ are all bounded and have bounded partial derivatives. Then, using the stochastic flow analysis (c.f. Theorem 3.4.1 and Theorem 3.4.2 in \citealt{kunita2019stochastic}), we can verify that, for  any $h\in C_k^4(\R^n)$ and $t'\in[0,T]$, the function $v^h(t,x):=\Ex[h(\tilde{X}_{t'}^{\pi})])\in C_k^4([0,t']\times\R^n)$ is the unique solution to PDE \eqref{eq:PDE-v}, where $\tilde{X}^{\pi}$ is given by \eqref{eq:martingale-measure} with $\tilde{X}_t^{\pi}=x$. Moreover, there  exists a constant $C>0$ such that $\|v^{h}\|_{C_k^4}\leq C\|h\|_{C_k^4}$.
\end{remark}

\begin{theorem}\label{thm:weak-convergence}
Let Assumptions {\bf(A$_{1}$)}, {\bf(A$_{2}$)}, {\bf(A$_{4}$)} and {\bf(A$_{5}$)} hold with $k\geq 2$. For $\pi\in \Pi_0$, consider the sampled dynamics  \eqref{eq:X-pi} and the relaxed version of dynamics \eqref{eq:martingale-measure} with initial value $x\in\R^n$. Then, for any $h\in C_k^4(\R^n)$, there exists a constant $C_k>0$ such that, for all grids $G_{0:T}$,
\begin{align}
\sup_{t\in[0,T]}\left|\Ex^{\Qx}\left[h(X_t^{\pi})\right]-\Ex^{\Qx}\left[h(\tilde{X}_t^{\pi})\right]\right|\leq C_k\|h\|_{C_k^4}\|G_{0:T}\|.
\end{align} 
Here, $\|G_{0:T}\|:=\max_{0\leq i\leq n_0-1}(s_{i+1}-s_i)$ is the mesh size of the time grid $G_{0:T}=\{0=s_0<s_1<\cdots<s_{n_0}=T\}$.
\end{theorem}

\begin{proof}
Without loss of generality, we only consider $n = 1$ for notational simplicity. By Assumption {\bf(A$_{5}$)}, for $h\in C_k^4(\R^n)$ and $t'=s_j$ with $1\leq j\leq n_0$, Eq.~\eqref{eq:PDE-v} has a unique solution $v^h(\cdot,\cdot;s_j)\in C_k^4([0,t']\times\R^n)$. For simplicity, we omit $h$, $s_j$, and $\pi$. It follows from the Feymann-Kac's formula that $\Ex^{\Qx}[h(\tilde{X}_t)]=v(0,x)$. Then, we have 
\begin{align}\label{eq:h-X-1}
\Ex^{\Qx}\left[h(X_{s_j})- h(\tilde{X}_{s_j})\right]&= \Ex^{\Qx}\left[ v(s_j,X_{s_j}) -v(0,x)\right]=\Ex^{\Qx}\left[\sum_{i=1}^j\left(v(s_i,X_{s_i})-v(s_{i-1},X_{s_{i-1}})\right)\right]\nonumber\\
&=:\sum_{i=1}^j\Ex^{\Qx}\left[I_i\right].
    \end{align}
By applying  It\^o's rule to $X$ from $s_{i-1}$ to $s_i$, we obtain 
\begin{align}\label{eq:h-X-2}
\Ex^{\Qx}\left[ I_i\right] 
=& \Ex^{\Qx}\Bigg[ \int_{s_{i-1}}^{s_i} \left( v_t(s,X_s)  + b(s,X_{s},u_{s_{i-1}}) v_x(s,X_s)+\frac{1}{2} \sigma^2(s,X_s,u_{s_{i-1}})v_{xx}(s,X_s) \right)ds\nonumber\\
&\quad+ \int_{s_{i-1}}^{s_i}\int_{\R^n}\left(v(s,X_s+\varphi(s,X_s,u_{s_{i-1}},z))-v(s,X_s)\right)\lambda(s,X_s,u_{s_{i-1}})\nu(dz) ds\Bigg].
\end{align}
For simplicity, let us introduce the function $H^v(s,x,u):[0,T]\times\R^n\times U\to \R$ given by
\begin{align*}
H^v(s,x,u):=\int_{\R^n}\left(v(s,x+\varphi(s,x,u,z))-v(s,x)\right)\lambda(s,x,u)\nu(dz).
\end{align*}
As $\xi_{i-1}$ is independent of $X_{s_{i-1}}$, and $v$ satisfies PDE \eqref{eq:PDE-v}, it holds that
\begin{align}\label{eq:h-X-3}
& \Ex^{\Qx}\Bigg[ \int_{s_{i-1}}^{s_i} \left( v_t(s_{i-1},X_{s_{i-1}})  + b(s_{i-1},X_{s_{i-1}},u_{s_{i-1}}) v_x(s_{i-1},X_{s_{i-1}})\right)ds\nonumber\\
&\quad+ \frac{1}{2}\int_{s_{i-1}}^{s_i}  \sigma^2(s_{i-1},X_{s_{i-1}},u_{s_{i-1}})v_{xx}(s_{i-1},X_{s_{i-1}}) ds+ \int_{s_{i-1}}^{s_i}H^v(s_{i-1},X_{s_{i-1}},u_{s_{i-1}})ds\Bigg]\nonumber\\
&=\Ex^{\Qx}\Bigg[ \int_{s_{i-1}}^{s_i} \left( v_t(s_{i-1},X_{s_{i-1}})  + \tilde{b}(s_{i-1},X_{s_{i-1}}) v_x(s_{i-1},X_{s_{i-1}})\right)ds\nonumber\\
&\qquad +\frac{1}{2}\int_{s_{i-1}}^{s_i} \tilde{\sigma}^2(s_{i-1},X_{s_{i-1}})v_{xx}(s_{i-1},X_{s_{i-1}}) ds+ \int_{s_{i-1}}^{s_i}\int_U H^v(s_{i-1},X_{s_{i-1}},u)\pi(u|s_{i-1},X_{s_{i-1}})  ds\Bigg]\nonumber\\
&=0.
\end{align}
It follows from \eqref{eq:h-X-2} and \eqref{eq:h-X-3} that
\begin{align}\label{eq:h-X-4}
\Ex^{\Qx}\left[ I_i\right]&= \Ex^{\Qx}\left[ \int_{s_{i-1}}^{s_i} \left( v_t(s,X_s) -v_t(s_{i-1},X_{s_{i-1}}) \right)ds\right]\nonumber\\
&\quad+ \Ex^{\Qx}\bigg[ \int_{s_{i-1}}^{s_i}b(s,X_{s},u_{s_{i-1}}) \left(  v_x(s,X_s) -v_x(s_{i-1},X_{s_{i-1}}) \right)ds\bigg]\nonumber\\
&\quad+ \Ex^{\Qx}\bigg[ \int_{s_{i-1}}^{s_i}\frac{1}{2}\sigma^2(s,X_{s},u_{s_{i-1}}) \left(  v_{xx}(s,X_s) -v_{xx}(s_{i-1},X_{s_{i-1}}) \right)ds\bigg]\nonumber\\
&\quad+ \Ex^{\Qx}\bigg[ \int_{s_{i-1}}^{s_i}v_x(s_{i-1},X_{s_{i-1}}) \left( b(s,X_{s},u_{s_{i-1}})  -b(s_{i-1},X_{s_{i-1}},u_{s_{i-1}}) ) \right)ds\bigg]\nonumber\\
&\quad+ \Ex^{\Qx}\bigg[ \int_{s_{i-1}}^{s_i}\frac{1}{2}v_{xx}(s_{i-1},X_{s_{i-1}}) \left( \sigma^2(s,X_{s},u_{s_{i-1}})  -\sigma^2(s_{i-1},X_{s_{i-1}},u_{s_{i-1}}) ) \right)ds\bigg]\nonumber\\
&\quad+  \Ex^{\Qx}\bigg[\int_{s_{i-1}}^{s_i}\left(H^v(s,X_{s},u_{s_{i-1}})-H^v(s_{i-1},X_{s_{i-1}},u_{s_{i-1}})\right)ds\bigg]\nonumber\\
&:=\Ex^{\Qx}\left[I_{i}^{(1)}\right]+\Ex^{\Qx}\left[I_{i}^{(2)}\right]+\Ex^{\Qx}\left[I_{i}^{(3)}\right]+\Ex^{\Qx}\left[I_{i}^{(4)}\right]+\Ex^{\Qx}\left[I_{i}^{(5)}\right]+\Ex^{\Qx}\left[I_{i}^{(6)}\right].
\end{align}
By using the fact $v\in C_k^4$ and Assumption {\bf(A$_{4}$)}, it is straightforward to check that  $H^v_t,H^v_x,H^v_{xx}$ have polynomial growth in $x$ and $u$. For $w=H,v_t,v_x,v_{xx},b(\cdot,\cdot,u_{s_{i-1}})$ or $\sigma^2(\cdot,\cdot,u_{s_{i-1}})$, we have from It\^o's rule that 
\begin{align}\label{eq:h-X-5}
&w(\ell,X_\ell)-w(s_{i-1},X_{s_{i-1}})\nonumber\\
&= \int_{s_{i-1}}^{\ell} \left( w_t(s,X_s)  + b(s,X_{s},u_{s_{i-1}}) w_x(s,X_s)+  \frac{1}{2} \sigma^2(s,X_s,u_{s_{i-1}})w_{xx}(s,X_s) \right)ds\nonumber\\
&\quad+ \int_{s_{i-1}}^{ \ell}\int_{\R^n}\left(w(s,X_s+\varphi(s,X_s,u_{s_{i-1}},z))-w(s,X_s)\right)\lambda(s,X_s,u_{s_{i-1}})\nu(dz) ds\nonumber\\
&\quad+ \int_{s_{i-1}}^{\ell}  \sigma(s,X_s,u_{s_{i-1}})w_{x}(s,X_s) dW_s\\
&\quad+ \int_{s_{i-1}}^{ \ell}\int_{\R^n\times\R_+}\left(w(s,X_{s-}+\varphi(s,X_{s-},u_{s_{i-1}},z))-w(s,X_{s-})\right){\bf 1}_{\{r\leq \lambda(s,X_{s-},u_{s_{i-1}})\}} \tilde{N}(ds,dz,dr)\nonumber
\end{align}
with $\tilde{N}(ds,dz,dr):=N(ds,dz,dr)-dr\nu(dz)ds$ being the compensated Poisson random measure. Using the fact $w\in C^2_k$ and Assumption {\bf(A$_{2}$)}, we deduce that
{\small\begin{align}\label{eq:h-X-6}
&\left|\int_{\R^n}\left(w(s,x+\varphi(s,x,u,z))-w(s,x)\right)\lambda(s,x,u)\nu(dz) \right|\leq M_k \left|\int_{\R^n}\left(1+|x|^k\right)|\varphi(s,x,u,z)|\lambda(s,x,u)\nu(dz) \right|\nonumber\\
&\qquad\leq M_k\left(1+|x|^k\right) \left|\int_{\R^n}|\varphi(s,x,u,z)|\lambda(s,x,u)\nu(dz) \right|\leq M_k(1+|x|^k+|u|^k),
\end{align}}where $M_k>0$ is a constant depending on $k$ that may vary from line to line. 

Similar to the proof of Theorem 4.1 in \cite{jia2025accuracy},  by using \eqref{eq:h-X-5} and \eqref{eq:h-X-6}, it is not difficult to verify that $\Ex^{\Qx}[|I_{i}^{(l)}|]$ for $l=1,\ldots,6$ can be bounded by $C_k\|v\|_{C_k^4}(s_{i}-s_{i-1})^2$. 
Thus, we have $\Ex^{\Qx}\left[ I_i\right]\leq C_k\|v\|_{C_k^4}(s_{i}-s_{i-1})^2$. Then, the desired result follows from $\sum_{i=1}^j(s_{i}-s_{i-1})^2\leq T\|G_{0:T}\|$.
\end{proof}

As a direct consequence of Theorem \ref{thm:weak-convergence}, we have the next result.
\begin{corollary}
Let Assumptions {\bf(A$_{1}$)}, {\bf(A$_{2}$)}, {\bf(A$_{4}$)}, {\bf(A$_{5}$)} and {\bf(A$_{6}$)} hold with $k\geq 2$. Given $(t,x)\in[0,T]\times\R^n$ and $\pi\in\Pi_t$, consider the corresponding objective functionals $J^0$ given by \eqref{eq:J-pi} and ${J}$ given by \eqref{eq:RL-problem}. If $l_p(\pi(\cdot|\cdot,\cdot))\in C_{k}^{2,0}$, $\tilde{l}\in C_k^2$ and $\tilde{l}_p(t,\cdot)\in C_k^4(\R^n)$ for any $t\in[0,T]$, there exists a constant $C_k>0$ only depending on $(k,b,\sigma,\varphi,\nu,f,g,\tilde{b},\tilde{\sigma},\tilde{f},\tilde{l}_p,\pi)$ such that $|J^0(t,x)-{J}(t,x)|\leq C_k\|G_{t:T}\|$ for all time grids $G_{t:T}$.
\end{corollary}

\subsection{Exploratory HJB equation and policy improvement iteration}

By applying dynamic programming arguments, the value function in \eqref{eq:RL-problem} satisfies the exploratory HJB equation given by
\begin{align}\label{eq:exHJB000}
V_t(t,x)+\sup_{\pi\in\mathcal{P}(U)}\Bigg\{& V_x(t,x)\cdot \int_{U} b(t,x,u)\pi(u)du+\frac{1}{2}\int_{U} \sigma^2(t,x,u)\pi(u)du \circ V_{xx}(t,x)\\
&\quad +\int_{U}\int_{\R^n}\left(V\left(t,x+\varphi(t,x,u,z)\right)-V(t,x)\right)\lambda(t,x,u)\nu(dz)\pi(u)du\nonumber\\
&\quad+\int_{U} (f(t,x,u)+\gamma l_p(\pi(u)))\pi(u)du\Bigg\}=0\nonumber
\end{align}
with the terminal condition $V(T,x)=g(x)$ for $x\in\R^n$. 

In order to characterize the optimal feedback policy, we introduce a scalar Lagrange multiplier $\psi(t,x):[0,T]\times \R^n\to \R$ to enforce the
constraint $\int_{U}\pi(u)du=1$, and a Karush–Kuhn–Tucker (KKT) multiplier $\xi(t,x,u):[0,T]\times \R^n\times U\to \R_+$ to enforce the constraint $\pi(u)\geq 0$. The Lagrangian is written by
\begin{align*}
{\cal L}(t,x;\pi)
&=\int_{U} \left(V_x(t,x)\cdot b(t,x,u)+\frac{1}{2}\sigma^2(t,x,u)\circ V_{xx}(t,x)\right)\pi(u)du\\
&\quad +\int_{U} \int_{\R^n}\left(V\left(t,x+\varphi(t,x,u,z)\right)-V(t,x)\right)\lambda(t,x,u)\nu(dz)\pi(u)du\\
&\quad +\int_{U} (f(t,x,u)+\gamma l_p(\pi(u)))\pi(u)du+\psi(t,x)\left(\int_{U}\pi(u)du-1\right)+\int_U \xi(t,x,u)\pi(u)du.
\end{align*}
We next discuss the candidate optimal feedback policy in terms of the entropy index $p\geq1$ by assuming that $V$ is a classical solution to the exploratory HJB equation \eqref{eq:exHJB000}:
\begin{itemize}
    \item The case $p>1$.  Using the first-order condition for the Lagrangian $\pi\to{\cal L}(t,x;\pi)$, we arrive at, the candidate optimal feedback policy is given by
\begin{align}\label{eq:pistarcadidate}
\pi_p^*(u|t,x)=\left(\frac{p-1}{p\gamma}\right)^{\frac{1}{p-1}}\left({\cal H}(t,x,u,V,V_x,V_{xx})+\psi(t,x)+\xi(t,x,u)\right)^{\frac{1}{p-1}},
\end{align}
where, the  Hamiltonian ${\cal H}(t,x,u,v)$ is defined as, for $(t,x,u)\in[0,T]\times \R^n\times U$ and $v\in C^{1,2}([0,T)\times \R^n)\cap  C([0,T]\times \R^n)$,
\begin{align}\label{eq:Hami}
{\cal H}(t,x,u,v,v_x,v_{xx})&:= b(t,x,u)\cdot v_x(t,x)+\frac{1}{2}\sigma^2(t,x,u)\circ v_{xx}(t,x)  +f(t,x,u)\nonumber\\
&\quad+\int_{\R^n}\left(v\left(t,x+\varphi(t,x,u,z)\right)-v(t,x)\right)\lambda(t,x,u)\nu(dz).
\end{align}

Then, it follows from the constraints on $\pi_p^*(u)\geq 0$ that
\begin{align}\label{eq:KKT-xi}
\mathcal{H}(t,x,u,V,V_x,V_{xx})+\psi(t,x)+\xi(t,x,u)\geq 0 .
\end{align}
As $\xi(t,x,u)$ is the KKT multiplier for the non‑negativity constraint, it satisfies $\xi(t,x,u)\geq 0$ and the complementary slackness condition $\xi(t,x,u)\,\pi^*_p(u)=0$.  If $\mathcal{H}(t,x,u,V,V_x,V_{xx})+\psi(t,x)\geq 0$, we may set $\xi(t,x,u)=0$ and \eqref{eq:KKT-xi} holds; if $\mathcal{H}(t,x,u,V,V_x,V_{xx})+\psi(t,x)<0$, then \eqref{eq:KKT-xi} forces $\xi(t,x,u)\geq -(\mathcal{H}(t,x,u,V,V_x,V_{xx})+\psi(t,x))>0$. However, complementary slackness implies that whenever $\xi(t,x,u)>0$, we must have $\pi^*_p(u)=0$. Hence, in this case, we must have  
\begin{align*}
\mathcal{H}(t,x,u,V,V_x,V_{xx})+\psi(t,x)+\xi(t,x,u)=0 .
\end{align*}
Combining the two cases yields that  
\begin{align}\label{eq:xitxu}
\xi(t,x,u)=(-{\cal H}(t,x,u,V,V_x,V_{xx})-\psi(t,x))_+,~~ \text{with}~ (x)_+:=\max\{x,0\}.
\end{align}
By plugging \eqref{eq:xitxu} into \eqref{eq:pistarcadidate}, we obtain
\begin{align}\label{eq:optimal-policy}
\pi_p^*(u|t,x)=\left(\frac{p-1}{p\gamma}\right)^{\frac{1}{p-1}}\left({\cal H}(t,x,u,V,V_x,V_{xx})+\psi(t,x)\right)_+^{\frac{1}{p-1}},
\end{align}
where, the Lagrange multiplier $\psi(t,x)$, 
which will be called \textit{normalizing function} from this point onwards, is determined by
\begin{align}\label{eq:psi}
\int_U\left(\frac{p-1}{p\gamma}\right)^{\frac{1}{p-1}}\left({\cal H}(t,x,u,V,V_x,V_{xx})+\psi(t,x)\right)_+^{\frac{1}{p-1}}du=1.
\end{align}
\item The case $p=1$. This case reduces to the conventional Shannon entropy case, in which the optimal feedback policy $\pi^*$ is a Gibbs measure given by
\begin{align}\label{eq:optimal-policy-Gibbs}
\pi_1^*(u|t,x) \propto\exp\left\{\frac{1}{\gamma}{\cal H}(t,x,u,V,V_x,V_{xx})\right\}.
\end{align}

\end{itemize}

\begin{remark}
For the general case when $p>1$, the optimal policy characterized in \eqref{eq:optimal-policy} and \eqref{eq:psi} may no longer be a Gibbs measure, and particularly may not be Gaussian even in the linear quadratic control framework. In fact, the distribution of the optimal policy now heavily relies on the expression of the normalizing function $\psi(t,x)$ in \eqref{eq:psi}. More importantly, the expression in \eqref{eq:optimal-policy} suggests that the density distribution of the optimal policy may not be supported on the whole $U$ in general, i.e., the sampled actions may concentrate on a compact set as the optimal policy is only defined on a compact subset of $U$; see the derived optimal policy distributions with compact support in Remarks \ref{example-one-distribution} and Remark~\ref{example-two-distribution} in two concrete examples.

\end{remark}

The next result shows that the candidate optimal policy given by \eqref{eq:optimal-policy} and \eqref{eq:optimal-policy-Gibbs} leads to an improvement iteration rule. Recall the objective function { $ J(t,x;\pi)$} for a fixed admissible policy $\pi$ given by \eqref{eq:RL-problem}. 
Then, if the objective function $ J(\cdot,\cdot;\pi)\in C^{1,2}([0,T)\times\R^n)\cap C([0,T]\times \R^n)$, it satisfies the following PDE:
\begin{align}\label{eq:PDE-J}
& J_t(t,x;\pi)+ J_x(t,x;\pi)\cdot \int_{U} b(t,x,u)\pi(u|t,x)du+\frac{1}{2} \int_{U} \sigma^2(t,x,u)\pi(u|t,x)du\circ J_{xx}(t,x;\pi)\nonumber\\
&\qquad\qquad\qquad+\int_{U}\int_{\R^n}\left( J\left(t,x+\varphi(t,x,u,z);\pi\right)- J(t,x;\pi)\right)\lambda(t,x,u)\nu(dz)\pi(u|t,x)du\nonumber\\
&\qquad\qquad\qquad\qquad\qquad\qquad\qquad+\int_{U} (f(t,x,u)+\gamma l_p(\pi(u|t,x)))\pi(u|t,x)du=0
\end{align}
with the terminal condition $ J(T,x:\pi)=g(x)$ for all $x\in\R^n$.

\begin{theorem}[Policy Improvement Iteration]\label{thm:PIT}
Let Assumptions {\bf(A$_{1}$)}, {\bf(A$_{2}$)} and {\bf(A$_{3}$)} hold. For any given $\pi \in \Pi_0$, assume that the objective function $\tilde  J(\cdot,\cdot;{\pi})\in C^{1,2}([0,T)\times\R^n)\cap C([0,T]\times \R^n)$ satisfies  \eqref{eq:PDE-J}, and for $p>1$, there exists a function $\psi(t,x):[0,T]\times \R^n\to \R$  satisfying
\begin{align}\label{eq:psi-PIT}
\int_U\left(\frac{p-1}{p\gamma}\right)^{\frac{1}{p-1}}\left({\cal H}(t,x,u, J, J_x, J_{xx})+\psi(t,x)\right)_+^{\frac{1}{p-1}}du=1,
\end{align}
where the Hamiltonian ${\cal H}$ is defined in \eqref{eq:Hami}.   We consider the following mapping ${\cal I}_p$  given by, for $\pi\in\Pi_{0}$,
\begin{align}\label{eq:I}
\mathcal{I}_p(\pi):= \left(\frac{p-1}{p\gamma}\right)^{\frac{1}{p-1}}\left({\cal H}(t,x,\cdot, J, J_x, J_{xx})+\psi(t,x)\right)_+^{\frac{1}{p-1}},\quad \forall p\geq 1,   
\end{align}
and ${\cal I}_1(\pi):=\lim_{p\downarrow1}{\cal I}_p(\pi)=\frac{\exp\left\{\frac{1}{\gamma}{\cal H}(t,x,\cdot, J, J_x, J_{xx})\right\}}{\int_{U}\exp\left\{\frac{1}{\gamma}{\cal H}(t,x,u, J, J_x, J_{xx})\right\}du}$.  Denote by $\pi^{\prime}={\cal I}_p(\pi)$ for $\pi\in\Pi_0$. If $\pi^{\prime}\in \Pi_0$, then $ J\left(t,x;\pi^{\prime}\right) \geq  J(t,x;\pi)$ for all $(t,x)\in[0,T]\times\R^n$. Moreover, if the mapping ${\cal I}_p$ has a fixed point $\pi^*\in\Pi_0$, then $\pi^*$ is the optimal policy that, for all $(t,x)\in[0,T]\times\R^n$, 
\begin{align*}
V(t,x)=\sup_{\pi\in\Pi_t}  J(t,x;\pi)= J(t,x;\pi^*).
\end{align*}
\end{theorem}

To prove Theorem \ref{thm:PIT}, we need the following auxiliary result.
\begin{lemma}\label{lem:optimal-policy}
Let $\gamma>0$ and $p\geq 1$. For a given function $q(u):U\to\R$, assume that there exists a  constant $\psi\in\R$ such that
\begin{align}\label{eq:psi-PIT2}
\int_U\left(\frac{p-1}{p\gamma}\right)^{\frac{1}{p-1}}\left(q(u)+\psi\right)_+^{\frac{1}{p-1}}du=1.
\end{align}
Then, $\pi^*(du)=\left(\frac{p-1}{p\gamma}\right)^{\frac{1}{p-1}}\left(q(u)+\psi\right)_+^{\frac{1}{p-1}}du$ is a probability measure on $U$, and it is the unique maximizer of the optimization problem:
\begin{align}\label{eq:prob-max}
\sup_{\pi\in{\cal P}(U)}\int_{U}\left(q(u)\pi(u)+\frac{\gamma}{p-1}\left(\pi(u)-\pi^p(u)\right)\right)du.
\end{align}
\end{lemma}

\begin{proof} 
Let $\xi=\psi-\frac{\gamma}{p-1}$, and consider the following optimization problem 
\begin{align}\label{eq:prob-max-2}
&\sup_{\pi\in{\cal P}(U)}\left[\int_{U}\left(q(u)\pi(u)+\frac{\gamma}{p-1}\left(\pi(u)-\pi^p(u)\right)\right)du+\xi\left(\int_{U}\pi(u)du-1\right)\right].
\end{align}
As $\int_{U}\pi(u)du=1$ for $\pi\in{\cal P}(U)$, the optimization problems \eqref{eq:prob-max} and \eqref{eq:prob-max-2} are  equivalent. For problem \eqref{eq:prob-max-2}, we have
\begin{align}\label{eq:prob-max-3}
&\sup_{\pi\in{\cal P}(U)}\left[\int_{U}\left(q(u)\pi(u)+\frac{\gamma}{p-1}\left(\pi(u)-\pi^p(u)\right)\right)du+\xi\left(\int_{U}\pi(u)du-1\right)\right]\nonumber\\
&\qquad=\sup_{\pi\in{\cal P}(U)}\int_{U}\left(q(u)\pi(u)+\xi\pi(u)+\frac{\gamma}{p-1}\left(\pi(u)-\pi^p(u)\right)\right)du-\xi\nonumber\\
&\qquad\leq\sup_{\pi(\cdot\geq 0}\int_{U}\left(q(u)\pi(u)+\xi\pi(u)+\frac{\gamma}{p-1}\left(\pi(u)-\pi^p(u)\right)\right)du-\xi\nonumber\\
&\qquad\leq  \int_{U}\sup_{\pi(\cdot)\geq 0}\left(q(u)\pi(u)+\xi\pi(u)+\frac{\gamma}{p-1}\left(\pi(u)-\pi^p(u)\right)\right)du-\xi.
\end{align}
We can then deduce that the unique maximizer of the optimization problem on the right-hand side of \eqref{eq:prob-max-3} is given by 
 \begin{align*}
 \pi^*(u)=\left(\frac{p-1}{p\gamma}\right)^{\frac{1}{p-1}}\left(q(u)+\psi\right)_+^{\frac{1}{p-1}},\quad  \text{for}~u\in U.
\end{align*}
 It follows from \eqref{eq:psi-PIT2} that $\pi^*\in{\cal P}(U)$, which completes the proof.
\end{proof}

By using Lemma \ref{lem:optimal-policy}, the proof of Theorem \ref{thm:PIT} is similar to that of Theorem 2 in \cite{jia2023q},  thus it is omitted here.
Note that the policy improvement iteration in Theorem \ref{thm:PIT} depends on the knowledge of model parameters. Thus, in order to devise a model free RL algorithm, we next generalize the q-leaning theory proposed in \cite{jia2023q} under Tsallis entropy regularization.

\section{Continuous-time q-Function and Martingale Characterization under Tsallis Entropy}\label{sec:q-func}

Given $\pi \in \Pi$ and $(t,x,u) \in [0,T]\times\R^n \times U$, let us consider a ``perturbed" policy of $\pi$, denoted by $\tilde{ \pi}$, as follows: for $\Delta t>0$, it takes the action $u\in U$ on $[t,t+\Delta t)$  and then follows $\pi$ on $[t+\Delta t, T]$. The resulting state process $X^{\tilde{\pi}}$ with $X_t^{\tilde{\pi}}=x$ can be split into two pieces. On $[t, t+\Delta t)$, $X^{\tilde{\pi}}=X^{u}$, which is the solution to the following equation: $X^{u}_t=x$, and for $s\in[t,t+\Delta t)$,
\begin{align*}
    dX_s^u = b(s,X_s^u,u_s)ds +\sigma(s,X_s^u,u_s)dW_s +{\int_{\R^n\times\R_+}\varphi(s,X_{s-}^u,u_s,z){\bf 1}_{\{r\leq \lambda(s,X_{s-}^u,u_s\}}N(ds,dz,dr)},
\end{align*}
while on $[t+\Delta t, T]$, $X^{\tilde{\pi}}=X^\pi$ by following \eqref{eq:X-pi} but with the initial time-state pair $(t+\Delta t, X_{t+\Delta t}^{u})$. For $\Delta t>0$, we consider the conventional Q-function with time interval $\Delta t$ that
\begin{align*}
& Q_{\Delta t}(t,x,u;\pi) \\
= & \mathbb{E}^{\mathbb{Q}}\left[\int_t^{t+\Delta t}f(s,X_s^u,u)ds+\Ex^{\mathbb{Q}}\left[\gamma\int_{t+\Delta t}^{T} l_p(\pi(u_s^{\pi}) ds+\int_{t+\Delta t}^{T} f(s,X_s^{\pi},u^{\pi})ds+g(X_T^{\pi})\Big| X_{t+\Delta t}^{u}\right]\Big| X_t^{\tilde{\pi}}=x\right] \\
= & \Ex^{\mathbb{Q}}\left[\int_t^{t+\Delta t} \left(\frac{\partial J}{\partial t}(s,X_s^u;\pi)+{\cal H}\left(s,X_s^u,u,J(\cdot,\cdot;\pi),J_x(\cdot,\cdot;\pi),J_{xx}(\cdot,\cdot;\pi)\right)\right) ds\right]+J(t,x;\pi)\\
= & J(t,x;\pi)+\left(\frac{\partial J}{\partial t}(t,x;\pi)+{\cal H}(t,x,u,J(\cdot,\cdot;\pi),J_x(\cdot,\cdot;\pi),J_{xx}(\cdot,\cdot;\pi)) \right) \Delta t+o(\Delta t),
\end{align*}
where we have used It\^o's rule in the above derivation.

We next define the q-function as the counterpart of the Q-function in the continuous time framework.

\begin{definition}[q-function]\label{def:q-function}
 The $q$-function of problem \eqref{eq:RL-problem} associated with a given policy $\pi \in \Pi_t$ is defined as, for all $(t,x,u) \in[0,T]\times\R^n\times U$,
\begin{align}\label{eq:q-function}
 q(t,x,u;\pi):=\frac{\partial J}{\partial t}(t,x;\pi)+{\cal H}(t,x,u,J(\cdot,\cdot;\pi),J_x(\cdot,\cdot;\pi),J_{xx}(\cdot,\cdot;\pi)).
\end{align}
\end{definition}
One can easily see that it is the first-order derivative of the conventional Q-function with respect to time that
\begin{align*}
 q(t,x,u;\pi)=\lim _{\Delta t \rightarrow 0} \frac{Q_{\Delta t}(t,x,u;\pi)- J(t,x;\pi)}{\Delta t}. 
\end{align*}
\begin{remark}\label{rem:q-func}
We also notice that the improved policy $\pi'$ in Theorem \ref{thm:PIT} can be represented in term of q-function  by
\begin{align*}
\pi'(u|t,x)= \left(\frac{p-1}{p\gamma}\right)^{\frac{1}{p-1}}\left(q(t,x,u;\pi)+\psi(t,x)\right)_+^{\frac{1}{p-1}},\quad \forall p\geq1,   
\end{align*}
where the Lagrange multiplier $\psi(t,x):[0,T]\times \R^n\to \R$  satisfies
\begin{align}\label{eq:psi-condition}
\int_U\left(\frac{p-1}{p\gamma}\right)^{\frac{1}{p-1}}\left(q(t,x,u;\pi)+\psi(t,x)\right)_+^{\frac{1}{p-1}}du=1.
\end{align}
A natural question is whether such a function $\psi(t,x)$ exists. We also emphasize that the function $\psi(t,x)$ can dependent on the policy $\pi$ in view of \eqref{eq:psi-condition}. In fact, given a policy $\pi\in\Pi_0$, for fixed $(t,x)\in[0,T]\times \R^n$, let us introduce the following mapping given by
\begin{align*}
a\mapsto F(a):=\int_U\left(\frac{p-1}{p\gamma}\right)^{\frac{1}{p-1}}\left(q(t,x,u;\pi)+a\right)_+^{\frac{1}{p-1}}du.
\end{align*}
Then, if the q-function satisfies some integral condition such that the mapping $a\mapsto F(a)$ is well-defined, then $a\mapsto F(a)$ is continuous and increasing with $F(a)\to -\infty$ as $a\to -\infty$ and $F(a)\to +\infty$ as $a\to +\infty$. This yields the existence and uniqueness of the function $\psi(t,x)$ satisfying \eqref{eq:psi-condition}.
\end{remark}

The following result gives the martingale characterization of the q-function under a given policy $ \pi$ when the value function is given. 

\begin{proposition}\label{prop:martingale-condition}
Suppose that Assumptions {\bf(A$_{1}$)}, {\bf(A$_{2}$)} and {\bf(A$_{3}$)} hold.  For a policy $\pi\in \Pi_0$, assume that its value function $J(\cdot,\cdot;\pi)\in C^{1,2}([0,T)\times \R^n)\cap C([0,T]\times\R^n)$. Given a continuous function $\hat{q}:[0,T]\times\R^n \times U \to \R$. Then, we have
\begin{itemize}
    \item [{\rm(i)}] 
$\hat{q}(t,x,u)=q(t,x,u;\pi)$ for all $(t,x,u) \in[0,T]\times \R^n\times U$ if and only if for all $(t,x) \in [0,T]\times \R^n$ and any time grid $ G_{t:T}$ on $[t,T]$, the following process
\begin{align}\label{eq:martinglae-J}
J\left(s,X_s^{\pi}; {\pi}\right)+\int_t^s (f(l,X_{l}^{\pi},u^{\pi}_{l})-\hat{q}(l,X_{l}^{\pi},u^{\pi}_{l}))dl,\quad s\in[t,T]
\end{align}
is an $(\Fx',\mathbb{Q})$-martingale. Here, $X^{\pi}=(X_s^{\pi})_{s\in[t,T]} $ is the solution to \eqref{eq:X-pi} under $\pi$ with $X_t^{ \pi}=x$. 
    \item [{\rm(ii)}] If $\hat{q}(t,x,u) = q(t,x,u;\pi)$ for all $(t,x,u) \in [0,T] \times \R^n \times {U}$, given any $\pi' \in \Pi_0$, for all $(t,x) \in [0,T] \times \R^n$ and any time grid ${G}_{t:T}$ on $[t,T]$, the following process
\begin{align}\label{eq:martinglae-J'}
   J\left(s, X_{s}^{ \pi'} ; \pi\right) + \int_{t}^{s}   \left[ f\left(l, X_{l}^{\pi'}, u_{l}^{ \pi'} \right) - \hat{q}\left(l, X_{l}^{ \pi'}, u_{l}^{\pi'} \right) \right] dl, \quad s\in[t,T]
\end{align}
is an  $(\Fx',\mathbb{Q})$-martingale.  Here, $X^{\pi'}=(X_s^{\pi'})_{s\in[t,T]} $ is the solution to \eqref{eq:X-pi} under $\pi'$ with $X_t^{ \pi'}=x$.

    \item [{\rm(iii)}]
 If there exists $\pi' \in \Pi_0$ such that for all $(t,x) \in [0,T] \times \R^n$ and any time grid ${G}_{t:T}$ on $[t,T]$, the process \eqref{eq:martinglae-J'} is an  $(\Fx',\mathbb{Q})$-martingale with initial condition $X_t^{ \pi'} = x$, then $\hat{q}(t,x,u) = q(t,x,u;\pi)$ for all $(t,x,u) \in [0,T] \times \R^n  \times U$.
\end{itemize}
Furthermore, in any of the three cases above, the $q$-function satisfies that
\begin{align}
   \int_{U} \left(q(t,x,u;\pi) + \gamma l_p( \pi(u|t,x)) \right) \pi(u|t,x) du = 0, \quad \forall (t,x) \in [0,T] \times \R^n.\label{eq:qgammar} 
\end{align}
\end{proposition}

\begin{proof}
(i) Applying It\^o's formula, it holds that, for $0\leq t<s\leq T$,
\begin{align*}
\qquad&J(s,X_s^{\pi};\pi)-J(t,x)+\int_t^s (f(l,X_l^{\pi},u_l^\pi)-\hat q (l,X_l^{\pi},u_l^\pi))dl\nonumber\\
=&\int_t^s \left[J_t(l,X_l^{\pi}) dl
+{\cal H}(l,X_l^{\pi},u_l^{\pi},J(l,X_l^{\pi}),J_x(l,X_l^{\pi}),J_{xx}(l,X_l^{\pi}))-\hat q (l,X_l^{\pi},u_l^\pi)\right]dl \nonumber\\
&\quad+ \int_t^s\int_{\R^n\times\R_+}\left(J(l,X_{l-}^\pi+\varphi(l,X_{l-}^\pi,u_{l}^\pi,z))-J(l,X_{l-}^\pi)\right) {\bf 1}_{\{r\leq \lambda(l,X_{l-}^{\pi},u_{l}^{\pi})\}}\tilde{N}(dl,dz,dr)\\
&\quad+\int_t^s J_x(l,X_l^{\pi})^{\top}\sigma(l,X_l^{\pi},u_l^{\pi}) dW_l\\
=&\int_t^s \left[ q (l,X_l^{\pi},u_l^\pi)-\hat q (l,X_l^{\pi},u_l^\pi)\right]dl +\int_t^s J_x(l,X_l^{\pi})^{\top}\sigma(l,X_l^{\pi},u_l^{\pi}) dW_l \notag\\
&\quad+ \int_t^s\int_{\R^n\times\R_+}\left(J(l,X_{l-}^\pi+\varphi(l,X_{l-}^\pi,u_{l}^\pi,z))-J(l,X_{l-}^\pi)\right) {\bf 1}_{\{r\leq \lambda(l,X_{l-}^{\pi},u_{l}^{\pi})\}}\tilde{N}(dl,dz,dr).
\end{align*}
 Hence, if $\hat q=q$,   \eqref{eq:martinglae-J}  is an $(\Fx,\mathbb{Q})$-martingale.
Conversely, if the right-hand side is a martingale, then, since its second term is a local martingale, it follows that \( \int_t^s \big[ q(l, X_l^\pi, u_l^\pi; \pi) - \hat{q}(l, X_l^\pi, u_l^\pi) \big] dl \) is a continuous local martingale with finite variation, and thus has zero quadratic variation.   The properties of local martingales (see, e.g., \cite{karatzas2012brownian}, Chapter 1, Exercise 5.21) indicate that a continuous local martingale with finite variation has zero quadratic variation. Thus, $\mathbb{Q}$-a.s. 
\begin{align*}
\int_t^s \left( q(l, X_l^\pi, u_l^\pi; \pi) - \hat{q}(l, X_l^\pi, u_l^\pi) \right) dl = 0 \quad \forall \, s \in [t, T].
\end{align*}
 
 Define \( h(t, x, u) = q(t, x, u; \pi) - \hat{q}(t, x, u) \), a continuous function on \( [0, T] \times \R^n \times U \). Suppose the conclusion fails: there exist \( (t^*, x^*, u^*)\in [0,T)\times \R^n \times U \) and \( \epsilon > 0 \) with \( h(t^*, x^*, u^*) > \epsilon \). By continuity of \( h \), there exists \( \delta > 0 \) such that \( h(l, x', u') > \epsilon/2 \) whenever \( |l - t^*| \vee |x' - x^*| \vee |u' - u^*| < \delta \) (where \( u \vee v = \max\{u, v\} \)).

Consider the discretely sampled state process, \( X^\pi \) starting from \( (t^*, x^*, u^*) \) with time grid $G_{t:T}$ satisfying $t^*=t_0<t^*+\delta<t_1<\cdots$. Define  the stopping time:
\[
\tau = \inf\left\{ l \geq t^* ;~ |l - t^*| > \delta \text{ or } |X_l^\pi - x^*| > \delta \right\} = \inf\left\{ l \geq t^*;~ |X_l^\pi - x^*| > \delta \right\} \wedge (t^* + \delta),
\]
with $u \wedge v = \min\{u, v\}$. By the stochastic continuity of \( X^\pi \), \( \tau > t^* \), \( \mathbb{Q} \)-a.s.

From the earlier result, there exists \( \Omega_0 \in \mathscr{F} \) with \( \mathbb{Q}(\Omega_0) = 0 \) such that for all \( \omega \in \Omega \setminus \Omega_0 \), \( \int_{t^*}^s   h(l, X_l^\pi(\omega), u_l^\pi(\omega)) dl = 0 \) for all \( s \in [t^*, T] \). By Lebesgue's differentiation theorem, for \( \omega \in \Omega \setminus \Omega_0 \):
\[
h\big(s, X_s^\pi(\omega), u_s^\pi(\omega)\big) = 0 \quad \text{a.s. } s \in [t^*, \tau(\omega)].\label{eq:hcontradiction}
\]

On the other hand, for the grid chosen above, for any \( s \in [t^{*},\tau(\omega)] \subset [t^{*},t^{*}+\delta] \), \( u_{s}^{ \pi}(\omega) = u_{t^{*}}^{ \pi}(\omega) = \phi(t^{*},x^{*},\xi_{0}(\omega)) \). Recall the definition of the admissible policy in Definition~\ref{def:admissible-pi}, we have
\[
\mathbb{P}\big(\phi(t^{*},x^{*},\xi_{0}(\omega)) \in B_{\delta}(u^{*})\big) = \int_{B_{\delta}(u^{*})} \pi(u|t^{*},x^{*}) du > 0,
\]
where \( B_{\delta}(u^{*}) = \{u' \in \mathcal{A} : |u' - u^{*}| < \delta\} \) is a neighborhood of \( u^{*} \). Hence, there exists \( \omega \in \Omega \setminus \Omega_{0} \) such that for every \( s \in [t^{*},\tau(\omega)] \),
\[
h\big(s,X_{s}^{\pi}(\omega),u_{s}^{ \pi}(\omega)\big) = h\big(s,X_{s}^{ \pi}(\omega),\phi(t^{*},x^{*},\xi_{0}(\omega))\big) > \frac{\epsilon}{2} > 0,
\]
contradicting \eqref{eq:hcontradiction}. This proves that \( q(t,x,u;\pi) = \hat{q}(t,x,u) \) for every \( (t,x,u)\in  [0,T]\times \R^n \times U\).

The proofs of (ii) and (iii) are similar to the first part of the proof of (i) and are hence omitted. Moreover, Eq.~\eqref{eq:qgammar} follows  from the definition of the q-function in Definition~\ref{eq:q-function} and 
the PDE \eqref{eq:PDE-J}.
\end{proof}
We can strengthen Proposition \ref{prop:martingale-condition} and characterize the q-function and the value function associated with a given policy ${\pi}$  simultaneously. 

\begin{theorem}\label{thm:martingale-condition}
Let Assumptions {\bf(A$_{1}$)}, {\bf(A$_{2}$)} and {\bf(A$_{3}$)} hold. For each $p\geq1$, let a policy $\pi_p\in\Pi_0$, a function $\hat{J}\in C^{1,2}([0,T)\times \R^n)\cap C([0,T]\times\R^n)$ and a continuous function $\hat{q}:[0,T]\times\R^n\times U \to \R$ be given such that, for all $(t,x)\in[0,T]\times\R^n$,
\begin{align}\label{eq:consistency}
\hat J(T,x)=g(x), \quad \int_{U}\{\hat{q}(t,x,u)+\gamma l_p(\pi_p(u|t,x))\} \pi_p (u|t,x)du=0.
\end{align}
Then, it holds that
\begin{itemize}
    \item [{\rm(i)}] $\hat{J}$ and $\hat{q}$ are respectively the value function satisfying~\eqref{eq:PDE-J} and the $q$-function associated with $\pi_p$ if and only if for all $(t,x)\in [0,T]\times \R^n$ and any time grid $ G_{t:T}$ on $[t,T]$,  the following process
\begin{align}\label{eq:hatJM}
\hat{J}\left(s,X_s^{\pi_p}; {\pi_p}\right)+\int_t^s (f(l,X_{l}^{\pi_p},u^{\pi_p}_{l})-\hat{q}(l,X_{l}^{\pi_p},u^{\pi_p}_{l}))dl,\quad s\in[t,T]
\end{align}
is an $(\Fx, \mathbb{Q})$-martingale. Here, $X^{\pi_p}=(X_s^{\pi_p})_{s\in[t,T]} $ is the solution to \eqref{eq:X-pi} under $\pi_p$ with $X_t^{\pi_p}=x$. 

\item [{\rm(ii)}] If $\hat{J}$ and $\hat{q}$ are respectively the value function and the $q$-function associated with $\pi_p$, given any $\pi' \in \Pi_0$, for all $(t,x) \in [0,T] \times \mathbb{R}^n $ and any time grid ${G}_{t:T}$ on $[t,T]$, the following process
\begin{align}\label{eq:hatJpip}
   \hat{J}\left(s, X_{s}^{ \pi'}\right) + \int_{t}^{s}   \left[ f\left(l, X_{l}^{ \pi'}, u_{l}^{ \pi'} \right) - \hat{q}\left(l, X_{l}^{ \pi'}, u_{l}^{ \pi'} \right) \right] dl,\quad s\in[t,T]  
\end{align}
is an $(\Fx, \mathbb{Q})$-martingale. Here, $X^{\pi'}=(X_s^{\pi'})_{s\in[t,T]} $ is the solution to \eqref{eq:X-pi} under $\pi'$ with $X_{t}^{ \pi'} = x$.

\item [{\rm(iii)}]   If there exists $\pi' \in \Pi_0$ such that, for all $(t,x) \in [0,T] \times \mathbb{R}^n$ and any time grid ${G}_{t:T}$ on $[t,T]$, the process \eqref{eq:hatJpip} is an $(\Fx, \mathbb{Q})$-martingale, where $(X_s^{ \pi'})_{s\in[t,T]}$ is the solution to \eqref{eq:X-pi} under $\pi'$ with $X_t^{ \pi'} = x$, then $\hat{J}$ and $\hat{q}$ are respectively the value function and the $q$-function associated with $\pi_p$.
\end{itemize}
Furthermore, in any of the three cases above, if it holds  that
\begin{align}
\pi_p(u|t,x)=\left(\frac{p-1}{p\gamma}\right)^{\frac{1}{p-1}}\left(\hat{q}(t,x,u)+\psi(t,x)\right)_+^{\frac{1}{p-1}},\quad p\geq1
\end{align}
with the normalizing function $\psi(t,x)$ satisfying $\int_U\left(\frac{p-1}{p\gamma}\right)^{\frac{1}{p-1}}\left(\hat{q}(t,x,u)+\psi(t,x)\right)_+^{\frac{1}{p-1}}du=1$ for all $(t,x)\in[0,T]\times \R^n$, then $\pi_p$ ($p\geq1$) is an optimal policy and $\hat{J}$ is the corresponding optimal value function.
\end{theorem}
\begin{proof}
  (i)
  Applying It\^o's formula, it holds that, for $0\leq t<s\leq T$,
\begin{align*}
\qquad& 
\hat J(s,X_s^{\pi_p};{\pi_p})-\hat J(t,x)+\int_t^s (f(l,X_l^{\pi_p},u_l^{\pi_p})-\hat q (l,X_l^{\pi_p},u_l^{\pi_p}))dl\nonumber\\
=&\int_t^s \Big[\hat J_t(l,X_l^{\pi_p}) dl
+{\cal H}(l,X_l^{\pi_p},u_l^{\pi_p},\hat J(l,X_l^{\pi_p}),\hat J_x(l,X_l^{\pi_p}),\hat J_{xx}(l,X_l^{\pi_p}))-\hat q (l,X_l^{\pi_p},{\pi_p})\Big]d l \notag\\
&\quad+ \int_t^s\int_{\R}\left(\hat J(l,X_{l-}^{\pi_p}+\varphi(l,X_{l-}^{\pi_p},u_{l}^{\pi_p},z))-\hat J(l,X_{l-}^{\pi_p})\right) {\bf 1}_{\{r\leq \lambda(l,X_{l-}^{\pi_p},u_{l}^{\pi_p})\}}\tilde{N}(dl,dz,dr)\notag\\
&\quad+\int_t^s \hat J_x(l,X_l^{\pi_p})^{\top}\sigma(l,X_l^{\pi_p},u_l^{\pi_p}) dW_l\notag\\
&=-\int_t^s \left[\hat q (l,X_l^{\pi_p};{\pi_p})+\gamma l_p({\pi_p})\right]dl+\int_t^s \hat J_x(l,X_l^{\pi_p})^{\top}\sigma(l,X_l^{\pi_p},u_l^{\pi_p}) dW_l \\
&\quad+\int_t^s\int_{\R}\left(\hat J(l,X_{l-}^{\pi_p}+\varphi(l,X_{l-}^{\pi_p},u_{l}^{\pi_p},z))-\hat J(l,X_{l-}^{\pi_p})\right){\bf 1}_{\{r\leq \lambda(l,X_{l-}^{\pi_p},u_{l}^{\pi_p})\}}\tilde{N}(dl,dz,dr).
\end{align*}
Hence, when  $\hat q$ satisfies~\eqref{eq:consistency},   \eqref{eq:hatJM}  is an $(\Fx,\mathbb{Q})$-martingale. 

We now prove the other side. Assume that \eqref{eq:hatJM}  is  an $(\Fx,\mathbb{Q})$-martingale. 
Using Proposition~\ref{prop:martingale-condition}-(i), we derive that 
$\hat q(t,x,u)=\hat J_t(t,x)+{\cal H}(t,x,u,\hat J,\hat J_x,\hat J_{xx})$, for $(t,x,u) \in[0,T]\times \R^n\times U$. Then, the constraint~\eqref{eq:consistency} reads 
\begin{align*}
    \hat J(T,x)=g(x), \quad \int_{U}\{\hat J_t(t,x)+{\cal H}(t,x,u,\hat J,\hat J_x,\hat J_{xx})+\gamma l_p(\pi_p(u|t,x))\} \pi_p (u|t,x)du=0.
\end{align*}
Therefore, $\hat J$ is the value function satisfying \eqref{eq:PDE-J}.
The claim (ii) follows immediately from Proposition~\ref{prop:martingale-condition}-(ii) and the proof of the claim (i); thus, we omit them here. The last conclusion follows from the same argument in \cite{jia2023q}.
\end{proof}

\begin{remark}\label{remk:martingale-entropy}
Proposition \ref{prop:martingale-condition} and Theorem \ref{thm:martingale-condition} show that the form of the martingale conditions are similar to those under Shannon entropy, which are robust to the specific choice of entropy. 
\end{remark}

Next, we present the martingale characterization of the optimal value function $J^*$ and
the optimal q-function. The proof of Theorem \ref{thm:optimal-martingale} closely follows those  of Proposition \ref{prop:martingale-condition} and Theorem \ref{thm:martingale-condition}, and is hence omitted.
\begin{theorem}\label{thm:optimal-martingale}
Let Assumptions {\bf(A$_{1}$)}, {\bf(A$_{2}$)} and {\bf(A$_{3}$)} hold. For each $p\geq1$, given a policy $\pi_p\in\Pi_0$, a function $\hat{J}\in C^{1,2}([0,T)\times \R^n)\cap C([0,T]\times\R^n)$ and a continuous function $\hat{q}:[0,T]\times\R^n\times U \to \R$ such that, for all $(t,x)\in[0,T]\times\R^n$,
\begin{align}\label{eq:consistency-optimal}
\hat J(T,x)=g(x), \\
\left(\frac{p-1}{p\gamma}\right)^{\frac{p}{p-1}} \int_{U}\left(\left(\hat{q}(t,x,u)+\psi(t,x)\right)_+-p\hat{q}(t,x,u)\right)\left(\hat{q}(t,x,u)+\psi(t,x)\right)_+^{\frac{1}{p-1}}du=1,
\end{align}
where $\psi(t,x)$ is  the normalizing function satisfying $\int_U\left(\frac{p-1}{p\gamma}\right)^{\frac{1}{p-1}}\left(\hat{q}(t,x,u)+\psi(t,x)\right)_+^{\frac{1}{p-1}}du=1$ for all $(t,x)\in[0,T]\times \R^n$.
Then, it holds that
\begin{itemize}
\item [{\rm(i)}] If $\hat{J}^*$ and $\hat{q}^*$ are respectively the optimal value function and the optimal $q$-function, given any $\pi' \in \Pi_0$, for all $(t,x) \in [0,T] \times \mathbb{R}^n $ and any time grid ${G}_{t:T}$ on $[t,T]$, the following process
\begin{align}\label{eq:hatJpip-optimal}
   \hat{J}^*\left(s, X_{s}^{ \pi'}\right) + \int_{t}^{s}   \left[ f\left(l, X_{l}^{ \pi'}, u_{l}^{ \pi'} \right) - \hat{q}^*\left(l, X_{l}^{ \pi'}, u_{l}^{ \pi'} \right) \right] dl,\quad s\in[t,T]  
\end{align}
is an $(\Fx, \mathbb{Q})$-martingale. Here, $X^{\pi'}=(X_s^{\pi'})_{s\in[t,T]} $ is the solution to \eqref{eq:X-pi} under $\pi'$ with $X_{t}^{ \pi'} = x$. Furthermore,  the policy $\pi^*_p$  given by
\begin{align}\label{eq:optimal-martingale-policy}
\pi_p^*(u|t,x)=\left(\frac{p-1}{p\gamma}\right)^{\frac{1}{p-1}}\left(\hat{q}(t,x,u)+\psi(t,x)\right)_+^{\frac{1}{p-1}},\quad p\geq1
\end{align} 
is the optimal policy.

\item [{\rm(ii)}]   If there exists some $\pi' \in \Pi_0$ such that, for all $(t,x) \in [0,T] \times \mathbb{R}^n$ and any time grid ${G}_{t:T}$ on $[t,T]$, the process \eqref{eq:hatJpip-optimal} is an $(\Fx, \mathbb{Q})$-martingale, where $(X_s^{ \pi'})_{s\in[t,T]}$ is the solution to \eqref{eq:X-pi} under $\pi'$ with $X_t^{ \pi'} = x$, then $\hat{J}^*$ and $\hat{q}^*$ are respectively the optimal value function and the optimal $q$-function.
\end{itemize}
\end{theorem}

\section{q-Learning Algorithms under Tsallis Entropy}\label{sec:algorithm}

\subsection{q-Learning algorithm when the normalizing function is available}\label{sec:normalizing -available}
 In this subsection, we design q-learning algorithms to simultaneously learn and update the parameterized value function and the policy based on the martingale condition in Theorem \ref{thm:optimal-martingale}.

We first consider the case when the normalizing function $\psi(t,x)$ is available or computable.
Given a policy $ \pi\in \Pi_0$, we parameterize the value function by a family of functions $J^\theta(\cdot)$, where $\theta \in \Theta \subset \R^{L_\theta}$ and $L_{\theta}$ is the dimension of the parameter, and parameterize the q-function by a family of functions $q^\zeta(\cdot,\cdot)$, where $\zeta \in \Psi \subset \R^{L_\zeta}$ and $L_{\zeta}$ is the dimension of the parameter. Then, 
we can get the normalizing function $\psi^{\zeta}(t,x)$ by the constraint 
\begin{align}\label{eq:normalizing}
\int_U \left(\frac{p-1}{p\gamma}\right)^{\frac{1}{p-1}}\left(q^{\zeta}(t,x,u)+\psi^{\zeta}(t,x)\right)_+^{\frac{1}{p-1}}du=1.
\end{align}
Moreover, the approximators $J^{\theta}$ and $q^{\zeta}$ should also satisfy
\begin{align}
J^{\theta}(T,x)=g(x),\quad \int_{U}[q^{\zeta}(t,x,u)+\gamma l_p(\pi^{\zeta}(u|t,x))]  \pi^{\zeta} (u|t,x)du=0,
\end{align}
where the policy $\pi^{\zeta}$ is given by, for all $(t,x,u)\in[0,T]\times \R\times U$,
\begin{align*}
\pi^{\zeta}(u|t,x)=\left(\frac{p-1}{p\gamma}\right)^{\frac{1}{p-1}}\left(q^{\zeta}(t,x,u)+\psi^{\zeta}(t,x)\right)_+^{\frac{1}{p-1}}.
\end{align*}
Then, the learning task is to find the ``optimal'' (in some sense) parameters $\theta$ and $\zeta$.  The key step in the algorithm design is to enforce the martingale condition stipulated in Theorem \ref{thm:martingale-condition}. By using martingale orthogonality condition, it is enough to explore the solution $(\theta^*,\zeta^*)$ of the following martingale orthogonality equation system:
\begin{align*}
\Ex^{\Qx}\left[\int_0^T \varrho_t \left(dJ^{\theta}\left(t,X_t^{ \pi^{\zeta}}\right)+ f(t,X_{t}^{\pi^{\zeta}},u^{\pi^{\zeta}}_{t})dt-q^{\zeta}(t,X_{t}^{\pi^{\zeta}},u^{\pi^{\zeta}}_{t})dt\right)\right]=0,
\end{align*}
and 
\begin{align*}
\Ex^{\Qx}\left[\int_0^T \varsigma_t \left(dJ^{\theta}\left(t,X_t^{ \pi^{\zeta}}\right)+ f(t,X_{t}^{\pi^{\zeta}},u^{\pi^{\zeta}}_{t})dt-q^{\zeta}(t,X_{t}^{\pi^{\zeta}},u^{\pi^{\zeta}}_{t})dt\right)\right]=0,
\end{align*}
where $X_t^{ \pi^{\zeta}}$ and $u^{\pi^{\zeta}}_t$ are defined in the sampling SDE \eqref{eq:X-pi} and
the test functions $\varrho=(\varrho_t)_{t\in[0,T]},\varsigma=(\varsigma_t)_{t\in[0,T]}$ are $\Fx$-adapted stochastic processes. This can be implemented offline by using stochastic approximation to update parameters as
\begin{align}\label{eq:J-q}
\begin{cases}
\displaystyle \theta\leftarrow \theta+\alpha_{\theta}\int_0^T \varrho_t \left(dJ^{\theta}\left(t,X_t^{ \pi^{\zeta}}\right)+ f(t,X_{t}^{\pi^{\zeta}},u^{\pi^{\zeta}}_{t})dt-q^{\zeta}(t,X_{t}^{\pi^{\zeta}},u^{\pi^{\zeta}}_{t})dt\right),\\[1em]
\displaystyle \zeta \leftarrow \zeta+\alpha_{\zeta}\int_0^T \varsigma_t \left(dJ^{\theta}\left(t,X_t^{ \pi^{\zeta}}\right)+ f(t,X_{t}^{\pi^{\zeta}},u^{\pi^{\zeta}}_{t})dt-q^{\zeta}(t,X_{t}^{\pi^{\zeta}},u^{\pi^{\zeta}}_{t})dt\right),
\end{cases}
\end{align}
where $\alpha_{\theta}$ and $\alpha_{\zeta}$ are learning rates. In this paper, we choose the test functions in the conventional sense that
\begin{align*}
\varrho_t=\frac{\partial J^\theta}{\partial \xi}\left(t,X_t^{\pi^{ \zeta}}\right),\quad \varsigma_t= \frac{\partial q^\zeta}{\partial \zeta}\left(t, X_t^{\pi^\zeta}, u_t^{\pi^\zeta}\right). 
\end{align*}

Based on the above updating rules, we present the pseudo-code of the offline q-learning algorithm in Algorithm \ref{Alg:Tsallis-q-Learning}. 
	\begin{algorithm}[h]
			\caption{\textbf{Offline q-Learning Algorithm when Normalizing Function Is available}}
		    \label{Alg:Tsallis-q-Learning}
			\hspace*{0.02in} {\bf Input:} 
			 Initial state pair $x_0$, horizon $T$, time step $\Delta t$, number of episodes $N$, number of mesh grids $K$, initial learning rates $\alpha_\theta(\cdot), \alpha_\zeta(\cdot)$  (a function of the number of episodes), functional forms of parameterized value function $J^\theta(\cdot)$, q-function $q^\zeta(\cdot)$, policy $\pi^\zeta(\cdot \mid \cdot)$ and temperature parameter $\gamma$.\\
			\hspace*{0.02in} {\bf Required Program:} an environment simulator $(x^{\prime}, f^{\prime})=$ Environment $_{\Delta t}(t,x,u)$ that takes current time-state pair $(t,x)$ and action $u$ as inputs and generates state $x^{\prime}$ and reward $f^{\prime}$ at time $t+\Delta t$ as outputs. \\
			\hspace*{0.02in} {\bf Learning Procedure:}
			\begin{algorithmic}[1]
\State Initialize $\theta$,~$\zeta$ and $i=1$. 
				\While{$i<N$}  
				\State Initialize $j = 0$.  Observe initial state $x_0$ and store $x_{t_0}\leftarrow x_0$.
				\While{$j < K$}
				  
				    \State \ \ \ \ \    Generate action $u_{t_j} \sim \pi^\zeta\left(\cdot \mid t_j,x_{t_j}\right)$.
                   \State \ \ \ \ \   Apply $u_{t_j} $ to environment simulator $(x, f)=$ Environment  $_{\Delta t}(t_j, x_{t_j},u_{t_j})$.
                   \State \ \ \ \ \   Observe  new state $x$ and $f$ as output. Store $x_{t_{j+1}} \leftarrow x$,  and $ f_{t_{j+1}} \leftarrow  f$. 
                \EndWhile{} 
                \State  For every $k=0,1,...,K-1$, compute 
\begin{align*}
 G_{k}&=J^\theta\left(t_{k+1},x_{t_{k+1}}\right)-J^\theta\left(t_{k},x_{t_k}\right)+f_{t_k}\Delta t-q^{\zeta}\left(t_k,x_{t_k},u_{t_k}\right) \Delta t.
\end{align*}

				\State Update $\theta$ and $\zeta$ by
\begin{align*}
\theta &\leftarrow \theta+\alpha_\theta(i) \sum_{k=0}^{K-1} \frac{\partial J^\theta}{\partial \theta}\left(t_k,x_{t_k}\right)G_{k}, \\
\zeta &\leftarrow \zeta+\alpha_\psi(i) \sum_{k=0}^{K-1}   \frac{\partial q^\zeta}{\partial \zeta}\left(t_k,x_{t_k},u_{t_k}\right)G_{k}.
\end{align*}
               \State   Update $i \leftarrow i+1$.
				          
              \EndWhile{}                
    \end{algorithmic}
\end{algorithm}

\subsection{q-Learning algorithm when the normalizing function is unavailable}
In this subsection, we handle the case when the normalizing function $\psi(t,x)$ does not admit an explicit form. In this case, by knowing the learned q-function, we cannot learn the optimal policy directly as there is an unknown term $\psi(t,x)$. We can still parameterize the value function by a family of functions $J^\theta(\cdot)$, where $\theta\in \Theta \subset \R^{L_\theta}$ and $L_{\xi}$ is the dimension of the parameter, and parameterize the q-function by a family of functions $q^\zeta(\cdot,\cdot)$, where $\zeta \in \Psi \subset \R^{L_\zeta}$ and $L_{\zeta}$ is the dimension of the parameter. However, we can not get the normalizing function $\psi(t,x)$ from \eqref{eq:normalizing}. In response, we parameterize the policy  by a family policy function $\pi^{\chi}(\cdot)$, where $\chi \in \Upsilon \subset \R^{L_\chi}$ and $L_{\chi}$ is the dimension of the parameter. Moreover, the approximators $J^{\theta}$ and $\pi^{\chi}$ should also satisfy $J^{\theta}(T,x)=g(x)$. Define the function $F:[0,T]\times \R^n\times {\cal P}(U)\times {\cal P}(U)\to \R$ by
\begin{align}
F(t,x;\pi',\pi):=\int_{U}\left(q(t,x,u;\pi)+\gamma l_p(\pi'(u|t,x))\right)  \pi' (u|t,x)du.
\end{align}
Then, we can devise an Actor-Critic q-learning algorithm to learn the q-function and the optimal policy alternatively. For the Actor-step (or policy improvement step), we update the policy $\pi^{\chi}$ by maximizing the function $F(t,x;\pi^{\chi'},\pi^{\chi})$ that
\begin{align*}
\max_{\chi'\in\Upsilon} F(t,x;\pi^{\chi'},\pi^{\chi})=\max_{\chi'\in \Upsilon}\int_{U}\left(q(t,x,u;\pi^{\chi})+\gamma l_p(\pi^{\chi'}(u|t,x))\right) \pi^{\chi'} (u|t,x)du
\end{align*}
In fact, we have the next result, which is a direct consequence of Theorem \ref{thm:PIT}.
\begin{lemma}
Suppose {\bf(A$_{1}$)}, {\bf(A$_{2}$)} and {\bf(A$_{3}$)} hold. Given $(t,x)\in[0,T]\times \R^n$ and $\pi,\pi'\in\Pi_t$, if it holds that $F(t,x;\pi',\pi)\geq F(t,x;\pi,\pi)$, then $J(t,x;\pi')\geq J(t,x;\pi)$.
\end{lemma}
Moreover, in order to employ the q-learning method based on Theorem \ref{thm:martingale-condition},  the policy function $\pi^{\chi}$ should satisfy $\pi^{\chi}\in {\cal P}(U)$ and the consistency condition \eqref{eq:consistency}. Here, we relax these constraints and consider the following maximization problem, for $w_1,w_2\geq 0$
\begin{align*}
\max_{\chi'\in\Upsilon} \left[F(t,x;\pi^{\chi'},\pi^{\chi})-w_1 F^2(t,x;\pi^{\chi'},\pi^{\chi'})-w_2\left(\int_U \pi^{\chi'}(u)du-1\right)^2\right].
\end{align*} 
By a direct calculation, we obtain
\begin{align*}
\frac{\partial  F(t,x;\pi^{\chi'},\pi^{\chi})}{\partial \chi'}&=\int_{U}\left(q(t,x,u;\pi^{\chi})+\gamma l_p(\pi^{\chi'}(u|t,x))\right) \frac{\partial  \pi^{\chi'} (u|t,x)}{\partial \chi'}du\nonumber\\
&\quad+\int_{U}\gamma l_p'(\pi^{\chi'}(u|t,x)) \frac{\partial  \pi^{\chi'} (u|t,x)}{\partial \chi'}\pi^{\chi'} (u|t,x)du\\
&=\int_{U}\left(q(t,x,u;\pi^{\chi})+\gamma l_p(\pi^{\chi'}(u|t,x))\right) \frac{\partial  \ln \pi^{\chi'} (u|t,x)}{\partial \chi'}\pi^{\chi'} (u|t,x)du\\
&\quad +\gamma \int_{U}l_p'(\pi^{\chi'}(u|t,x)) \frac{\partial  \pi^{\chi'} (u|t,x)}{\partial \chi}\pi^{\chi'} (u|t,x)du.
\end{align*}
Hence, we can update $\chi$ by using the stochastic gradient descent that

\begin{align*}
\chi\leftarrow\chi&+\alpha_{\chi}\Bigg(\int_0^T\Bigg(\left(q(t,X_t,u^{\pi^{\chi}};\pi^{\chi})+\gamma l(\pi^{\chi}(u^{\pi^{\chi}}|t,X_t))\right) \frac{\partial  \ln \pi^{\chi} (u^{\pi^{\chi}}|t,X_t)}{\partial \chi}\nonumber\\
&+\gamma l'(\pi^{\chi}(u^{\pi^{\chi}}|t,X_t)) \frac{\partial  \pi^{\chi} (u^{\pi^{\chi}}|t,X_t)}{\partial \chi}\Bigg)dt-2w_1 \int_0^T F(t,X_t;\pi^{\chi},\pi^{\chi})\frac{\partial  F(t,X_t;\pi^{\chi},\pi^{\chi})}{\partial \chi}dt\nonumber\\
&-2w_2\int_0^T \left(\int_U \pi^{\chi}(u|t,X_t)du-1\right)\int_U \frac{\partial \pi^{\chi}}{\partial \chi}(u|t,X_t)dudt\Bigg).
\end{align*}

Next, for the Critic-step (or the policy evaluation step), we can follow the same updating rules of parameters for the value function and q-function according to \eqref{eq:J-q} in the previous algorithm in subsection \ref{sec:normalizing -available}. We present the pseudo-code of the Actor-Critic q-learning algorithm when normalizing function is unavailable in Algorithm \ref{Alg:Tsallis-q-Learning-normalizing-unavailable}.

\begin{algorithm}[htbp]			\caption{\textbf{Offline q-Learning Algorithm When Normalizing Function Is Unavailable }}
		    \label{Alg:Tsallis-q-Learning-normalizing-unavailable}
			\hspace*{0.02in} {\bf Input:} 
			 Initial state pair $x_0$, horizon $T$, time step $\Delta t$, number of episodes $N$, number of mesh grids $K$, initial learning rates $\alpha_\theta(\cdot), \alpha_\zeta(\cdot)$  (a function of the number of episodes), functional forms of parameterized value function $J^\theta(\cdot)$, q-function $q^\zeta(\cdot)$, policy $\pi^\chi(\cdot \mid \cdot)$ and temperature parameter $\gamma$.\\
			\hspace*{0.02in} {\bf Required Program:} an environment simulator $(x^{\prime}, f^{\prime})=$ Environment $_{\Delta t}(t,x,u)$ that takes current time-state pair $(t,x)$ and action $u$ as inputs and generates state $x^{\prime}$ and reward $f^{\prime}$ at time $t+\Delta t$ as outputs. \\
			\hspace*{0.02in} {\bf Learning Procedure:}
			\begin{algorithmic}[1]
\State Initialize $\theta$,~$\zeta$ and $i=1$. 
				\While{$i<N$}  
				\State Initialize $j = 0$.  Observe initial state $x_0$ and store $x_{t_0}\leftarrow x_0$.
				\While{$j < K$}
				  
				    \State \ \ \ \ \    Generate action $u_{t_j} \sim \pi^\chi\left(\cdot \mid t_j,x_{t_j}\right)$.
                   \State \ \ \ \ \   Apply $u_{t_j} $ to environment simulator $(x, f)=$ Environment  $_{\Delta t}(t_j, x_{t_j},u_{t_j})$.
                   \State \ \ \ \ \   Observe  new state $x$ and $f$ as output. Store $x_{t_{j+1}} \leftarrow x$,  and $ f_{t_{j+1}} \leftarrow  f$. 
                \EndWhile{} 
                \State  For every $k=0,1,...,K-1$, compute 
\begin{align*}
 G_{k}&=J^\theta\left(t_{k+1},x_{t_{k+1}}\right)-J^\theta\left(t_{k},x_{t_k}\right)+f_{t_k}\Delta t-q^{\zeta}\left(t_k,x_{t_k},u_{t_k}\right) \Delta t.
\end{align*}

				\State For the Critic (policy evaluation) step, update $\theta$ and $\zeta$ (using the updated $\chi$) by
\begin{align*}
\theta &\leftarrow \theta+\alpha_\theta(i) \sum_{k=0}^{K-1} \frac{\partial J^\theta}{\partial \theta}\left(t_k,x_{t_k}\right)G_{k}, \\
\zeta &\leftarrow \zeta+\alpha_\psi(i) \sum_{k=0}^{K-1}   \frac{\partial q^\zeta}{\partial \zeta}\left(t_k,x_{t_k},u_{t_k}\right)G_{k}.\\
\end{align*}

\State For the Actor (policy improvement) step, update $\chi$ (using the updated $\theta$ and $\zeta$) by
\begin{align*}
\chi&\leftarrow\chi+\alpha_{\chi}(i)\Bigg(\sum_{k=0}^{K-1} \left(\left(q^{\zeta}(t_k,x_{t_k},u_{t_k})+\gamma l(\pi^{\chi}(u_{t_k})\right) \frac{\partial  \ln \pi^{\chi} (u_{t_k})}{\partial \chi}+\gamma l'(\pi^{\chi}(u_{t_k})) \frac{\partial  \pi^{\chi} (u_{t_k})}{\partial \chi}\right)\\
&\qquad-2w_1(i) \sum_{k=0}^{K-1}F(t_k,x_{t_k};\pi^{\chi},\pi^{\chi})\frac{\partial  F(t_k,x_{t_k};\pi^{\chi},\pi^{\chi})}{\partial \chi}\\
&\qquad-2w_2(i)\sum_{k=0}^{K-1} \left(\int_U \pi^{\chi}(u|t_k,x_{t_k})du-1\right)\int_U \frac{\partial \pi^{\chi}}{\partial \chi}(u|t_k,x_{t_k})du\Bigg).
\end{align*}
               \State   Update $i \leftarrow i+1$.
				          
              \EndWhile{}                
    \end{algorithmic}
\end{algorithm}

\section{Applications and Numerical Examples}\label{sec:appl}

\subsection{The optimal portfolio liquidation problem}

Consider an optimal portfolio liquidation problem in which a large investor has access both to a classical exchange and to a dark pool with adverse selection. As in \cite{kratz2014optimal,kratz2015portfolio}, the trading and price formation are described as the classical exchange as a linear price impact model. The trade execution can be enforced by selling aggressively;  however, this results in quadratic execution costs due to a stronger price impact. The order execution in the dark pool is modeled by a Poisson process ${N}=(N_t)_{t\geq0}$ with intensity parameter $\lambda>0$, where orders submitted to the dark pool are executed at the jump times of Poisson processes. The split of orders between the dark pool and the exchange is thus driven by the trade-off between execution uncertainty and price impact costs. Next, we formulate the optimal portfolio liquidation problem in detail. Consider an investor who has to liquidate an asset position $x \in \R$ within a finite trading horizon $[0, T]$. The investor can control her trading intensity $\xi=(\xi_t)_{t\in[0,T]}$, and can place orders $\eta=(\eta_t)_{ t\in[0, T]}$ in the dark pool. Given a trading strategy $u=(\xi_t,\eta_t)_{t\in[0,T]}$ (as r.c.l.l. $\Fx$-predictable processes) taking values in the policy space $U=\R^2$, the asset holdings of the investor at time $t\in[0,T)$ is given by  
\begin{align}\label{eq:Xut}
    X_t^u:=x-\int_0^t \xi_s d s-\int_0^t  \eta_s d {N}_s.
\end{align}
Then, a liquidation strategy $u\in{\cal U}$ yields the following trading costs at  $(t,x)\in [0,T]\times\R$,
\begin{align}\label{eq:darkpoolJDP}
{J}_{\rm DP}(t,x;u):=\mathbb{E}^{\Px}\left[\int_t^T\left(\kappa |\xi_s|^2+c |X_s^u|^2\right) ds +g_T(X_T)\Big | X_t=x\right],
\end{align}
where,
 according to the liquidation constraint in the Definition 2.3 of  \cite{{kratz2015portfolio}}, it holds that
\begin{align}\label{eq:darkpool-g}
    g_T(x)=\begin{cases}
    \displaystyle    0, & \text{if $x=0$},\\
     \displaystyle   +\infty, & \text{otherwise}.
    \end{cases}
\end{align}
The first term of the right-hand side of the objective~\eqref{eq:darkpoolJDP} refers to the linear price impact costs generated by trading in the traditional market, while the second term is the quadratic risk cost penalizing slow liquidation. Then, the goal of the investor is to minimize the liquidation cost that
\begin{align}\label{eq:jlpi}
v(t,x)&:=\inf_{u\in\mathcal{U}} J_{\rm DP}(t,x;u)=-\sup _{u\in\mathcal{U}}J(t,x;u)\notag\\
&=-\sup _{u\in\mathcal{U}}\mathbb{E}_t^{\Px}\left[\int_t^T\left(-\kappa |\xi_s|^2-c |X_s^u|^2\right) d s-g_T(X_T)\Big | X_t=x\right].
\end{align}
Using the exploratory formulation in \eqref{eq:RL-problem}, we first consider the entropy-regularized relaxed control problem with \eqref{eq:Xut} and \eqref{eq:jlpi} that
  \begin{align}\label{eq:omega}
\begin{cases}
\displaystyle w(t,x):=\sup_{\pi\in\Pi}\Ex^{\Px}\left[\int_t^T\int_{U} \left(-\kappa|u_1|^2-c|X_s^u|^2+\gamma l_p (\pi_s(u))\right)\pi_s(u)duds-g_T( X_T^u)\Big | X_t=x\right],\\[1.2em]
\displaystyle ~{\rm s.t.}~X_t^{\pi} =x-\int_0^t\int_{\R^2} u_1\pi_s(du)ds - \int_0^t\int_{\R^2}u_2\mathcal{N}(ds,du),\quad \forall t\in[0,T].
\end{cases}
\end{align}

To continue, we first relax the liquidation constraint by introducing a penalty term when the liquidation is not completely exercised. That is, we consider, for $\ell>0$,
\begin{align*}
J^{(\ell)}(t,x;u):= \mathbb{E}^{\Px}\left[\int_t^T\left(-\kappa |\xi_s|^2-c |X_s^u|^2\right) d s-\ell X_T^2\Big | X_t=x\right].
\end{align*}
Consequently, the associated exploratory formulation of the control problem under Tsallis entropy is given by
\begin{align}\label{eq:RL-problem-appl}
\begin{cases}
\displaystyle V^{(\ell)}(t,x):=\sup_{\pi\in\Pi} J^{(\ell)}(t,x;\pi)\\[0.8em]
\displaystyle~~~~~~~~~~~=\sup_{\pi\in\Pi}\Ex^{\Px}\left[\int_t^T\int_{U} \left(-\kappa|u_1|^2-c|X_s^u|^2+\gamma l_p (\pi_s(u))\right)\pi_s(u)duds-\ell X_T^2\Big | X_t=x\right],\\[1.2em]
\displaystyle ~{\rm s.t.}~X_t^{\pi} =x-\int_0^t\int_{\R^2} u_1\pi_s(du)ds - \int_0^t\int_{\R^2}u_2\mathcal{N}(ds,du),\quad \forall t\in[0,T].
\displaystyle 
\end{cases}
\end{align}
Then, we have
\begin{lemma}\label{lem:Vell}
The liquidation cost minimization reinforcement learning problem~\eqref{eq:RL-problem-appl} under the  Tsallis entropy regularizer has the following explicit value function given by, for any $\ell>0$,
    \begin{align*}
        V^{(\ell)} (t,x)=\frac{\alpha^{(\ell)}(t)}{2}x^2+\beta^{(\ell)}(t),\quad\forall (t,x)\in[0,T]\times\R,
    \end{align*}
where the coefficients are given by
\begin{align*}
    \alpha^{(\ell)}(t)&=-\frac{ \left(\ell\kappa (w -\lambda)+4c \kappa\right)e^{w (T-t)}+ \ell\kappa (w +\lambda)-4c \kappa}{ \left(\kappa (w +\lambda)+\ell\right)e^{w (T-t)} + \kappa (w -\lambda)-\ell},\quad {\rm and} \\
    \beta^{(\ell)}(t)&=\begin{cases}
    \displaystyle -\frac{p^2 \gamma ^{\frac{1}{p}}}{(2p-1)(p-1)}\int_t^T \left(\frac{1}{\pi}\sqrt{\frac{-\kappa \lambda \alpha^{(\ell)}(s)}{2}}\right)^{\frac{p-1}{p}}ds+\frac{\gamma}{p-1}(T-t),&p>1,\\[1.2em]
    \displaystyle \gamma \int_t^T \ln \left(\frac{\gamma \pi }{\sqrt{-\kappa \lambda \alpha^{(\ell)}(s)/2}}\right)ds,&p=1
    \end{cases}
\end{align*}
with the constant $w:=\sqrt{\lambda^2+\frac{4 c}{\kappa}}$. Moreover, the optimal policy is given by, for $u=(u_1,u_2)\in\R^2$,
{\footnotesize
\begin{align*}
\widehat{\pi}^{(\ell)}(u|t,x)=\begin{cases}
 \displaystyle   \left(\frac{p-1}{\gamma  p}\right)^{\frac{1}{p-1}}\left(\psi(t,x)- u_1V_x^{(\ell)}(t,x)+\lambda(V^{(\ell)}(t,x-u_2)-V^{(\ell)}(t,x))-\kappa u_1^2-c x^2+\frac{\gamma}{p-1}\right)_+^{\frac{1}{p-1}},~ p>1,\\ \\
 \displaystyle   \frac{\exp\left(- u_1V_x^\ell(t,x)+\lambda(V^{(\ell)}(t,x-u_2)-V^{(\ell)}(t,x))-\kappa u_1^2-c x^2-\gamma \right)}{\int_{\R^2}\exp\left(- u_1V_x^{(\ell)}(t,x)+\lambda(V^{(\ell)}(t,x-u_2)-V^{(\ell)}(t,x))-\kappa u_1^2-c x^2-\gamma \right) du},~ p=1.
\end{cases}
\end{align*}}
\end{lemma}

\begin{proof}
Under the formulation of problem~\eqref{eq:RL-problem-appl}, we have the following exploratory HJB equation that,  for $u=(u_1,u_2)\in\R^2$, 
\begin{align}\label{eq:exHJBnew}
\begin{cases}
\displaystyle    0={ V_t^{(\ell)}(t,x)+}\sup_{\pi_t\in\mathcal{P}(U)}\bigg\{-V_x^{(\ell)}(t,x)\int_{U} u_1\pi(u|t,x)du\\[1.2em]
\displaystyle \qquad\qquad+\lambda\int_{U}\left(V^{(\ell)}\left(t,x-u_2\right)-V^{(\ell)}(t,x)\right)\pi(u|t,x)du\\[1em]
\displaystyle\qquad\qquad+\int_{U} (-\kappa u_1^2-cx^2+\gamma l_p(\pi(u|t,x)))\pi(u|t,x)du\bigg\},\\[1.2em]
\displaystyle V^{(\ell)}(T,x)=-\ell x^2.
\end{cases}
\end{align}
To enforce the constraints $\int_{U} \pi(u|t,x)du=1$ and $\pi(u|t,x)\geq 0$ for 
 $(t,x,u)\in[0,T]\times\R^3$, we introduce the Lagrangian given by
\begin{align*}
{\cal L}(t,x,\pi,\xi,\psi)&=
-V_x^{(\ell)}(t,x)\int_{U} u_1\pi(u|t,x)du+\lambda\int_{U}\left(V^{(\ell)}\left(t,x-u_2\right)-V^{(\ell)}(t,x)\right)\pi(u|t,x)du\nonumber\\
&\quad+\int_{U} \left((-\kappa u_1^2-c x^2)\pi(u|t,x)+\frac{\gamma}{p-1}(\pi(u|t,x)-\pi^{p}(u|t,x))\right)du\notag\\
&\quad +\psi(t,x)\left(\int_{U} \pi(u|t,x)du-1\right)
+\int_U \zeta(t,u)\pi(u|t,x)du,
\end{align*}
where $\psi(t,x)$ is a function of $(t,x)\in [0,T]\times \R$ and $\zeta(t,u)$ is a function of $(t,u)\in [0,T] \times \R^2$.
It follows from the first-order condition that, the candidate optimal policy is 
{\footnotesize
\begin{align*}
\widehat{\pi}^{(\ell)}(u|t,x)=\begin{cases}
\displaystyle    \left(\frac{p-1}{\gamma  p}\right)^{\frac{1}{p-1}}\left(\psi(t,x)- u_1V_x^{(\ell)}(t,x)+\lambda(V^{(\ell)}(t,x-u_2)-V^{(\ell)}(t,x))-\kappa u_1^2-c x^2+\frac{\gamma }{p-1}\right)_+^{\frac{1}{p-1}},~p>1,\\[0.8em]
 \displaystyle   \frac{\exp\left(- u_1V_x^{(\ell)}(t,x)+\lambda(V^{(\ell)}(t,x-u_2)-V^{(\ell)}(t,x))-\kappa u_1^2-c x^2-\gamma \right)}{\int_{\R^2}\exp\left(- u_1V_x^{(\ell)}(t,x)+\lambda(V^{(\ell)}(t,x-u_2)-V^{(\ell)}(t,x))-\kappa u_1^2-c x^2-\gamma \right) du},~p=1
\end{cases}
\end{align*}}with the multiplier $\zeta(t,u)$ given by
\begin{align*}
\zeta(t,u)=
    \left(u_1V_x^{(\ell)}(t,x)-\lambda(V^{(\ell)}(t,x-u_2)-V^{(\ell)}(t,x))+\kappa u_1^2+c x^2-\frac{
    \gamma }{p-1}-\psi(t,x)\right)_+,\quad \forall p>1.
\end{align*}
We only provide the details on the construction of the solution to \eqref{eq:exHJBnew} for the case $q>1$ as the case $q= 1$ is essentially the same. Consider the form $V^{(\ell)}(t,x)=\frac{\alpha^{(\ell)}}{2}(t)x^2+\beta^{(\ell)}(t)$ for $(t,x)\in[0,T]\times\R$. By substituting it into the above policy, we have, for $p>1$, 
\begin{align*}
\widehat{\pi}^{(\ell)}(u|x)
&= \left(\frac{(p-1)\tilde \psi(t,x)}{\gamma p}\right)^{\frac{1}{p-1}}\left(1-
    \frac{\kappa}{\tilde \psi(t,x)} \left(u_1+\frac{\alpha^{(\ell)}(t)x}{2\kappa}\right)^2+
    \frac{\alpha^{(\ell)}(t)\lambda }{2\tilde \psi(t,x)}(u_2-x)^2\right)_+^{\frac{1}{p-1}},
\end{align*}
where $\tilde \psi(t,x)= \psi(t,x)-c x^2+\frac{(\alpha^{(\ell)}(t))^2}{4\kappa}x^2-\frac{\alpha^{(\ell)}(t)\lambda}{2}x^2+\frac{\gamma }{p-1}$ is assumed to be greater than zero, which will be verified later. Then, using the constraint  $\int_{U} \pi(u|x)du=1$, we have
\begin{align*}
   1&= \left(\frac{(p-1)\tilde \psi(t,x)}{\gamma p}\right)^{\frac{1}{p-1}}\int_{\R^2}\left(1-
    \frac{\kappa}{\tilde \psi(t,x)} \left(u_1+\frac{\alpha^{(\ell)}(t)x}{2\kappa}\right)^2+
    \frac{\alpha^{(\ell)}(t)\lambda }{2\tilde \psi(t,x)}(u_2-x)^2\right)_+^{\frac{1}{p-1}}du\\
    &=\left(\frac{(p-1)\tilde \psi(t,x)}{\gamma p}\right)^{\frac{1}{p-1}}\sqrt{\frac{2\tilde \psi^2(t,x)}{-\lambda \kappa \alpha^{(\ell)}(t)}}\int_{y^2+z^2\leq 1}\left(1-
    y^2-
    z^2\right)^{\frac{1}{p-1}}dydz\\
    &=\tilde \psi^{\frac{p}{p-1}}(t,x)\left(\frac{p-1}{\gamma p}\right)^{\frac{1}{p-1}}\sqrt{\frac{2}{-\lambda \kappa \alpha^{(\ell)}(t)}}\Psi,
\end{align*}
where $ \Psi:=\int_{y^2+z^2\leq 1}\left(1-
    y^2-
    z^2\right)^{\frac{1}{p-1}}dydz$. By using the polar coordinate transformation $(y,z)=(\rho cos \theta,\rho \sin \theta)$ for $(\rho,\theta) \in [0,1]\times [0,2\pi]$, we derive that
\begin{align*}
    \Psi&=\int_{y^2+z^2\leq 1}\left(1-
    y^2-
    z^2\right)^{\frac{1}{p-1}}dydz=\int_0^{2\pi}\int_{0}^{1}\left(1-
    \rho^2\right)^{\frac{1}{p-1}}\rho d\rho d\theta=\frac{p-1}{p}\pi.
\end{align*}
Furthermore, it holds that
\begin{align*}
\tilde \psi(t,x)\equiv \left(\frac{1}{\Psi}\sqrt{\frac{-\lambda \kappa \alpha^{(\ell)}(t)}{2}}\right)^{\frac{p-1}{p}}\left(\frac{\gamma p}{p-1}\right)^{\frac{1}{p}}=\left(\frac{1}{\pi}\sqrt{\frac{-\lambda \kappa \alpha^{(\ell)}(t)}{2}}\right)^{\frac{p-1}{p}}\frac{p }{p-1}\gamma ^{\frac{1}{p}},
\end{align*}
and $\psi(t,x)=\left(c -\frac{(\alpha^{(\ell)}(t))^2}{4\kappa}+\frac{\alpha^{(\ell)}(t)\lambda }{2}\right)x^2+\left(\frac{1}{\pi}\sqrt{\frac{-\kappa \lambda \alpha^{(\ell)}(t)}{2}}\right)^{\frac{p-1}{p}}\frac{p }{p-1}\gamma ^{\frac{1}{p}}-\frac{\gamma }{p-1}$. As $\tilde \psi(t,x)$  does not depend on  $x\in \R$, we may write it as $\tilde \psi (t)$.

To solve \eqref{eq:exHJBnew}, we consider the following moments, for $(u_1,u_2)\in\R^2$, 
\begin{align*}
    &\int_{\R^2} u_1^m \pi^{(\ell)}(u|t,x)du\\
    &=\tilde \psi(t)^{\frac{p}{p-1}}\left(\frac{p-1}{\gamma p}\right)^{\frac{1}{p-1}}\sqrt{\frac{2}{-\lambda \kappa \alpha^{(\ell)}(t)}}\int_{y^2+z^2\leq 1}\left(\sqrt{\frac{\tilde{\psi}(t)}{\kappa }}y-\frac{\alpha^{(\ell)}(t)x}{2\kappa}\right)^m \left(1-
    y^2-z^2\right)^{\frac{1}{p-1}}dydz\notag\\
    &=\tilde{\psi}(t)^{\frac{p}{p-1}}\left(\frac{p-1}{\gamma p}\right)^{\frac{1}{p-1}}\sqrt{\frac{2}{-\lambda \kappa \alpha^{(\ell)}(t)}}\int_0^{2\pi}\int_0^1\left(\sqrt{\frac{\tilde \psi(t)}{\kappa}}\rho \cos \theta-\frac{\alpha^{(\ell)}(t)x}{2\kappa}\right)^m \left(1-
   \rho^2\right)^{\frac{1}{p-1}}\rho d\rho d\theta\\
   &=\begin{cases}
     -\frac{\alpha^{(\ell)}(t)x}{2\kappa},~m=1,\\
     \ \\
     \frac{\tilde{\psi}(t) (p-1)}{2\kappa (2p-1 )}+\frac{(\alpha^{(\ell)}(t) x)^2}{4\kappa^2},~m=2,
   \end{cases}
   \end{align*}
   as well as 
   \begin{align*}
    &\int_{\R^2} u_2^m \pi^{(\ell)}(u|t,x)du\\
    &=\tilde \psi(t)^{\frac{p}{p-1}}\left(\frac{p-1}{\gamma p}\right)^{\frac{1}{p-1}}\sqrt{\frac{2}{-\lambda \kappa\alpha^{(\ell)}(t)}}\int_{y^2+z^2\leq 1}\left(\sqrt{-\frac{2\tilde{\psi}(t)}{\alpha^{(\ell)}(t)\lambda }}y+x\right)^m \left(1-
    y^2-z^2\right)^{\frac{1}{p-1}}dydz\notag\\
    &=\tilde{\psi}(t)^{\frac{p}{p-1}}\left(\frac{p-1}{\gamma p}\right)^{\frac{1}{p-1}}\sqrt{\frac{2}{-\lambda \kappa \alpha^{(\ell)}(t)}}\int_0^{2\pi}\int_0^1\left(\sqrt{-\frac{2\tilde{\psi}(t)}{\alpha^{(\ell)}(t)\lambda }}\rho \cos \theta+x\right)^m \left(1-
   \rho^2\right)^{\frac{1}{p-1}}\rho d\rho d\theta\\
   &=\begin{cases}
  \displaystyle     x, & m=1,\\[0.8em]
  \displaystyle -\frac{\tilde{\psi}(t) (p-1)}{\alpha^{(\ell)}(t)\lambda(2p-1)} +x^2, & m=2,
   \end{cases}
\end{align*}
and 
\begin{align*}
    &~\int_{\R^2}  \frac{1}{p-1}\left(\pi^{(\ell)}(u|t,x)-\pi^{(\ell)}(u|t,x)^{p}\right)du\\
    &=\frac{1}{p-1}-\frac{1}{p-1}\left(\frac{(p-1)\tilde\psi(t)}{\gamma p}\right)^{\frac{p}{p-1}}\sqrt{\frac{2\tilde\psi(t)^2}{-\lambda \kappa \alpha^{(\ell)}(t)}}\int_0^{2\pi}\int_{0}^1 \left(1-
\rho^2\right)^{\frac{p}{p-1}}\rho d\rho d\theta\\
&=\frac{1}{p-1}-\frac{\tilde\psi(t)}{(2p-1)\gamma }.
\end{align*}
Then, substituting the candidate solution $V^{(\ell)}(t,x)=\frac{\alpha^{(\ell)}(t)}{2}x^2+\beta^{(\ell)}(t)$ back into \eqref{eq:exHJBnew}, we obtain that
\begin{equation*}
\begin{cases}
 \displaystyle {\alpha^{(\ell)}_t(t)}=-\frac{(\alpha^{(\ell)}(t))^2}{2\kappa}+{\lambda} \alpha^{(\ell)}(t)+2c, ~~\alpha^{(\ell)}(T)=-\ell,\\[0.8em]  
 \displaystyle \beta^{(\ell)}_t(t)=\frac{\tilde \psi(t) (p-1)}{2p-1}-\gamma \left(\frac{1}{p-1}-\frac{\tilde \psi(t)}{(2p-1)\gamma}\right),~~\beta^{(\ell)}(T)=0.
\end{cases}
\end{equation*}
Solving the above equation, we have 
\begin{align*}
\begin{cases}
 \displaystyle    \alpha^{(\ell)}(t)=-\frac{ \left(\ell\kappa (w -\lambda)+4c \kappa\right)e^{w (T-t)}+ \ell\kappa (w +\lambda)-4c \kappa}{ \left(\kappa (w +\lambda)+\ell\right)e^{w (T-t)} + \kappa (w -\lambda)-\ell},\\[1.2em]
  \displaystyle   \beta^{(\ell)}(t)=-\frac{p^2 \gamma ^{\frac{1}{p}}}{(2p-1)(p-1)}\int_t^T \left(\frac{1}{\pi}\sqrt{\frac{-\kappa \lambda \alpha^{(\ell)}(s)}{2}}\right)^{\frac{p-1}{p}}ds+\frac{\gamma}{p-1}(T-t).
\end{cases}
\end{align*}
Thus, the proof of the lemma is completed.
\end{proof}

Building upon Lemma~\ref{lem:Vell}, under the liquidation constrain, we consider the reinforcement learning problem in the limit sense that 
\begin{align*}
V(t,x):=\lim_{\ell \to \infty}V^{(\ell)} (t,x),\quad \forall (t,x) \in [0,T]\times \R.
\end{align*}
Then, by some standard verification arguments, we have the next result.
\begin{theorem}
\label{th:LQ-entropytsallis}
The liquidation cost minimization problem~\eqref{eq:omega} under the Tsallis entropy  has the explicit solution that
    \begin{align*}
        V(t,x)=\frac{\alpha^*(t)}{2}x^2+\beta^*(t),\quad \forall (t,x)\in[0,T]\times\R,
    \end{align*}
where the coefficients are specified by
\begin{align*}
    \alpha^*(t)&= \lim_{\ell\to\infty}\alpha^{(\ell)}(t)=- \kappa (\omega -\lambda)-\frac{ 2\kappa \omega}{ e^{\omega (T-t)} -1}<0,\\
    \beta^*(t)&= \lim_{\ell\to\infty}\beta^{(\ell)}(t)=\begin{cases}
      \displaystyle  -\frac{p^2 \gamma 
    ^{\frac{1}{p}}}{(2p-1)(p-1)}\int_t^T \left(\frac{1}{\pi}\sqrt{\frac{-\kappa \lambda \alpha^*(s)}{2}}\right)^{\frac{p-1}{p}}ds+\frac{\gamma }{p-1}(T-t),&p>1,\\[1.2em]
    \displaystyle \gamma \int_t^T \ln \left(\frac{\gamma \pi }{\sqrt{-\kappa \lambda \alpha^*(s)/2}}\right)ds,&p=1
    \end{cases}
\end{align*}
with $\omega:=\sqrt{\lambda^2+\frac{4 c}{\kappa}}$. The optimal policy for problem ~\eqref{eq:omega} is given by
\begin{align}\label{eq:optimal-policy-darkpool}
\widehat{\pi}(u|t,x)
    &=\begin{cases}
 \displaystyle     \left(\frac{p-1}{\gamma p}\right)^{\frac{1}{p-1}}\left(\tilde \psi(t)-
    {\kappa} \left(u_1+\frac{\alpha^*(t)x}{2\kappa}\right)^2+
    \frac{\alpha^*(t)\lambda }{2}(u_2-x)^2\right)_+^{\frac{1}{p-1}}, & p>1,\\[1.4em]
\displaystyle    \frac{1 }{\gamma \pi  }\sqrt{-\frac{\kappa \lambda \alpha^*(t)}{2}}\exp\left(-\frac{\kappa \left(u_1+\frac{\alpha^*(t) x}{2\kappa }\right)^2}{\gamma}+\frac{\lambda \alpha^*(t) (u_2-x)^2}{2 \gamma}\right), & p=1.
    \end{cases} 
\end{align}
Here, $\tilde \psi(t)=
        \left(\frac{1}{\pi}\sqrt{\frac{-\kappa \lambda \alpha^*(t)}{2}}\right)^{\frac{p-1}{p}}\frac{p \gamma^{\frac{1}{p}}}{p-1}$ for $t\in[0,T]$.

\end{theorem}

We have the next remark on different entropy index:

\begin{remark}\label{example-one-distribution}
    For the case with $p=1$, the optimal policy $\widehat{\pi}$ given by \eqref{eq:optimal-policy-darkpool} is a two-dimensional Gaussian distribution; while for $p>1$, the optimal policy becomes a two-dimensional $p$-Gaussian distribution with a compact support set, see Figure \ref{fig:policy-darkpool} for the illustration.
        The mean and variance of the policy $(u_1,u_2)$ are given by 
  \begin{align}\label{eq:moment01}
&{\rm mean} (u_1,u_2)=\left(\int_{\R^2} u_1 \widehat{\pi}(u)du,\int_{\R^2} u_2 \widehat{\pi}(u)du\right)=
  \left( -\frac{\alpha^{(\ell)}(t)x}{2\kappa},x\right),\\
&{\rm Var} (u_1,u_2)=\left(\int_{\R^2} u_1^2 \widehat{\pi}(u)du-\left(\int_{\R^2} u_1 \widehat{\pi}(u)du\right)^2,\int_{\R^2} u_2^2 \widehat{\pi}(u)du-\left(\int_{\R^2} u_2 \widehat{\pi}(u)du\right)^2\right)\notag\\
&\qquad \qquad =
 \left(
     \frac{\tilde{\psi}(t) (p-1)}{2\kappa (2p-1 )},-\frac{\tilde{\psi}(t) (p-1)}{\alpha^{(\ell)}(t)\lambda(2p-1)}\right).
\end{align}
    In fact, for $p>1$ and $(t,x)\in[0,T]\times \R_+$, we have
    \begin{align}
    u_1&\in{\cal O}_1^p:=\left[-\frac{\alpha^*(t)x}{2\kappa}-\sqrt{\frac{\tilde{\psi}(t)}{\kappa}},-\frac{\alpha^*(t)x}{2\kappa}+\sqrt{\frac{\tilde{\psi}(t)}{\kappa}}\right],\label{eq:set-1}\\
    u_2&\in{\cal O}_2^p:=\left[x-\sqrt{-\frac{2\tilde{\psi}(t)}{\lambda \alpha^*(t)}},x+\sqrt{-\frac{2\tilde{\psi}(t)}{\lambda \alpha^*(t)}}\right],\label{eq:set-2}
    \end{align}
    where the functions $t\mapsto\alpha^*(t)$ and $t\mapsto\tilde{\psi}(t)$ are given in Theorem \ref{th:LQ-entropytsallis}.
\begin{figure}[htbp]
\centering
  \subfigure[]{
        \includegraphics[width=5cm]{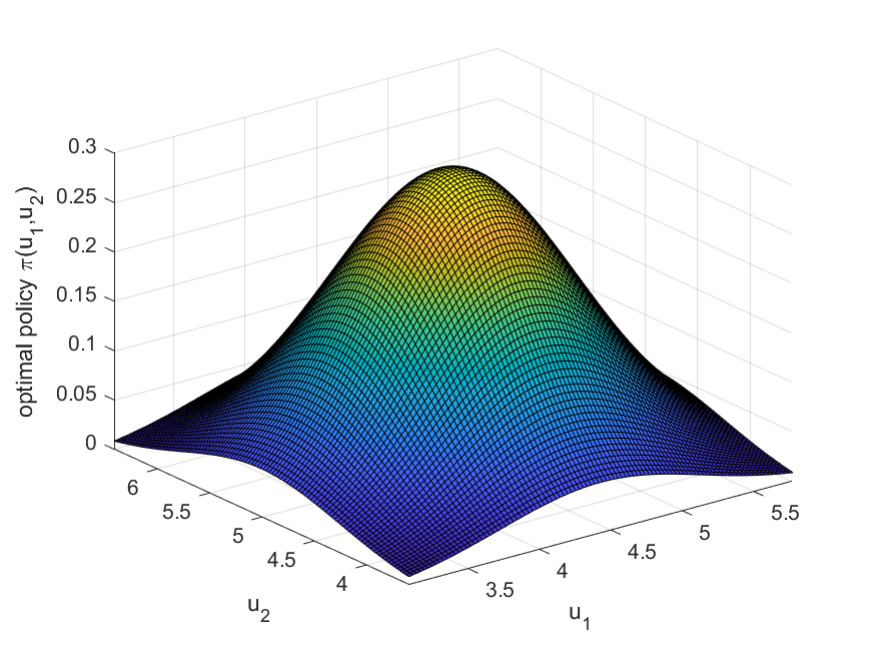}
    }\hspace{-4mm}
  \subfigure[]{
        \includegraphics[width=5cm]{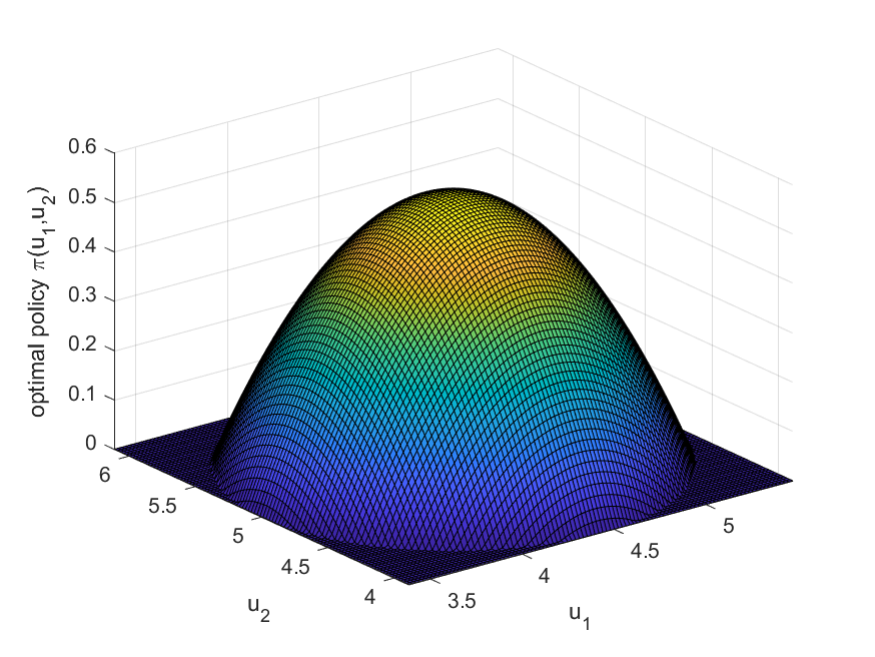}
    }\hspace{-4mm}
  \subfigure[]{
        \includegraphics[width=5cm]{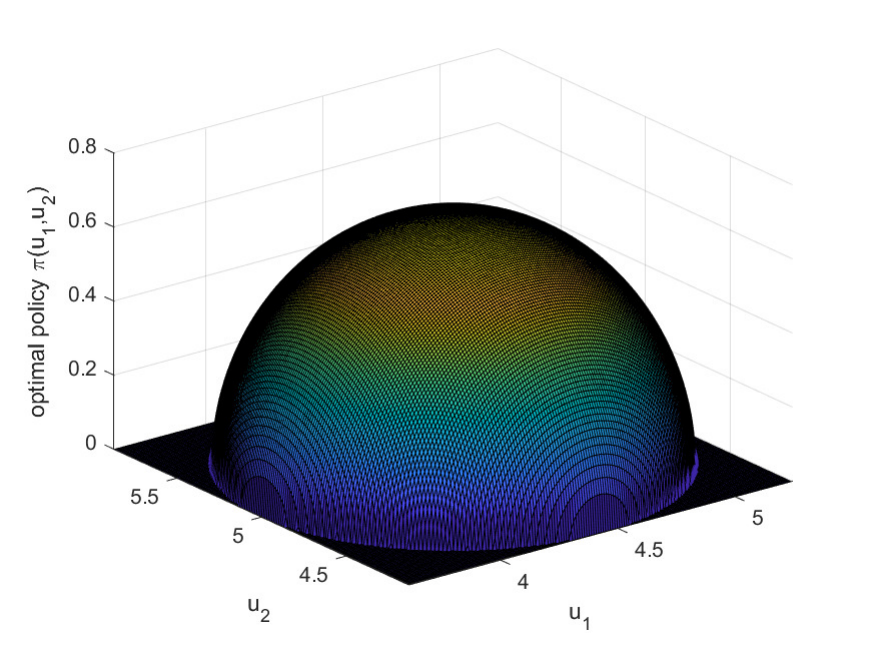}
    }
 \caption{\small (a) The optimal policy $(u_1,u_2)\to \hat{\pi}(u_1,u_2)$ with $p=1$. (b): The optimal policy $(u_1,u_2)\to \hat{\pi}(u_1,u_2)$ with $p=2$. (c): The optimal policy $(u_1,u_2)\to \hat{\pi}(u_1,u_2)$ with $p=3$. The model parameters are set to be  $\lambda=1,~\kappa=1,~c=1,~\gamma=1,~t=1,~T=2,~x=5$.}\label{fig:policy-darkpool}
\end{figure}
Furthermore, we can see that the length of the compact support set for the control $u_1$ (or $u_2$) can either decrease initially and then increase as the entropy index $p$ increases  (see Figure \ref{fig:set-1}), or decrease monotonically as $p$ increases (see Figure \ref{fig:set-2}). As the index $p\to \infty$ , we have that
\begin{align*}
\lim_{p \to \infty}\tilde \psi(t)=\frac{1}{\pi}\sqrt{\frac{-\kappa \lambda \alpha^*(t)}{2}},\quad \forall t\in[0,T).
\end{align*}
Thus, the compact support sets in \eqref{eq:set-1}-\eqref{eq:set-2} converge to the following domains respectively
\begin{align*}
\lim_{p\to \infty}{\cal O}_1^p&=\left[-\frac{\alpha^*(t)x}{2\kappa}-\sqrt{\frac{1}{\pi}\sqrt{\frac{- \lambda \alpha^*(t)}{2\kappa}}},-\frac{\alpha^*(t)x}{2\kappa}+\sqrt{\frac{1}{\pi}\sqrt{\frac{- \lambda \alpha^*(t)}{2\kappa}}}\right],\\
\lim_{p\to \infty}{\cal O}_2^p&=\left[x-\sqrt{\frac{1}{\pi}\sqrt{\frac{-2\kappa }{\lambda \alpha^*(t)}}},x+\sqrt{\frac{1}{\pi}\sqrt{\frac{-2\kappa }{\lambda \alpha^*(t)}}}\right],
\end{align*}
which can also be observed from Figure \ref{fig:set-1} and  Figure \ref{fig:set-2}.
It demonstrates that the entropy index 
$p$ can essentially determine the compact support of the exploratory policy. As $p\to 1$ this framework recovers the Shannon entropy limit. Therefore, while a Tsallis-regularized algorithm may not offer a direct advantage in convergence or sample efficiency over its Shannon counterpart, it brings a practical benefit of bounded sampling of actions from the compact support, which might be crucial in some applications. 

\begin{figure}[htbp]
\centering
  \subfigure[]{
        \includegraphics[width=7cm]{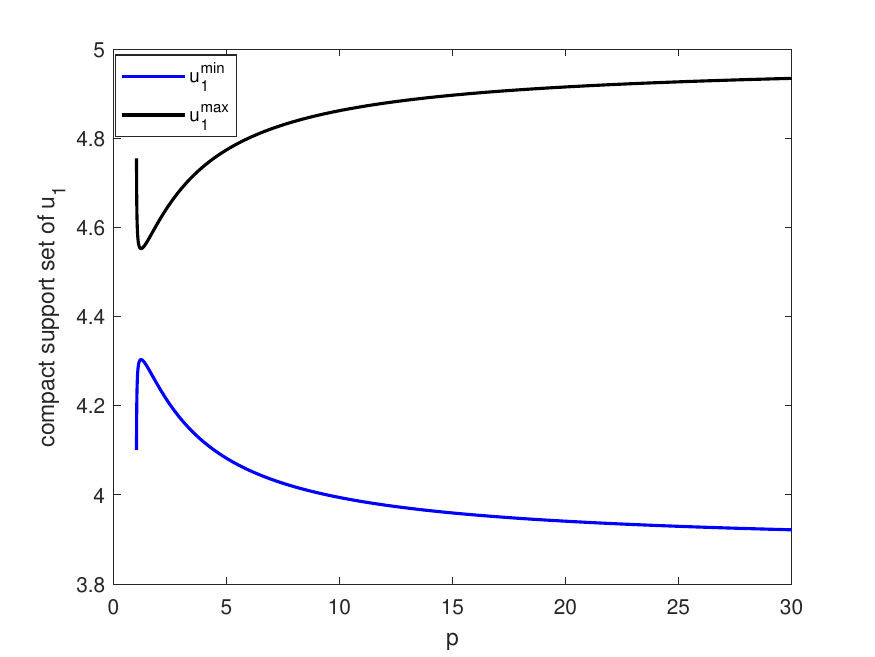}
    }\hspace{-4mm}
  \subfigure[]{
        \includegraphics[width=7cm]{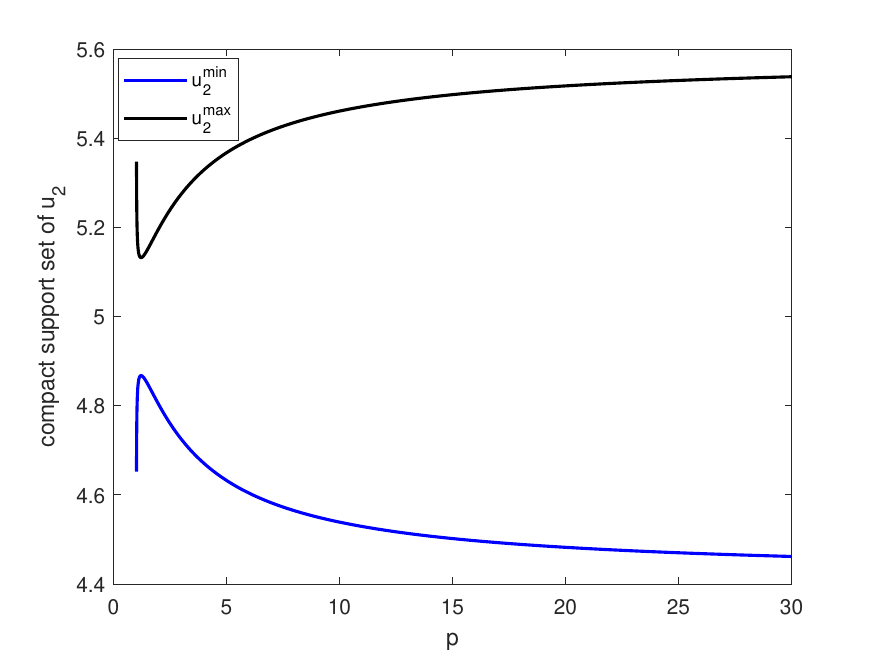}
    }
     \caption{\small (a): The compact support set of $u_1$. (b) The compact support set of $u_2$. The model parameters are set to be  $\lambda=1,~\kappa=1,~c=1,~\gamma=0.001,~t=1,~T=2,~x=5$.}\label{fig:set-1}
    \end{figure}
 \begin{figure}[htbp]
\centering
  \subfigure[]{
        \includegraphics[width=7cm]{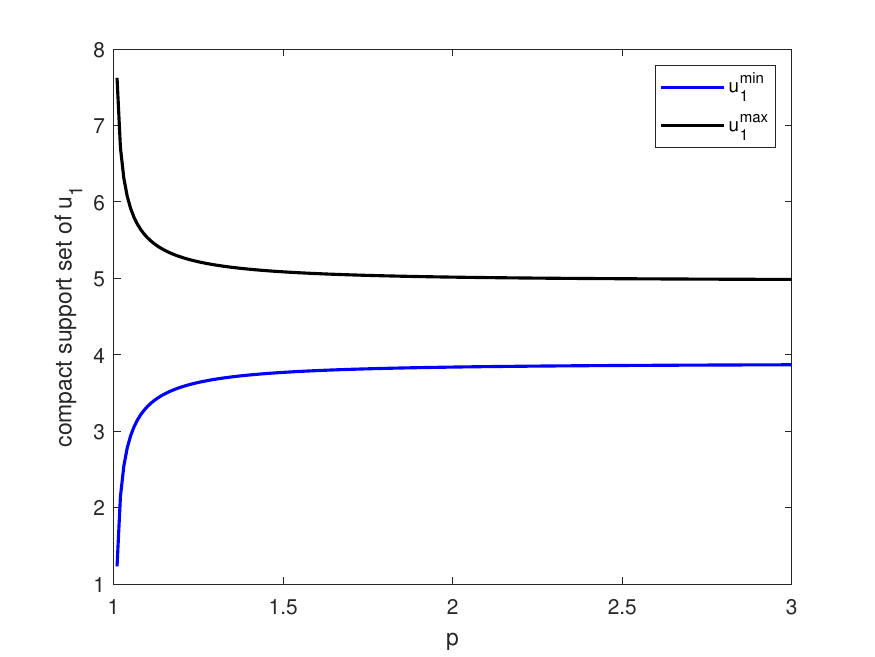}
    }\hspace{-4mm}
     \subfigure[]{
        \includegraphics[width=7cm]{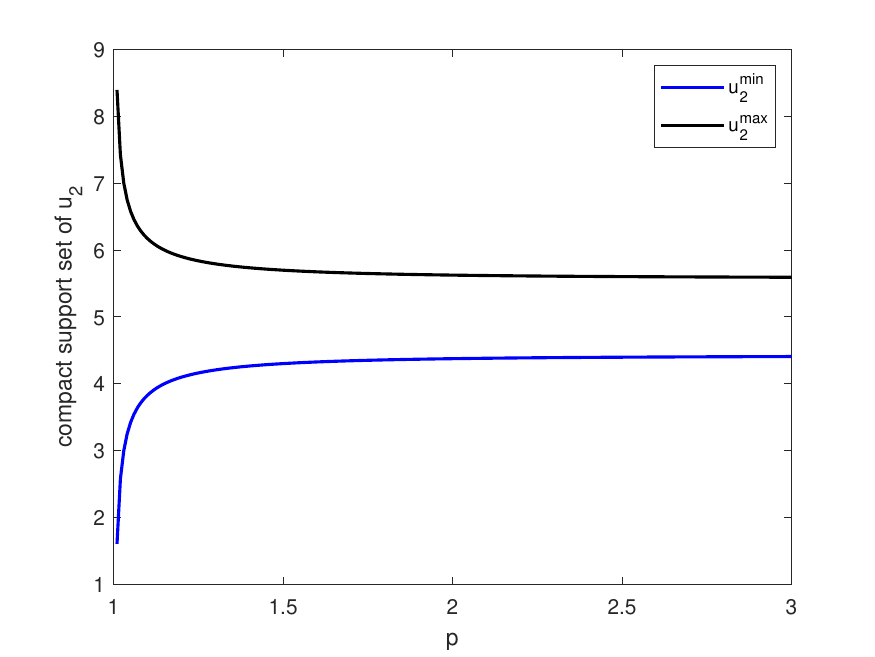}
    }\hspace{-4mm}
 \caption{\small (a): The compact support set of $u_1$. (b) The compact support set of $u_2$. The model parameters are set to be  $\lambda=1,~\kappa=1,~c=1,~\gamma=0.1,~t=1,~T=2,~x=5$.}\label{fig:set-2}
\end{figure}
\end{remark}

\begin{remark}
When the temperature parameter $\gamma$ goes to $0$, we have $\tilde{\psi}(t) (p-1)\to 0$. It then follows that
\begin{align*}
\int_{\R^2} (u_1^2+u_2^2) \pi(u|t,x)du &=
\frac{\tilde{\psi}(t) (p-1)}{2\kappa (2p-1 )}+\frac{(\alpha^* (t) x)^2}{4\kappa^2}-\frac{\tilde{\psi}(t) (p-1)}{\alpha^*(t)\lambda(2p-1)} +x^2\\
&\xrightarrow[]{\gamma\to 0} \left(\frac{\alpha^* (t) x}{2\kappa}\right)^2+x^2,
\end{align*}
which yields the convergence of the optimal trading policy to a  deterministic strategy that $(\xi_t,\eta_t)\xrightarrow[\gamma\to 0]{L^2} \left(\frac{\alpha^* (t) x}{2\kappa},x\right)$ for all $(t,x)\in [0,T]\times\R$.  Note that $( \frac{\alpha^*(t)x}{2\kappa}, x)$ is the optimal strategy to the classical stochastic control problem ($\gamma=0$). That is, the exploratory problem reduces to the classical problem as temperature parameter $\gamma \to 0$.

\end{remark}
Due to the singularity of terminal condition in \eqref{eq:darkpool-g}, applying q-learning algorithm directly to the primal problem \eqref{eq:darkpoolJDP} may bring great numerical error. Therefore, we provide a parameterization method of the value function and q-function of the auxiliary problem \eqref{eq:RL-problem-appl}, which can also help us learn the value function  given by \eqref{eq:darkpoolJDP}.
Let us define the true parameters by
\begin{align}
\theta_1^*=\kappa (w-\lambda),\quad
\theta_2^*=\kappa (w+\lambda),\quad
\theta_3^*=w,\quad
\theta_4^*=c\kappa,\quad
\theta_5^*=\kappa \lambda,
\nonumber\\
\zeta_1^*=\kappa (w-\lambda),\quad
\zeta_2^*=\kappa (w+\lambda),\quad
\zeta_3^*=w,\quad
\zeta_4^*=c\kappa,\quad
\zeta_5^*=\kappa \lambda\quad
\zeta_6^*=\kappa.
\end{align}
In lieu of Theorem~\ref{th:LQ-entropytsallis}, we can parameterize the optimal value function and the optimal q-function in the exact form by
 \begin{align*}
     J^{\theta}(t,x)=&-\frac{1}{2}\frac{ \left(\ell\theta_1+4\theta_4\right)e^{\theta_3 (T-t)}+ \ell\theta_2-4\theta_4}{ \left(\theta_2+\ell\right)e^{\theta_3 (T-t)} + \theta_1^*-\ell}x^2+\frac{\gamma }{p-1}(T-t)\\
    & -\frac{p^2 \gamma 
    ^{\frac{1}{p}}}{(2p-1)(p-1)}\int_t^T \left(\frac{1}{\pi}\sqrt{\frac{\theta_5 }{2}\frac{ \left(\ell\theta_1+4\theta_4\right)e^{\theta_3 (T-t)}+ \ell\theta_2-4\theta_4}{ \left(\theta_2+\ell\right)e^{\theta_3 (T-t)} + \theta_1-\ell}}\right)^{\frac{p-1}{p}}ds,\\
     q^{\zeta}(t,x,\xi,\eta)=&-\zeta_6\left(\xi-\frac{1}{2\zeta_6}\frac{ \left(\ell\zeta_1+4\zeta_4\right)e^{\zeta_3 (T-t)}+ \ell\zeta_2-4\zeta_4}{ \left(\zeta_2+\ell\right)e^{\zeta_3 (T-t)} + \zeta_1-\ell}x\right)^2\\
    &-\frac{\zeta_5}{2\zeta_6}\frac{ \left(\ell\zeta_1+4\zeta_4\right)e^{\zeta_3 (T-t)}+ \ell\zeta_2-4\zeta_4}{ \left(\zeta_2+\ell\right)e^{\zeta_3 (T-t)} + \zeta_1-\ell}(\eta-x)^2\\
&+\frac{p^2\gamma^{\frac{1}{p}}}{(p-1)(2p-1)}\left(\frac{1}{\pi}\sqrt{\frac{\zeta_5}{2}\frac{ \left(\ell\zeta_1+4\zeta_4\right)e^{\zeta_3 (T-t)}+ \ell\zeta_2-4\zeta_4}{ \left(\zeta_2+\ell\right)e^{\zeta_3 (T-t)} + \zeta_1-\ell}}\right)^{\frac{p}{p-1}}-\frac{\gamma}{p-1},
\end{align*}
where $(\theta_1,\theta_2,\theta_3,\theta_4,\theta_5)\in\R_+^5$ and $(\zeta_1,\zeta_2,\zeta_3,\zeta_4,\zeta_5,\zeta_6)\in\R_+^6$ are unknown parameters to be learned. As a result, the parameterized policy $\pi^{\zeta}$ is given by
{\small\begin{align}\label{eq:parameter-policy}
&\pi^{\zeta}(\xi,\eta|t,x)=\left(\frac{p-1}{p\gamma}\right)^{\frac{1}{p-1}}\Bigg(\left(\frac{1}{\pi}\sqrt{\frac{\zeta_5}{2}\frac{ \left(\ell\zeta_1+4\zeta_4\right)e^{\zeta_3 (T-t)}+ \ell\zeta_2-4\zeta_4}{ \left(\zeta_2+\ell\right)e^{\zeta_3 (T-t)} + \zeta_1-\ell}}\right)^{\frac{p-1}{p}}\frac{p }{p-1}\gamma ^{\frac{1}{p}}\\
&-\zeta_6\left(\xi-\frac{1}{2\zeta_6}\frac{ \left(\ell\zeta_1+4\zeta_4\right)e^{\zeta_3 (T-t)}+ \ell\zeta_2-4\zeta_4}{ \left(\zeta_2+\ell\right)e^{\zeta_3 (T-t)} + \zeta_1-\ell}x\right)^2-\frac{\zeta_5}{2\zeta_6}\frac{ \left(\ell\zeta_1+4\zeta_4\right)e^{\zeta_3 (T-t)}+ \ell\zeta_2-4\zeta_4}{ \left(\zeta_2+\ell\right)e^{\zeta_3 (T-t)} + \zeta_1-\ell}(\eta-x)^2\Bigg)_+^{\frac{1}{p-1}}.\nonumber
\end{align}}

 \textbf{Simulator}: In what follows, we apply Algorithm \ref{Alg:Tsallis-q-Learning} with the above parameterized value function and q-function.  To generate sample trajectories, we first apply the acceptance-rejection sampling method (\citealt{flury1990acceptance}) to generate the control pair $(u^1_t,u^2_t)$  from the $p$-Gaussian distribution with density function given by \eqref{eq:parameter-policy} at time $t$. Then, the control pair  $(u^1_t,u^2_t)$ is used to the following simulator
 \begin{align*}
X_{t+\Delta t}-X_t=- u^1_t\Delta t - u^2_t N(\Delta t),
\end{align*}
where $N(\Delta t)$ is a Poisson random variable with rate $\lambda \Delta t$. 

\textbf{Algorithm Inputs}: We set the coefficients of the simulator to $\lambda=0.01,X_0=2,T=0.25$, the known parameters as $\gamma=0.01,p=3,c=1,\kappa=1,\ell=10,x=2,T=0.25$, the time step as $dt=0.01$, and the number of iterations as $N=10000$. 
The learning rates are set as follows:
{\small\begin{align*}
    \alpha_{\theta_1}(k)&=\begin{cases}
     \displaystyle   0.01,~\text{if } 1\leq k\leq 2500,\\
     \displaystyle   \frac{0.001}{{\rm linspace}_{(1,20,N)}(k)},~\text{if } 2500< k\leq N,
    \end{cases}~
     \alpha_{\theta_2}(k)=\begin{cases}
     \displaystyle   0.005,~\text{if } 1\leq k\leq 4000,\\
     \displaystyle   \frac{0.005}{{\rm linspace}_{(1,100,N)}(k)}, ~\text{if } 4000< k\leq N,
    \end{cases}\\
     \alpha_{\theta_3}(k)&=\begin{cases}
     \displaystyle   0.01,~\text{if } 1\leq k\leq 4000,\\
     \displaystyle   \frac{0.005}{{\rm linspace}_{(1,20,N)}(k)},~\text{if } 4000< k\leq N,
    \end{cases}~
     \alpha_{\theta_4}(k)=\begin{cases}
     \displaystyle   0.03,~\text{if } 1\leq k\leq 3000,\\
      \displaystyle  \frac{0.005}{{\rm linspace}_{(1,20,N)}(k)}, ~\text{if } 3000< k\leq N,
    \end{cases}\\
    \alpha_{\theta_5}(k)&=\begin{cases}
     \displaystyle   0.05,~\text{if } 1\leq k\leq 3000,\\
     \displaystyle   \frac{0.0005}{{\rm linspace}_{(1,20,N)}(k)},~\text{if } 3000< k\leq N,
    \end{cases}~
     \alpha_{\zeta_1}(k)=\begin{cases}
    \displaystyle    0.03,~\text{if } 1\leq k\leq 3500,\\
     \displaystyle   \frac{0.00135}{{\rm linspace}_{(1,10,N)}(k)}, ~\text{if } 3500< k\leq N,
    \end{cases}\\
\alpha_{\zeta_2}(k)&=\begin{cases}
      \displaystyle  0.1,~\text{if } 1\leq k\leq 3500,\\
      \displaystyle  \frac{0.0002}{{\rm linspace}_{(1,500,N)}(k)},~\text{if } 3500< k\leq N,
        \end{cases}~
     \alpha_{\zeta_3}(k)=\begin{cases}
      \displaystyle  0.1,~\text{if } 1\leq k\leq 2000,\\
      \displaystyle  0.002, ~\text{if } 2000< k\geq 5000,\\
      \displaystyle   \frac{0.0005}{{\rm linspace}_{(1,20,N)}(k)},~\text{if } k \geq 5000,
    \end{cases}\\
     \alpha_{\zeta_4}(k)&=\begin{cases}
      \displaystyle  0.005,~\text{if } 1\leq k\leq 7000,\\
     \displaystyle   \frac{0.001}{{\rm linspace}_{(1,100,N)}(k)},~\text{if } 7000< k\leq N,
    \end{cases}~
     \alpha_{\zeta_5}(k)=\begin{cases}
    \displaystyle    0.006,~\text{if } 1\leq k\leq 5000,\\
     \displaystyle    \frac{0.002}{{\rm linspace}_{(1,10,N)}(k)},~\text{if }  5000<k\leq N,
    \end{cases}\\
    \alpha_{\zeta_6}(k)&=\begin{cases}
    \displaystyle    0.006,~\text{if } 1\leq k\leq 5000,\\
    \displaystyle     \frac{0.002}{{\rm linspace}_{(1,10,N)}(k)},~\text{if }  5000<k\leq N,
    \end{cases}
\end{align*}}
where 
${\rm linspace}_{(a,b,n)}(\cdot)$ is the Matlab function that returns a row vector of $n$ points linearly spaced between and including $a$ and $b$ with
the spacing between the points being $\frac{b-a}{n-1}$.

\begin{figure}[htb]
    \centering
\includegraphics[width=0.85\textwidth]{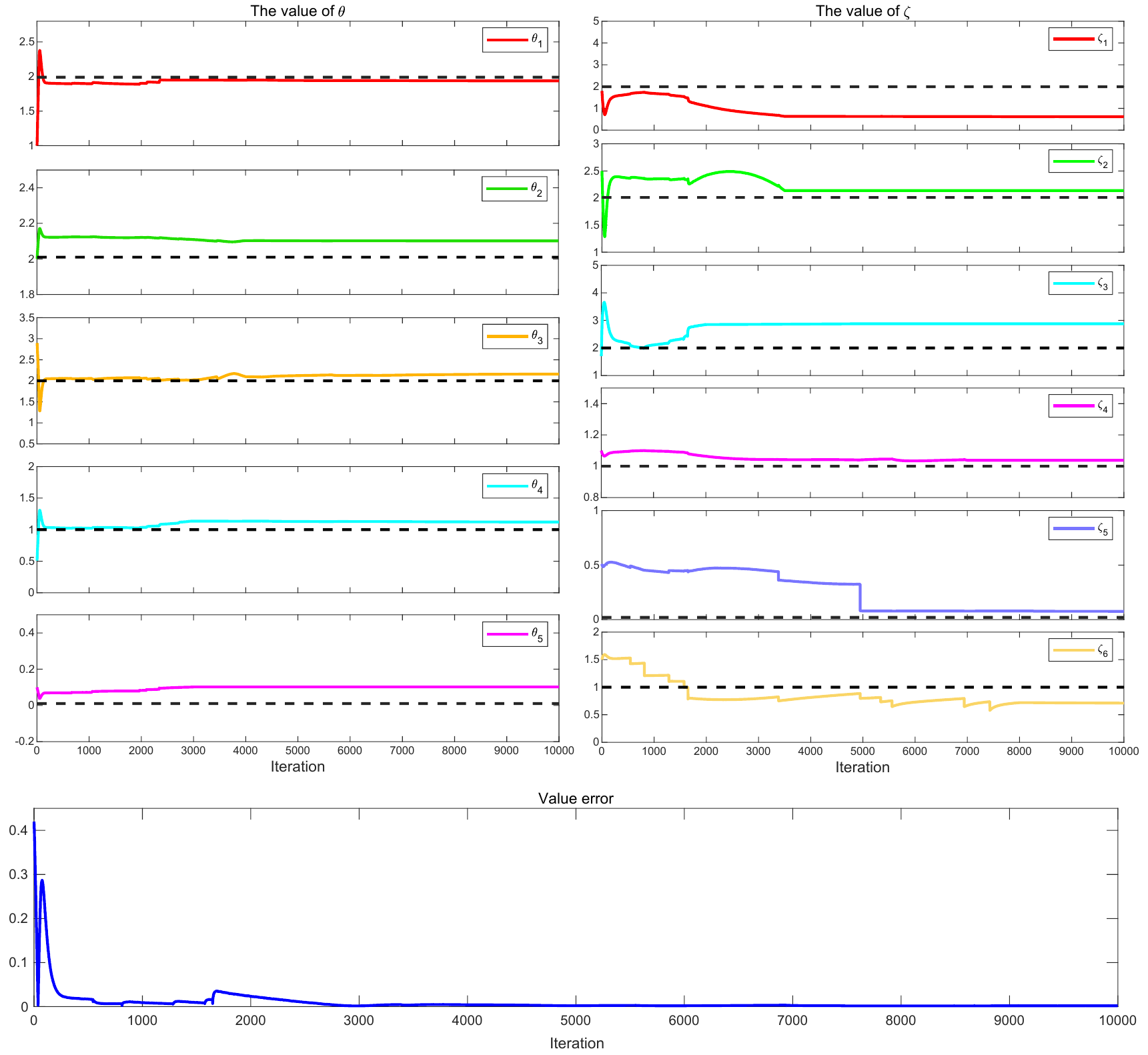}
    \caption{Convergence of Algorithm \ref{Alg:Tsallis-q-Learning} using a market simulator. The upper panels show the convergence of parameter iterations for $(\theta_1,\theta_2,\theta_3,\theta_4,\theta_5,\zeta_1,\zeta_2,\zeta_3,\zeta_4,\zeta_5,\zeta_6)$; the bottom panel shows the value error along the iterations. }
    \label{fig:RL1}
\end{figure}

\begin{table}[htb]
\caption{Parameters}
\label{tab:my-table1}
\begin{tabular}{c|cc}
\hline\hline
Parameters                                          & True value               & Learned by Algorithm~\ref{Alg:Tsallis-q-Learning}                            \\ \hline
$(\theta_1,\theta_2,\theta_3,\theta_4,\theta_5)$    & $(1.99,2.01,2,1,0.01)$   & $(1.9362, 2.1013,2.1604,1.1215,0.1008)$          \\
$(\zeta_1,\zeta_2,\zeta_3,\zeta_4,\zeta_5,\zeta_6)$ & $(1.99,2.01,2,1,0.01,1)$ & $(0.6185, 2.1372, 2.8776, 1.0380,0.1008,0.7107)$ \\ \hline\hline
\end{tabular}
\end{table}

Table~\ref{tab:my-table1} reports 
the true parameter values and the learned parameter values by Algorithm \ref{Alg:Tsallis-q-Learning}. Figure~\ref{fig:RL1}
 plots the convergence behavior of the dark pool trading problem by the offline learning algorithm within the framework of Tsallis entropy.  After sufficient iterations, these parameters converge to the true values. The convergence of both the model parameters and the value error underscores the effectiveness of the offline q-learning algorithm. 
 
\subsection{A non-LQ optimal repo rate control problem}

In this section, we consider a class of non-LQ stochastic control problems with jumps in which we can obtain the closed-form solution with the choice of the Tsallis entropy index $p=2$. More precisely, let $u=(u_t)_{t\in[0,T]}=(\xi_t,\eta_t)_{t\in[0,T]}\in{\cal U}$ be the corresponding control strategy taking values on $U=\R^2$.

Let us consider the controlled state process under $u=(\xi,\eta)\in{\cal U}$ that, for $s\in(t,T]$, 
\begin{align}\label{eq:state-exam2}    
\frac{dX_s^u}{X_{s-}^u}&=\xi_{s-} \mu_1 ds+\eta_{s-} \mu_2 ds+\sigma dW_s-\nu dN_s,\quad X_t^u=x>0,
\end{align}
where the parameters $\mu_1,\mu_2\in\R$, $\sigma>0$ and $\nu<1$. We use $\mu_1,\mu_2\geq 0$ to denote the rate charged by a hedger when he lends money to two kinds of repo market and implements his short-selling position (see \citealt{bichuch2018arbitrage}). Then, the dynamics \eqref{eq:state-exam2} describes the cash flow controlled by the lending strategy $(\xi,\eta)$. 
The value function is given by
\begin{align}\label{eq:Nonlinear}	
v(t,x)=\sup_{u=(\xi,\eta)\in{\cal U}}\Ex^{\Px}\left[U(X_T^u)-\int_t^T \left\{A \xi_s^2 (X_s^u)^{2h}+B \eta_s^2(X_s^u)^{2h}\right\}ds\right].
\end{align}
Here, $U(x)=\frac{x^h}{h}$ is the standard power utility for $x>0$, $0<h<1$, and $A,B>0$ are the cost parameters.

The exploratory formulation of the problem \eqref{eq:state-exam2}-\eqref{eq:Nonlinear} under Tsallis entropy is given by, for $(t,x)\in[t,T]\times\R_+$,
\begin{align}\label{eq:RLV}
V(t,x)&=\sup_{\pi\in\Pi}\Ex^{\Px}\left[\frac{(X_T^{\pi})^h}{h}+\int_t^T\int_{\R^2}\left\{-Au_1^2(X_s^{\pi})^{2h}-Bu_2^2(X_s^{\pi})^{2h}+\gamma l_p (\pi_s(u))\right\}\pi_s(u)duds\right],\nonumber\\
{\rm s.t.}~X_s^{\pi}&=x+\int_t^s\int_{\R^2} (u_1\mu_1+u_2\mu_2)X_v^{\pi}\pi_v(du)dv + \int_t^s\sigma X_v^{\pi} dW_v\\
&\quad-\int_0^t\nu X_{v-}^{\pi}dN_v,~\forall s\in[t,T].\nonumber
\end{align}
Then, the exploratory HJB equation satisfied by $V(t,x)$ is written by
\begin{align}\label{eq:exHJBnew-nonLQ}
0 &= V_t+\frac{\sigma^2 }{2}x^2 V_{xx}+\lambda (V(t,(1-\nu) x)-V(t,x))\nonumber\\
&\quad+\sup_{\pi_t\in\mathcal{P}(U)}\bigg\{xV_x\left(\mu_1\int_{U} u_1\pi(u|t,x)du+\mu_2\int_{\R^2} u_2\pi(u|t,x)du\right)\nonumber\\
&\quad-Ax^{2h}\int_{\R^2} u_1^2\pi(u|t,x)du-Bx^{2h}\int_{\R^2} u_2^2\pi(u|t,x)du+\gamma\int_{\R^2} l_p(\pi(u|t,x))\pi(u|t,x)du\bigg\}
\end{align}
with terminal condition $V(T,x)=\frac{x^h}{h}$ for all $x\in\R_+$. To enforce the constraints $\int_{U} \pi(u|t,x)du=1$ and $\pi(u|t,x)\geq 0$ for 
	$(t,x,u)\in[0,T]\times\R_+\times U$, we introduce the Lagrangian given by
	\begin{align*}
		{\cal L}(t,x,\pi,\psi,\zeta)&:=
		xV_x\left(\mu_1\int_{U} u_1\pi(u|t,x)du  +\mu_2\int_{U} u_2\pi(u|t,x)du \right)
		\\
		&\quad-Ax^{2h}\int_{U} u_1^2\pi(u|t,x)du-Bx^{2h}\int_{U} u_2^2\pi(u|t,x)du\\
		&\quad+\frac{\gamma}{p-1}\int_{U}(\pi(u|t,x)-\pi^{p}(u|t,x))du+\psi(t,x)\left(\int_{U} \pi(u|t,x)du-1\right)\notag\\
		&\quad 
		+\int_U \zeta(t,x,u)\pi(u|t,x)du,
	\end{align*}
	where $\psi(t,x)$ is a function of $(t,x)\in [0,T]\times \R_+$, and $\zeta(t,x,u)$ is a function of $(t,x,u)\in [0,T] \times\R_+\times U$.
	It follows from the first-order condition that
	\begin{align}\label{eq:hatpi1}
		\widehat{\pi}(u|t,x)=
			\left(\frac{p-1}{\gamma p }\right)^{\frac{1}{p-1}}\left(\psi(t,x)+
			\mu_1xV_x u_1-Ax^{2h}  u_1^2+\mu_2xV_x u_2-B x^{2h} u_2^2+\frac{\gamma}{p-1}   \right)_+
	\end{align}
	with the multiplier $\zeta(t,x,u)$ given by
	\begin{align*}
		\zeta(t,x,u)=\left(-\frac{\gamma}{p-1}-
			\mu_1xV_x u_1+A x^{2h} u_1^2-\mu_2xV_x u_2+Bx^{2h}  u_2^2-\psi(t,x)\right)_+,\quad p>1,
	\end{align*}

We next derive the closed-form solution to the exploratory HJB equation \eqref{eq:exHJBnew-nonLQ} for $p=2$. We guess that the exploratory HJB equation \eqref{eq:exHJBnew-nonLQ} has the solution in the form of
\begin{align}\label{eq:Vtxexam2}
V(t,x)=\alpha^*(t)\frac{x^h}{h}+\beta^*(t),\quad \forall (t,x)\in[0,T]\times\R_+.
\end{align}
Plugging this solution into \eqref{eq:hatpi1}, we obtain
	\begin{align*}
		\widehat{\pi}(u|t,x)
		&= \left(\frac{(p-1)\tilde \psi(t,x)}{\gamma p}\right)^{\frac{1}{p-1}}\left(1-
		\frac{Ax^{2h}}{\tilde\psi(t,x)} \left(u_1-Y_1(t,x)\right)^2-
		\frac{Bx^{2h}}{ \tilde \psi(t,x)}\left(u_2-Y_2(t,x)\right)^2\right)_+^{\frac{1}{p-1}}
	\end{align*}
with $\tilde \psi(t,x)= \psi(t,x)+\frac{\gamma}{p-1}+
Y_1^2(t,x)+Y_2^2(t,x)$, which is assumed to be greater than zero and will be verified later. Here, we define $Y_1(t,x):=\frac{\mu_1\alpha^*(t)}{2Ax^{h}}$ and $Y_2(t,x):=\frac{\mu_2\alpha^*(t)}{2Bx^{h}}$. Using the constraint $\int_{U} \pi(u|x)du=1$, we have
\begin{align*}
1&= \left(\frac{(p-1)\tilde \psi(t,x)}{\gamma p}\right)^{\frac{1}{p-1}}\int_{\R^2}\left(1-\frac{Ax^{2h}}{\tilde\psi(t,x)} \left(u_1-Y_1(t,x)\right)^2-\frac{Bx^{2h}}{\tilde \psi(t,x)}\left(u_2-Y_2(t,x)\right)^2\right)_+^{\frac{1}{p-1}}du\\
&=\left(\frac{(p-1)\tilde\psi(t,x)}{p}\right)^{\frac{p}{p-1}}\frac{\pi}{\gamma^{\frac{1}{p-1}}}\frac{1}{\sqrt{AB}x^{2h}}.
\end{align*}
This yields that, for all $(t,x)\in[0,T]\times\R_+$, $\tilde{\psi}(t,x)= \left(\frac{\sqrt{AB}x^{2h}}{\pi}\right)^{\frac{p-1}{p}}\frac{p}{p-1}\gamma^{\frac{1}{p}}$. As $p>1$, it follows that $\tilde{\psi}(t,x)$ is positive. In order to determine the coefficients $\alpha^*(t)$ and $\beta^*(t)$ in \eqref{eq:Vtxexam2}, we first compute the following moments of the  optimal policy that
\begin{align}
\int_{\R^2} u_1^k \widehat{\pi}(u|t,x)du=
\begin{cases}
\displaystyle \frac{\mu_1}{2A x^{h}}\alpha^*(t),\quad k=1,\\[1.2em]
\displaystyle\frac{(p-1) \tilde{\psi}(t,x) }{ 2A(2p-1 )
			x^{2h}}+\left(\frac{\mu_1}{2A x^{h}}\alpha^*(t)\right)^2,\quad k=2,
		\end{cases}
\end{align}
and 
\begin{align}\label{eq:moment2}
\int_{\R^2} u_2^k \widehat{\pi}(u|t,x)du=
\begin{cases}
\displaystyle \frac{\mu_2}{2Bx^h}\alpha^*(t), \quad k=1,\\[1.2em]
\displaystyle \frac{(p-1) \tilde\psi(t,x) }{2B (2p-1 )
				 x^{2h}}+\left(\frac{\mu_2}{2Bx^{h}}\alpha^*(t)\right)^2,\quad k=2.
		\end{cases}
\end{align}
Moreover, it holds that
\begin{align}\label{eq:eqn2exam2}
\int_{\R^2}\frac{1}{p-1}\left(\widehat{\pi}(u|t,x)-\widehat{\pi}(u|t,x)^{p}\right)du=\frac{1}{p-1}-\frac{\tilde{\psi}(t,x)}{\gamma (2p-1)}.
\end{align}
Substituting the above terms into \eqref{eq:exHJBnew-nonLQ}, we derive
\begin{align}\label{eq:eqnalphabetap=2}
		0&=\frac{d\alpha^*(t)}{dt}\frac{x^h}{h}+\frac{d\beta(t)}{dt}+
		\frac{\sigma^2}{2}(h-1)\alpha^*(t) x^h+
		\frac{\mu_1^2}{4A}(\alpha^*(t))^2+\frac{\mu_2^2}{4B}(\alpha^*(t))^2 +\lambda\frac{(1-\nu)^h-1}{h}\alpha^*(t)x^h\nonumber\\
		&\quad-\frac{p}{2p-1}\left(\frac{\sqrt{AB}x^{2h}}{\pi}\right)^{\frac{p-1}{p}}\frac{p}{p-1}\gamma^{\frac{1}{p}}+\frac{\gamma}{p-1}.
	\end{align}
Then, we have the following explicit solution for the exploratory problem~\eqref{eq:RLV}.
	\begin{proposition}\label{prop:nonLQ-entropytsallis}
	Under the Tsallis entropy regularization with $p=2$, the RL problem~\eqref{eq:RLV} has the following explicit value function that
	\begin{align}\label{eq:value2}
		V(t,x)=\frac{\alpha^*(t)}{h}x^h+\beta^*(t),\quad \forall (t,x)\in[0,T]\times\R_+,
	\end{align}
	where the coefficients $\alpha^*(t)$ and $\beta^*(t)$ for $t\in[0,T]$ are given by
\begin{align*}
\alpha^*(t)=&\left(1-\frac{\frac{4h}{3}\sqrt{\frac{\gamma}{\pi}}{(AB)}^{\frac{1}{4}}}{ \frac{\sigma^2}{2}(h-1)h+\lambda((1-\nu)^h-1) }\right)e^{\left(\frac{\sigma^2}{2}(h-1)h+\lambda((1-\nu)^h-1)\right)(T-t)}\\
&+\frac{\frac{4h}{3}\sqrt{\frac{\gamma}{\pi}}{(AB)}^{\frac{1}{4}}}{ \frac{\sigma^2}{2}(h-1)h+\lambda((1-\nu)^h-1) },
\end{align*}

\begin{align*}
\beta^*(t)&=\left(\frac{\mu_1^2}{4A}+\frac{\mu_2^2}{4B}\right)\frac{\left(1-\frac{\frac{4h}{3}\sqrt{\frac{\gamma}{\pi}}{(AB)}^{\frac{1}{4}}}{ \frac{\sigma^2}{2}(h-1)h+\lambda((1-\nu)^h-1) }\right)^2}{2\left(\frac{\sigma^2}{2}(h-1)h+\lambda((1-\nu)^h-1)\right)}\left(e^{2\left(\frac{\sigma^2}{2}(h-1)h+\lambda((1-\nu)^h-1)\right)(T-t)}-1\right)\\
&+\left(\frac{\mu_1^2}{4A}+\frac{\mu_2^2}{4B}\right)\frac{2\left(1-\frac{\frac{4h}{3}\sqrt{\frac{\gamma}{\pi}}{(AB)}^{\frac{1}{4}}}{ \frac{\sigma^2}{2}(h-1)h+\lambda((1-\nu)^h-1) } \right){\frac{4h}{3}\sqrt{\frac{\gamma}{\pi}}{(AB)}^{\frac{1}{4}}}}{\left( \frac{\sigma^2}{2}(h-1)h+\lambda((1-\nu)^h-1) \right)^2}\left(e^{\left(\frac{\sigma^2}{2}(h-1)h+\lambda((1-\nu)^h-1)\right)(T-t)}-1\right)
\\
&+\left(\left(\frac{\mu_1^2}{4A}+\frac{\mu_2^2}{4B}\right)\left(\frac{\frac{4h}{3}\sqrt{\frac{\gamma}{\pi}}{(AB)}^{\frac{1}{4}}}{ \frac{\sigma^2}{2}(h-1)h+\lambda((1-\nu)^h-1) }\right)^2+\gamma\right) (T-t).
	\end{align*}
The optimal policy is given by, for $(t,x)\in[0,T]\times\R_+$ and $u=(u_1,u_2)\in U=\R^2$,
	\begin{align}\label{eq:optimal-policy-lend}
		\widehat{\pi}(u|t,x)
		&=\frac{1}{2\gamma}\left\{2(AB)^{\frac{1}{4}}\sqrt{\frac{\gamma}{\pi}}x^h-
			Ax^{2h}\left(u_1-\frac{\mu_1 \alpha^*(t)}{2A x^h}\right)^2-Bx^{2h}
			\left(u_2-\frac{\mu_2 \alpha^*(t)}{2B x^h}\right)^2\right\}_+.
	\end{align}
	\end{proposition}
 
\begin{proof}
For $p=2$, Eq.~\eqref{eq:eqnalphabetap=2} yields that
		\begin{align*}
		0&=(\alpha^*(t))'\frac{x^h}{h}+(\beta^*(t))'+
		\frac{\sigma^2}{2}(h-1)\alpha^*(t) x^h+\lambda\frac{(1-\nu)^h-1}{h}\alpha^*(t)x^h\\
  &\quad+
		\frac{\mu_1^2}{4A}(\alpha^*(t))^2+\frac{\mu_2^2}{4B}(\alpha^*(t))^2 -\frac{4}{3}\left(\frac{{(AB)}^{\frac{1}{4}}x^{h}}{\sqrt{\pi}}\right)
	 \sqrt{\gamma}+\gamma.
	\end{align*}
Then, it holds that
	\begin{align*}
 \begin{cases}
     \displaystyle \frac{(\alpha^*(t))'}{h}+
		\left(\frac{\sigma^2}{2}(h-1)+\lambda\frac{(1-\nu)^h-1}{h}\right)\alpha^*(t)-\frac{4}{3}\sqrt{\frac{\gamma}{\pi}}{(AB)}^{\frac{1}{4}}
		=0, ~\alpha^*(T)=1,\\[0.8em]
  \displaystyle (\beta^*(t))'+
		\left(\frac{\mu_1^2}{4A}+\frac{\mu_2^2}{4B}\right)(\alpha^*(t))^2+\gamma=0,~\beta^*(T)=0.
 \end{cases}
\end{align*}
Furthermore, we can solve the above ODEs explicitly to obtain the desired result. By some standard verification arguments, the function~\eqref{eq:value2} is the optimal value function of the exploratory problem~\eqref{eq:RLV}.
\end{proof}

\begin{remark}\label{example-two-distribution}
    Notably, the explicit results in Proposition~\ref{prop:nonLQ-entropytsallis} are exclusive to the  Tsallis entropy when $p=2$, while no exact parameterization is available under the conventional Shannon entropy in this example. The optimal policy given by \eqref{eq:optimal-policy-lend} is also a two-dimensional $p$-Gaussian distribution with a compact support set (see Figure \ref{fig:policy-lend}).
    The mean and variance of the policy $(u_1,u_2)$ are given by 
  \begin{align}
&{\rm mean} (u_1,u_2)=\left(\int_{\R^2} u_1 \widehat{\pi}(u)du,\int_{\R^2} u_2 \widehat{\pi}(u)du\right)=
 \frac{\alpha^*(t)}{2 x^{h}}\left(\frac{\mu_1}{A},\frac{\mu_2}{B}\right),\\
&{\rm Var} (u_1,u_2)=\left(\int_{\R^2} u_1^2 \widehat{\pi}(u)du-\left(\int_{\R^2} u_1 \widehat{\pi}(u)du\right)^2,\int_{\R^2} u_2^2 \widehat{\pi}(u)du-\left(\int_{\R^2} u_2 \widehat{\pi}(u)du\right)^2\right)\notag\\
&\qquad \qquad =
 \frac{(p-1) \tilde{\psi}(t,x) }{ 2 (2p-1 )
			x^{2h}} \left(\frac{1}{A},\frac{1}{B}\right).
\end{align}
    In fact, for $(t,x)\in[0,T]\times (0,+\infty)$, we have
    \begin{align*}
        &u_1\in\left[\frac{\mu_1 \alpha^*(t)}{2A x^h}-\sqrt{\frac{2B^{\frac{1}{4}}}{A^{\frac{3}{4}}x^h}\sqrt{\frac{\gamma}{\pi}}},\frac{\mu_1 \alpha^*(t)}{2A x^h}+\sqrt{\frac{2B^{\frac{1}{4}}}{A^{\frac{3}{4}}x^h}\sqrt{\frac{\gamma}{\pi}}}\right],\\
        &u_2\in\left[\frac{\mu_2 \alpha^*(t)}{2B x^h}-\sqrt{\frac{2A^{\frac{1}{4}}}{B^{\frac{3}{4}}x^h}\sqrt{\frac{\gamma}{\pi}}},\frac{\mu_2 \alpha^*(t)}{2B x^h}+\sqrt{\frac{2A^{\frac{1}{4}}}{B^{\frac{3}{4}}x^h}\sqrt{\frac{\gamma}{\pi}}}\right].
    \end{align*}
    
 \begin{figure}[h]
\centering
\includegraphics[width=0.6\textwidth]{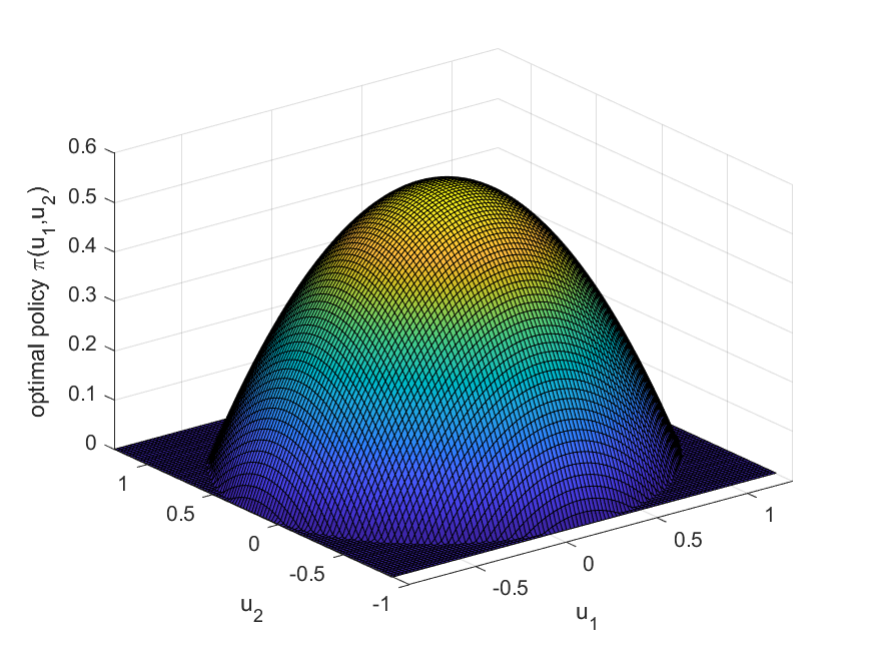}
 \caption{\small The optimal policy $(u_1,u_2)\to \widehat{\pi}(u_1,u_2)$. The model parameters are set to be  $\lambda=1,~\sigma=1,~\nu=0.5,~A=B=1,~\mu_1=\mu_2=0.5,~h=1.5,~\gamma=1,~t=1,~T=2,~x=1$.}\label{fig:policy-lend}
\end{figure}

Moreover, when the temperature parameter $\gamma$ goes to $0$, we have $\tilde{\psi}(t,x) (p-1)\to 0$. Then, it holds that 
\begin{align*}
    \int_{\R^2} (u_1^2+u_2^2) \pi(u|t,x)du&=
    \frac{(p-1) \tilde{\psi}(t,x) }{ 2A(2p-1 )
			x^{2h}}+\left(\frac{\mu_1}{2A x^{h}}\alpha^*(t)\right)^2+\frac{(p-1) \tilde\psi(t,x) }{2B (2p-1 )
				 x^{2h}}+\left(\frac{\mu_2}{2Bx^{h}}\alpha^*(t)\right)^2 \notag\\
     &\xrightarrow[]{\gamma\to 0} \left(\frac{\mu_1}{2A x^{h}}\alpha^*(t)\right)^2+\left(\frac{\mu_2}{2Bx^{h}}\alpha^*(t)\right)^2.
   \end{align*}
This implies the convergence of the borrowing and lending policy to a  deterministic strategy that $(\xi_t,\eta_t)\xrightarrow[\gamma\to 0]{L^2} \left(\frac{\mu_1}{2A x^{h}}\alpha^*(t),\frac{\mu_2}{2Bx^{h}}\alpha^*(t)\right)$ for $(t,x)\in [0,T]\times\R_+$, which is the optimal strategy of the classical stochastic control problem.
\end{remark}

Let us consider the true parameters defined by
\begin{align}
\begin{cases}
\displaystyle \theta_1^*=\frac{\sigma^2}{2}(h-1)h+\lambda((1-\nu)^h-1),\\[0.8em]
\displaystyle \theta_2^*=\left(\frac{\mu_1^2}{4A}+\frac{\mu_2^2}{4B}\right)\frac{1}{\frac{\sigma^2}{2}(h-1)h+\lambda((1-\nu)^h-1)},\\[0.8em]
\displaystyle \theta_3^*=\frac{4h}{3}\sqrt{\frac{\gamma}{\pi}}(AB)^{\frac{1}{4}}\frac{1}{\frac{\sigma^2}{2}(h-1)h+\lambda((1-\nu)^h-1)},
\end{cases}
\end{align}
and
\begin{align*}
&\zeta_1^*=\frac{\sigma^2}{2}(h-1)h+\lambda((1-\nu)^h-1),\quad
\zeta_2^*=A,\quad
\zeta_3^*=B,\quad
\\
&\zeta_4^*=\frac{\mu_1}{2A},\quad
\zeta_5^*=\frac{\mu_2}{2B},\quad
\zeta_6^*=\frac{4h}{3}\sqrt{\frac{\gamma}{\pi}}(AB)^{\frac{1}{4}}\frac{1}{\frac{\sigma^2}{2}(h-1)h+\lambda((1-\nu)^h-1)}.
\end{align*}
By Theorem~\ref{th:LQ-entropytsallis}, we can parameterize the optimal value function and the optimal q-function in the exact form by 
 \begin{align}
     J^{\theta}(t,x)&=\frac{(1-\theta_3)e^{\theta_1(T-t)}+\theta_3}{h}x^h+\frac{1}{2}\theta_2(1-\theta_3)^2(e^{2\theta_1(T-t)}-1)+2\theta_2\theta_3(1-\theta_3)(e^{2\theta_1(T-t)}-1)\nonumber\\
&\quad+(\theta_1\theta_2\theta_3^2+\gamma)(T-t),\label{eq:parameter-J}\\
     q^{\zeta}(t,x,\xi,\eta)&=-\zeta_2x^{2h}\left(u_1-\frac{\zeta_4((1-\zeta_6)e^{\zeta_1(T-t)}+\zeta_6)}{x^h}\right)^2\nonumber\\
     &\quad-\zeta_3x^{2h}\left(u_2-\frac{\zeta_5((1-\zeta_6)e^{\zeta_1(T-t)}+\zeta_6)}{x^h}\right)^2+\frac{4}{3}\sqrt{\frac{\gamma}{\pi}}(\zeta_2\zeta_3)^{\frac{1}{4}}-\gamma,\label{eq:parameter-q}
 \end{align}
where $(\theta_1,\theta_2,\theta_3)\in\R^3$ and $(\zeta_1,\zeta_2,\zeta_3,\zeta_4,\zeta_5,\zeta_6)\in\R^6$ are unknown parameters to be learned. As a consequence, the parameterized policy $\pi^{\zeta}$ is given by
\begin{align}
\pi^{\zeta}(u_1,u_2|t,x)&=\frac{1}{2\gamma}\left(2(\zeta_2\zeta_3)^{\frac{1}{4}}\sqrt{\frac{\gamma}{\pi}}x^h-\zeta_2x^{2h}\left(u_1-\frac{\zeta_4((1-\zeta_6)e^{\zeta_1(T-t)}+\zeta_6)}{x^h}\right)^2\right.\nonumber\\
&\qquad\qquad\quad\left.-\zeta_3x^{2h}\left(u_2-\frac{\zeta_5((1-\zeta_6)e^{\zeta_1(T-t)}+\zeta_6)}{x^h}\right)^2\right)_+.
\end{align}

\begin{figure}[htb]
    \centering
    \includegraphics[width=0.9\textwidth]{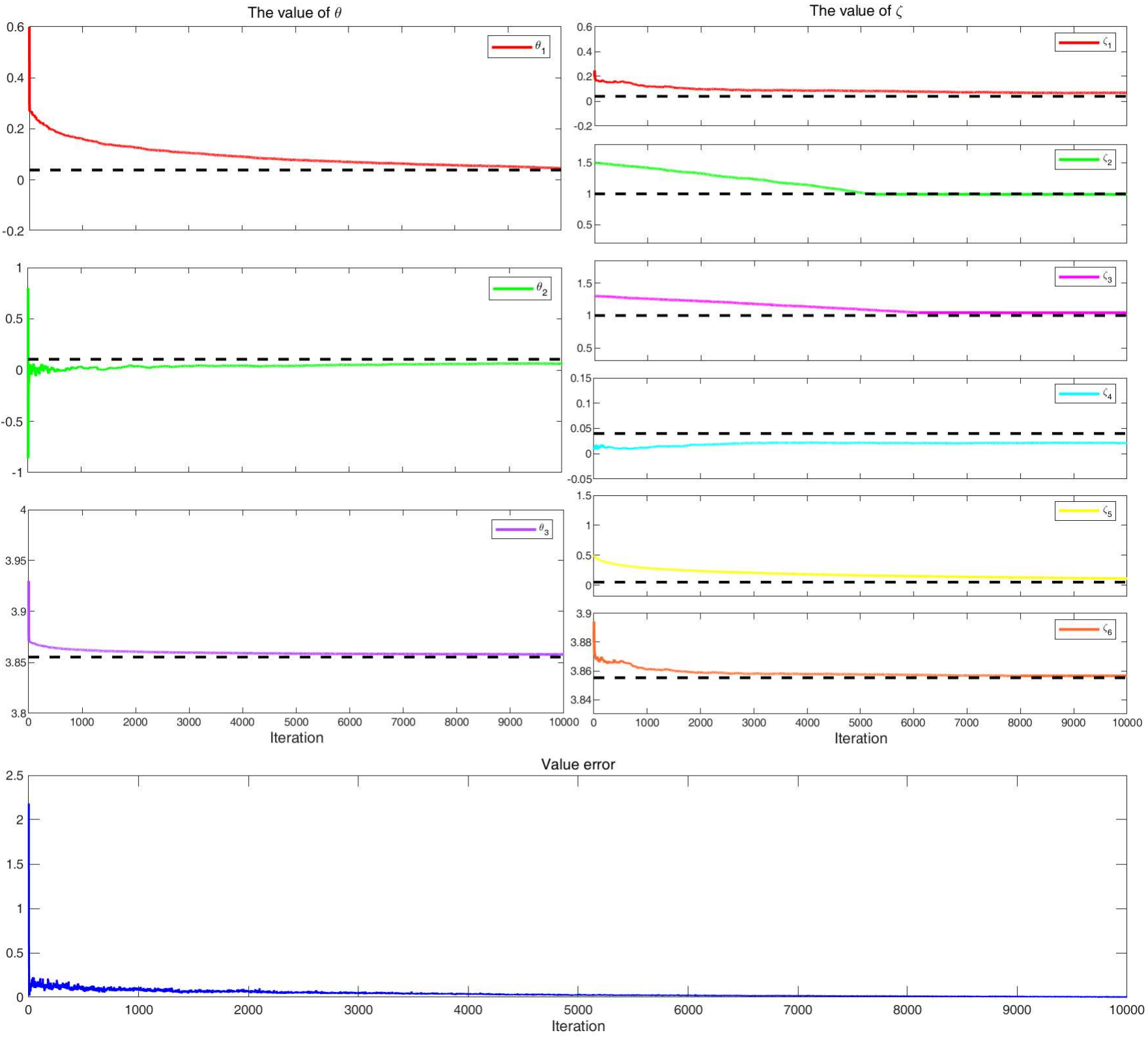}
    \caption{Convergence of Algorithm \ref{Alg:Tsallis-q-Learning} using a market simulator. The upper panels show the convergence of parameter iterations 
    for $(\theta_1,\theta_2,\theta_3,\zeta_1,\zeta_2,\zeta_3,\zeta_4,\zeta_5,\zeta_6)$;    
    the bottom panel shows the value error along the iterations. }
    \label{fig:RL2}
\end{figure}

\textbf{Simulator}: With the above parameterized value function and q-function, we next apply the  Algorithm \ref{Alg:Tsallis-q-Learning}.   We  use the acceptance-rejection sampling method to generate the control pair $(u^1_t,u^2_t)$  from the $p$-Gaussian distribution with density function given by \eqref{eq:parameter-policy} at time $t$. Then the control pair  $(u^1_t,u^2_t)$ is used to generate sample trajectories through the following simulator
 \begin{align*}
X_{t+\Delta t}-X_t=(\mu_1 u^1_t+\mu_2 u_2)X_t\Delta t +\sigma X_t W(\Delta t)- \nu X_t N(\Delta t),
\end{align*}
where $N(\Delta t)$ is a Poisson random variable with rate $\lambda \Delta t$, $W(\Delta)$ is a random variable obeying normal distribution ${\cal N}(0,\Delta t)$.

\textbf{Algorithm Inputs}: We set the coefficients of the simulator to $\lambda=0.01,\mu_1=0.08,\mu_2=0.1,\sigma=0.2,\nu=0.05, X_0=2,T=0.5$, the known parameters as $\gamma=0.01,p=2,A=1,B=1,h=2,x=2,T=0.5$, the time step as $dt=0.01$, and the number of iterations as $N=10000$. The learning rates are set as follows:

{\small\begin{align*}
    \alpha_{\theta_1}(k)&=
        \frac{0.0023}{{\rm linspace}_{(1,90,N)}(k)},~~\text{for } 1\leq k\leq N,~~~~~~~~~~~~~~~~
     \alpha_{\theta_2}(k)= \frac{0.0325}{{\rm linspace}_{(1,90,N)}(k)},~~\text{for } 1\leq k\leq N,\\
     \alpha_{\theta_3}(k)&= \frac{0.0017}{{\rm linspace}_{(1,60,N)}(k)},~~\text{for } 1\leq k\leq N,~~~~~~~~~~~~~~~
     \alpha_{\zeta_1}(k)=\frac{0.0026}{{\rm linspace}_{(1,50,N)}(k)}, ~\text{for } 1\leq  k\leq N,\\
\alpha_{\zeta_2}(k)&=\begin{cases}
     \displaystyle   0.005,~\text{if } 1\leq k\leq 5200,\\
     \displaystyle   \frac{0.01}{{\rm linspace}_{(1,500,N)}(k)},~\text{if } 5200< k\leq N,
        \end{cases}~~~
     \alpha_{\zeta_3}(k)=\begin{cases}
      \displaystyle  0.002,~\text{if } 1\leq k \leq 6100,\\
      \displaystyle   \frac{0.005}{{\rm linspace}_{(1,500,N)}(k)},~\text{if }  6100 < k \leq N,
    \end{cases}\\
     \alpha_{\zeta_4}(k)&=\frac{0.0046}{{\rm linspace}_{(1,150,N)}(k)},~~\text{for }1\leq k\leq N,~~~~~~~~~~~~~~~~
     \alpha_{\zeta_5}(k)=\frac{0.0045}{{\rm linspace}_{(1,150,N)}(k)},~~\text{for }1\leq k\leq N,\\
    \alpha_{\zeta_6}(k)&=\begin{cases}
     \displaystyle   \frac{0.015}{{\rm linspace}_{(1,80,N)}(k)},~\text{if } 1\leq k\leq 8000,\\
     \displaystyle    0.00001,~\text{if }  8000<k\leq N,
    \end{cases}
\end{align*}}

\begin{table}[htb]
\centering
\caption{Parameters}
\label{tab:my-table2}
\begin{tabular}{c|cc}
\hline\hline
Parameters                                  & True value               & Learned by Algorithm \ref{Alg:Tsallis-q-Learning}                    \\ \hline
$(\theta_1,\theta_2,\theta_3)$              & $(0.039,0.105,3.855)$ & $(0.039,0.065,3.857)$                 \\
$(\zeta_1,\zeta_2,\zeta_3,\zeta_4,\zeta_5,\zeta_6)$ & $(0.039,1,1,0.040,0.050,3.855)$ & $(0.056, 1, 1.042, 0.022,0.074,3.855)$ \\ \hline\hline
\end{tabular}
\end{table}

We then track the parameters $(\theta_1,\theta_2,\theta_3)$, $(\zeta_1,\zeta_2,\zeta_3,\zeta_4,\zeta_5,\zeta_6)$, and the error of the value function throughout the iterative process.  Table~\ref{tab:my-table2} reports 
the true parameter values and the learned parameter values by Algorithm \ref{Alg:Tsallis-q-Learning}. 
After sufficient iterations, it can be seen from Figure~\ref{fig:RL2} that
the iterations of parameters exhibit good convergence, and the value error converges to zero, illustrating the satisfactory performance of the q-learning algorithm under the Tsallis entropy.

In what follows, we are also interested in illustrating the effectiveness of our Actor-Critic q-learning algorithm \ref{Alg:Tsallis-q-Learning-normalizing-unavailable} when the normalizing function is not available. Although the true normalizing function can be derived explicitly in this example, we can still choose different parameters for the optimal q-function and the optimal policy and see whether the Actor-Critic iterations will converge. Meanwhile, we
still take advantage of the explicit expression of the optimal q-function and the distribution of the optimal policy in this example. Precisely, we parameterize the optimal  value function and the optimal q-function the same as \eqref{eq:parameter-J} and \eqref{eq:parameter-q} but choose different parameters $(\chi_1,\chi_2,\chi_3,\chi_4,\chi_5,\chi_6)$ to approximate the optimal policy as following: for $(t,x,u_1,u_2)\in[0,T]\times \R_+\times \R^2$,
\begin{align*}
\pi^{\chi}(u_1,u_2|t,x)
&=\frac{1}{2\gamma}\left(2(\chi_2\chi_3)^{\frac{1}{4}}\sqrt{\frac{\gamma}{\pi}}x^h-\chi_2x^{2h}\left(u_1-\frac{\chi_4((1-\chi_6)e^{\chi_1(T-t)}+\chi_6)}{x^h}\right)^2\right.\\
&\qquad\qquad\quad\left.-\chi_3x^{2h}\left(u_2-\frac{\chi_5((1-\chi_6)e^{\chi_1(T-t)}+\chi_6)}{x^h}\right)^2\right)_+,\nonumber
\end{align*}
with true values given by
\begin{align*}
&\chi_1^*=\frac{\sigma^2}{2}(h-1)h+\lambda((1-\nu)^h-1),\quad
\chi_2^*=A,\quad
\chi_3^*=B,\quad
\\
&\chi_4^*=\frac{\mu_1}{2A},\quad
\chi_5^*=\frac{\mu_2}{2B},\quad
\chi_6^*=\frac{4h}{3}\sqrt{\frac{\gamma}{\pi}}(AB)^{\frac{1}{4}}\frac{1}{\frac{\sigma^2}{2}(h-1)h+\lambda((1-\nu)^h-1)}.
\end{align*}
We use the same simulator and algorithm inputs as in the previous case when the normalizing function is available, except the learning rates that are now chosen by
{\small\begin{align*}
     \alpha_{\chi_1}(k)&=\frac{0.026}{{\rm linspace_{(1,100,N)}}(k)},~\text{for } 1\leq  k\leq N,~
\alpha_{\chi_2}(k)=\frac{0.05}{{\rm linspace_{(1,500,N)}}(k)},~\text{for } 1\leq  k\leq N,\\
     \alpha_{\chi_3}(k)&=\begin{cases}
      \displaystyle  0.002,~\text{if } 1\leq k\leq 6100,\\
      \displaystyle   \frac{0.005}{{\rm linspace}_{(1,500,N)}(k)},~\text{if }  6100\leq k\leq N,
    \end{cases}~
     \alpha_{\chi_4}(k)=
        \frac{0.00461}{{\rm linspace_{(1,150,N)}}(k)},~\text{for } 1\leq  k\leq N,
    \\
     \alpha_{\chi_5}(k)&=\frac{0.005}{{\rm linspace_{(1,200,N)}}(k)},~\text{for } 1\leq  k\leq N,~~~~~~~~~~~~\alpha_{\chi_6}(k)=\begin{cases}
      \displaystyle  \frac{0.0015}{{\rm linspace}_{(1,80,N)}(k)},~\text{if } 1\leq k\leq 8000,\\
      \displaystyle   0.00001,~\text{if }  80001\leq k\leq N.
    \end{cases}
\end{align*}}

The true parameters and the learned parameters by Algorithm \ref{Alg:Tsallis-q-Learning-normalizing-unavailable} are reported in Table~\ref{tab:my-table3}. We can see from both Figure~\ref{fig:RL3} and Table~\ref{tab:my-table3} that the iterations of parameters exhibit good convergence, and the value error converges to zero after sufficient iterations.

\begin{table}[htb]
\centering
\caption{Parameters used in the simulator.}
\label{tab:my-table3}
\begin{tabular}{c|cc}
\hline\hline
Parameters                                  & True value               & Learned by Algorithm \ref{Alg:Tsallis-q-Learning-normalizing-unavailable}                    \\ \hline
$(\theta_1,\theta_2,\theta_3)$              & $(0.039,0.105,3.855)$ & $(0.034,0.142,3.856)$                 \\
$(\zeta_1,\zeta_2,\zeta_3,\zeta_4,\zeta_5,\zeta_6)$ & $(0.039,1,1,0.040,0.050,3.855)$ & $(0.062, 1, 1.043, 0.039,0.0916,3.856)$ \\
$(\chi_1,\chi_2,\chi_3,\chi_4,\chi_5,\chi_6)$ & $(0.039,1,1,0.04,0.05,3.855)$ & $(0.070, 1.304, 1.301, 0.044,0.0101,3.859)$ \\
\hline\hline
\end{tabular}
\end{table}

\begin{figure}[htb]
    \centering
    \includegraphics[width=0.9\textwidth]{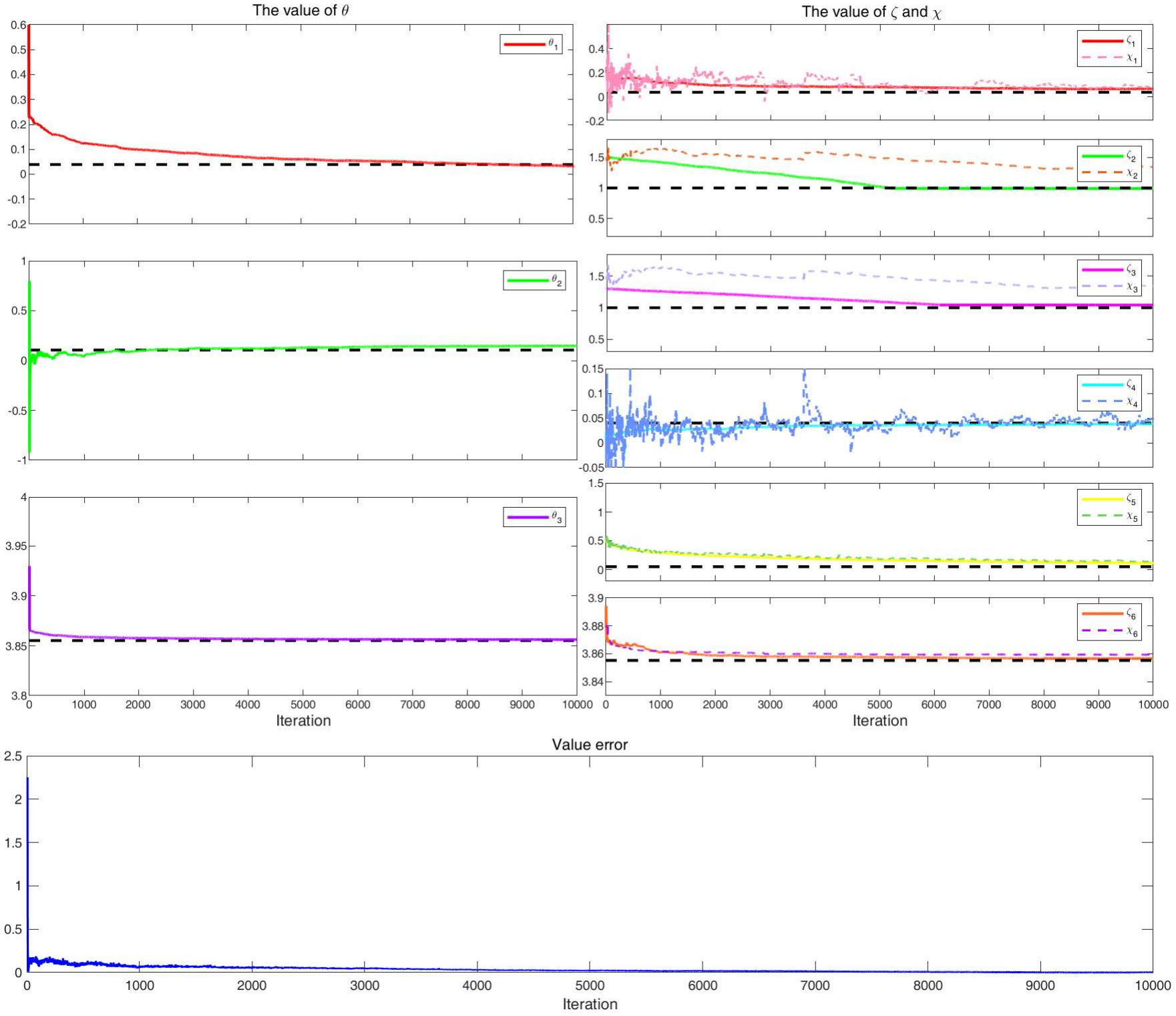}
    \caption{Convergence of Algorithm  \ref{Alg:Tsallis-q-Learning-normalizing-unavailable} using a market simulator. The upper panels show the convergence of parameter iterations 
    for $(\theta_1,\theta_2,\theta_3,\zeta_1,\zeta_2,\zeta_3,\zeta_4,\zeta_5,\zeta_6,\chi_1,\chi_2,\chi_3,\chi_4,\chi_5,\chi_6)$;    
    the bottom panel shows the value error along the iterations. }
    \label{fig:RL3}
\end{figure}

\vspace{1cm}

\section{Conclusions}\label{sec:conclusion}

This paper studies the continuous-time q-learning in jump-diffusion models under Tsallis entropy regularization. By considering the general entropy beyond the Shannon type, we obtain optimal policies that are not necessarily Gibbs measures but with compact support. We establish the martingale characterization of the q-function and devise several q-learning algorithms. The effectiveness of our algorithms is demonstrated in two financial applications where the parameter iterations exhibit desirable convergence. One appealing future study is to investigate the continuous-time RL algorithms under Tsallis entropy for mean field game and mean field control problems. Moreover, as the sampling for the resulting optimal policy with compact support might be challenging in high-dimensional scenarios, another interesting future research direction is to examine some tailor-made sampling methods, such as \citealt{neal2003slice,tong2020mala},
in our RL algorithms for some high-dimensional financial applications.

\ \\
\noindent\textbf{Acknowledgement.}\quad 
L. Bo and Y. Huang are supported by National Natural Science of Foundation of China (No. 12471451), Natural Science Basic Research Program of Shaanxi (No. 2023-JC-JQ-05), Shaanxi Fundamental Science Research Project for Mathematics and Physics (No. 23JSZ010) and Fundamental Research Funds for the Central Universities (No. 20199235177). X. Yu is supported by the Hong Kong RGC General Research Fund (GRF) under grant no. 15211524 and the Hong Kong Polytechnic University research grant under no. P0045654. T. Zhang is supported by the National Natural Science of Foundation of China (No.12401619) and the Natural Science Foundation of the Jiangsu Higher Education Institutions of China (No. 24KJB110023).

\bibliographystyle{informs2014}
\bibliography{referencelist}

\end{document}